\documentclass[reqno,11pt]{amsart}

\usepackage{amssymb}
\usepackage[margin=1in,letterpaper]{geometry}
\usepackage{layout}
\usepackage{tikz-cd}
\usepackage{amsmath,amsthm,mathrsfs} 
\usepackage{colonequals}
\usepackage{enumitem}
\usepackage{mathdots}
\usepackage{framed}
\usepackage{url}
\usepackage{multicol}
\usepackage{color}
\usepackage{bbm}
\usepackage[unicode,colorlinks]{hyperref}
\usepackage{microtype}
\usepackage{marginnote}
\usepackage{mathtools}

     \definecolor{dark-red}{rgb}{0.54,0,0}
     \definecolor{dark-green}{rgb}{0,0.54,0}
     \definecolor{dark-magenta}{rgb}{0.54,0,0.54}
     \definecolor{dark-cyan}{rgb}{0,0.54,0.54}
     \hypersetup{%
         unicode=true,
         breaklinks=true,
         pdfcreator=,
         pdfproducer=,
         pdfstartview =, 
         pdfpagelayout = ,
         pdfhighlight = /O,
         linkcolor=dark-red, 
         linkbordercolor=dark-red, 
         anchorcolor=red, 
         citecolor=dark-green, 
         citebordercolor=dark-green, 
         filecolor=dark-cyan, 
         filebordercolor=dark-cyan, 
         menucolor=dark-red, 
         menubordercolor=dark-red, 
         runcolor=dark-cyan, 
         runbordercolor=dark-cyan, 
         urlcolor=dark-cyan, 
         urlbordercolor=dark-cyan, 
}

\newcommand{\Affine}{\mathbb{A}}

\newcommand\FF{\protect\mathbb{F}}

\newcommand\QQ{\protect\mathbb{Q}}

\newcommand\ZZ{\protect\mathbb{Z}}

\newcommand\GG{\protect\mathbb{G}}

\newcommand\bA{\mathbb{A}}
\newcommand\bB{\mathbb{B}}

\newcommand\bG{\mathbb{G}}

\newcommand\bN{\mathbb{N}}

\newcommand\bQ{\mathbb{Q}}

\newcommand\bT{\mathbb{T}}
\newcommand\bU{\mathbb{U}}
\newcommand\bV{\mathbb{V}}
\newcommand\bW{\mathbb{W}}

\newcommand\bZ{\mathbb{Z}}

\newcommand\sT{\mathscr{T}}

\newcommand\cA{\mathcal{A}}

\newcommand\cF{\mathcal{F}}

\newcommand\cI{\mathcal{I}}
\newcommand\cJ{\mathcal{J}}

\newcommand\cO{\mathcal{O}}

\newcommand\cR{\mathcal{R}}

\newcommand\cT{\mathcal{T}}

\newcommand\cV{\mathcal{V}}
\newcommand\cW{\mathcal{W}}

\newcommand\fX{\mathfrak{X}}

\newcommand\fb{\mathfrak{b}}

\DeclareMathOperator\GL{GL}
\DeclareMathOperator\Hom{Hom}
\DeclareMathOperator\Tr{Tr}

\DeclareMathOperator\Gal{Gal}

\DeclareMathOperator\Char{char}

\DeclareMathOperator\sgn{sgn}

\DeclareMathOperator\Res{Res}

\DeclareMathOperator\Ind{Ind}

\DeclareMathOperator{\diag}{diag}

\newcommand{\from}{\colon}

\theoremstyle{theorem} \newtheorem{proposition}{Proposition}[section]
\theoremstyle{definition} \newtheorem{definition}[proposition]{Definition}
\theoremstyle{theorem} \newtheorem{lemma}[proposition]{Lemma}
\theoremstyle{remark} \newtheorem{remark}[proposition]{Remark}
\theoremstyle{remark} 
\theoremstyle{remark} 
\theoremstyle{definition} \newtheorem{example}[proposition]{Example}
\theoremstyle{definition} \newtheorem{notation}[proposition]{Notation}
\theoremstyle{theorem} 
\theoremstyle{theorem} 
\theoremstyle{theorem} \newtheorem{theorem}[proposition]{Theorem}
\theoremstyle{theorem} \newtheorem{corollary}[proposition]{Corollary}
\theoremstyle{definition} 
\theoremstyle{theorem} 
\theoremstyle{remark} 
\theoremstyle{definition} 
\theoremstyle{definition} 
\theoremstyle{definition} 
\theoremstyle{definition} 
\theoremstyle{remark} \newtheorem*{claim*}{Claim}
\theoremstyle{remark} 
\theoremstyle{theorem} 
\theoremstyle{theorem} \newtheorem{conjecture}[proposition]{Conjecture}
\theoremstyle{theorem} 
\theoremstyle{definition} 
\theoremstyle{definition} 
\theoremstyle{theorem} 
\theoremstyle{remark} 
\theoremstyle{definition} 
\theoremstyle{remark}
\theoremstyle{theorem} \newtheorem{main}{Main Theorem}

\newcommand\ur{\text{nr}}

\newcommand\Khat{\widehat{K}^{\ur}}

\DeclareMathOperator\Fr{Fr}

\DeclareMathOperator\pr{pr}

\newcommand\Loc{\mathcal{L}}

\newcommand\F{\FF_{q^n}}
\newcommand\Unip{\bU_{h,k}}
\newcommand\GUnip{\bG_{h,k}}
\newcommand\TUnip{\bT_{h,k}}

\DeclareMathOperator\Nm{Nm}

\DeclareMathOperator\Reg{Reg}

\DeclareMathOperator\wt{wt}

 {\endMakeFramed\noindent\ignorespacesafterend}
 
 \numberwithin{equation}{section}
 
\DeclareRobustCommand{\SkipTocEntry}[5]{}

\title[The cohomology of semi-infinite Deligne--Lusztig varieties]
{The cohomology of \\ semi-infinite Deligne--Lusztig varieties}
\author{Charlotte Chan}
\subjclass[2010]{Primary: 11G25, 20G25, 14F20. Secondary: 14G10, 11S37.}

\begin{document}

\maketitle

\begin{abstract}
We prove a 1979 conjecture of Lusztig on the cohomology of semi-infinite Deligne--Lusztig varieties attached to division algebras over local fields. We also prove the two conjectures of Boyarchenko on these varieties. It is known that in this setting, the semi-infinite Deligne--Lusztig varieties are ind-schemes comprised of limits of certain finite-type schemes $X_h$. Boyarchenko's two conjectures are on the maximality of $X_h$ and on the behavior of the torus-eigenspaces of their cohomology. Both of these conjectures were known in full generality only for division algebras with Hasse invariant $1/n$ in the case $h = 2$ (the ``lowest level'') by the work of Boyarchenko--Weinstein on the cohomology of a special affinoid in the Lubin--Tate tower. We prove that the number of rational points of $X_h$ attains its Weil--Deligne bound, so that the cohomology of $X_h$ is pure in a very strong sense. We prove that the torus-eigenspaces of $H_c^i(X_h)$ are irreducible representations and are supported in exactly one cohomological degree. Finally, we give a complete description of the homology groups of the semi-infinite Deligne--Lusztig varieties attached to any division algebra, thus giving a geometric realization of a large class of supercuspidal representations of these groups. The techniques developed in this paper should be useful in studying these constructions for $p$-adic groups in general.
\end{abstract}

\small

\tableofcontents

\normalsize

\newpage

\section{Introduction}\label{s:introduction}

The seminal work of Deligne and Lusztig on the representations of finite reductive groups \cite{DL76} has influenced an industry studying parallel constructions in the same theme. Classically, one begins with a reductive group $\bG$ over $\FF_q$ together with a maximal torus $\bT \subset \bG$, and writing $G \colonequals \bG(\FF_q)$, $T \colonequals \bT(\FF_q)$, one considers the $G$-representations arising from the torus-eigenspaces $H_c^i(\widetilde X_{T \subset G}, \overline \QQ_\ell)[\theta]$ of the cohomology of the Deligne--Lusztig variety $\widetilde X_{T \subset G}$. The Deligne--Lusztig variety is a $T$-torsor over a subvariety of the flag variety, and can be defined to be
\begin{equation*}
\widetilde X_{T \subset G} \colonequals (\widetilde U \cap F^{-1}(\widetilde U)) \backslash \{g \in \widetilde G : F(g) g^{-1} \in \widetilde U\},
\end{equation*}
where $\widetilde T = \bT(\overline \FF_q)$, $\widetilde G = \bG(\overline \FF_q)$, $\widetilde U \subset \widetilde G$ is the unipotent radical of a Borel $\widetilde B \subset \widetilde G$ containing $\widetilde T$, and $F \from \widetilde G \to \widetilde G$ is the group automorphism induced by the Frobenius element in $\Gal(\overline \FF_q/\FF_q)$. 

Following this philosophy, it is natural to ask whether a similar construction can be used to study representations of reductive groups over finite rings or over local fields. Both of these set-ups have been studied by Lusztig---for reductive groups over finite rings, see for example \cite{L04}, and for reductive groups over local fields, see \cite{L79}. In the present paper, we study the latter situation using the construction proposed by Lusztig in \cite{L79}. We recall Lusztig's construction now. The analogy with the classical Deligne--Lusztig varieties $\widetilde X_{T \subset G}$ will be apparent.

For a reductive group $\bG$ over a non-Archimedean local field $K$ with residue field $\FF_q$, let $\bT \subset \bG$ be a maximal unramified torus anisotropic over $K$, and write $T \colonequals \bT(K)$, $G \colonequals \bG(K)$, and $\widetilde T \colonequals \bT(\Khat)$, $\widetilde G \colonequals \bG(\Khat)$. Assume that there exists a Borel subgroup $\widetilde B \subset \widetilde G$ defined over $\Khat$ and containing $\widetilde T$, and let $\widetilde U$ be its unipotent radical. Let $F \from \widetilde G \to \widetilde G$ denote the group automorphism induced by the Frobenius element in $\Gal(\Khat/K)$. Lusztig's construction, which we call the \textit{semi-infinite Deligne--Lusztig set}, is the quotient 
\begin{equation*}
\widetilde X \colonequals (\widetilde U \cap F^{-1}(\widetilde U)) \backslash \{g \in \widetilde G : F(g) g^{-1} \in \widetilde U\}.
\end{equation*}
In \cite{L79}, Lusztig suggests that $\widetilde X$ should have the structure of an infinite-dimensional variety over $\overline \FF_q$ and that for a fixed character $\theta \from T \to \overline \QQ_\ell^\times$, the subspace $H_i(\widetilde X, \overline \QQ_\ell)[\theta]$ wherein $T$ acts by $\theta$ should be zero for large $i$ and be concentrated in a single cohomological degree if $\theta$ is in general position. The long-term goal is to give a uniform construction of supercuspidal representations of $p$-adic groups by realizing them in the cohomology of $\widetilde X$.

We make this precise and realize this goal in the case when $G = D^\times$ for an $n^2$-dimensional division algebra $D = D_{k/n}$ over $K$ and $T = L^\times$, where $L$ is the maximal unramified extension of $K$. Following Boyarchenko \cite{B12}, we define an ind-scheme structure on $\widetilde X$ together with homology groups $H_i(\widetilde X, \overline \QQ_\ell)$  that yield smooth representations of $T \times G$. We now state the main theorem, which gives a complete description of the subspaces $H_i(\widetilde X, \overline \QQ_\ell)[\theta]$.


\begin{main}\label{mt:1}
Let $\theta \from T \to \overline \QQ_\ell^\times$ be a smooth character. 

\begin{enumerate}[label=(\alph*)]
\item
There exists an $r_\theta \in \bZ_{\geq 0}$ such that 
\begin{equation*}
H_i(\widetilde X, \overline \QQ_\ell)[\theta] \neq 0 \quad \Longleftrightarrow \quad i = r_\theta.
\end{equation*}
Furthermore, $r_\theta$ can be determined in terms of the Howe factorization of $\theta$, which measures the extent to which $\theta$ arises from characters that factor through norm maps.

\item
For any very regular element $x \in \cO_L^\times \subset G$,
\begin{equation*}
\Tr\left(x; H_{r_\theta}(\widetilde X, \overline \QQ_\ell)[\theta]\right) = (-1)^{r_\theta} \cdot \sum_{\gamma \in \Gal(L/K)} \theta^\gamma(x).
\end{equation*}

\item
If $\theta$ has trivial $\Gal(L/K)$-stabilizer, then the $G$-representation $H_{r_\theta}(\widetilde X, \overline \QQ_\ell)[\theta]$ is irreducible.
\end{enumerate}
\end{main}

\begin{remark}
When $\theta$ has trivial $\Gal(L/K)$-stabilizer, the corresponding induced representation of the Weil group $\cW_K$ is irreducible. In this setting, by the work of Henniart \cite{H93}, part (b) of the Main Theorem implies that as we vary the division algebra $D$ and its semi-infinite Deligne--Lusztig variety $\widetilde X$, the irreducible $D^\times$-representations $H_{r_{\theta}}(\widetilde X, \overline \QQ_\ell)[\theta]$ correspond to each other under Jacquet--Langlands. \hfill $\Diamond$
\end{remark}

The proof of the Main Theorem \ref{mt:1} splits into two parts:
\begin{enumerate}[label=(\arabic*)]
\item
Prove that if $\theta|_{U_L^h} = 1$ where $U_L^h = 1 + \pi^h \cO_L$, then
\begin{equation*}
H_i(\widetilde X, \overline \QQ_\ell)[\theta] = \text{``}\Ind\left(H_c^{2d-i}(X_h, \overline \QQ_\ell)[\chi]\right)\text{''},
\end{equation*}
where $X_h$ is a finite-type variety of pure dimension $d$ and $\chi \colonequals \theta|_{U_L^1}$.


\item
For any $\chi \from U_L^1/U_L^h \to \overline \QQ_\ell^\times$, characterize $H_c^i(X_h, \overline \QQ_\ell)[\chi]$ as a representation of a certain finite unipotent group $\Unip(\F)$.
\end{enumerate}
Part (1) was proved by Boyarchenko in \cite{B12}, where he made two conjectures about Part (2) (see Conjectures 5.16 and 5.18 of \textit{op.\ cit}). In light of this, the obstruction to understanding $H_i(\widetilde X, \overline \QQ_\ell)[\theta]$ reduces to the study of $H_c^i(X_h, \overline \QQ_\ell)[\chi]$. For the reader's convenience, we recall the two conjectures on these cohomology groups here.

\begin{conjecture}[Boyarchenko]\label{c:boy 1}
$X_h$ is a maximal variety in the sense of Boyarchenko--Weinstein \cite{BW16}. Equivalently, we have $H_c^i(X_h, \overline \QQ_\ell) = 0$ unless $i$ or $n$ is even, and the geometric Frobenius $\Fr_{q^n}$ acts on $H_c^i(X_h, \overline \QQ_\ell)$ by the scalar $(-1)^i q^{ni/2}$.
\end{conjecture}

\begin{conjecture}[Boyarchenko]\label{c:boy 2}
Given a character $\chi \from U_L^1/U_L^h \to \overline \QQ_\ell^\times$, there exists $r_\chi \geq 0$ such that $H_c^i(X_h, \overline \QQ_\ell)[\chi] = 0$ for all $i \neq r_\chi$. Moreover, $H_c^{r_\chi}(X_h, \overline \QQ_\ell)[\chi]$ is an irreducible representation of $\Unip(\F)$.
\end{conjecture}

\begin{remark}
A finite-type $\FF_Q$-scheme $S$ is a \textit{maximal variety} if $S(\FF_Q)$ attains its Weil--Deligne bound. More explicitly, by the Grothendieck--Lefschetz trace formula and Deligne's work on the Weil conjectures \cite[Theorem 3.3.1]{D80},
\begin{equation*}
\#S(\FF_Q) = \sum_{i \in \bZ} (-1)^i \Tr(\Fr_Q, H_c^i(S, \overline \QQ_\ell)) \leq \sum_{i \in \bZ} Q^{i/2} \dim H_c^i(S, \overline \QQ_\ell).
\end{equation*}
Observe that this bound is realized exactly when $\Fr_Q$ acts by the scalar $(-1)^i Q^{i/2}$. \hfill $\Diamond$
\end{remark}

We now describe the progress on these two conjectures prior to the present work. Under the assumption $k = 1$, Conjectures \ref{c:boy 1} and \ref{c:boy 2} were proved in the $h = 2$ case by Boyarchenko and Weinstein in \cite{BW16}, where they show that the perfection of $X_2$ is the special fiber of a particular open affinoid $\cV$ in the Lubin--Tate tower, and then prove that the cohomology of $\cV$ realizes the local Langlands and Jacquet--Langlands correspondences by calculating $H_c^i(X_2, \overline \QQ_\ell)[\chi]$. For $h > 2$, the results are more sparse. Nearly nothing was known about Conjecture \ref{c:boy 1}, and the only work on Conjecture \ref{c:boy 2} required the assumption that $\chi$ is a primitive character (that is, the restriction of $\chi$ to $U_L^{h-1}/U_L^h \cong \FF_{q^n}$ has trivial $\Gal(\F/\FF_q)$-stabilizer). For primitive $\chi$, Conjecture \ref{c:boy 2} was proved: in \cite{B12} for $k=1,$ $h=3$, $n=2$, and $\Char K > 0$; in \cite{C16} for $k,h$ arbitrary, $n = 2$, and $\Char K > 0$; and in \cite{C15} for $k, h, n$ arbitrary, and $\Char K > 0$. In these works, it was also shown that $\Fr_{q^n}$ acts on $H_c^i(X_h, \overline \QQ_\ell)[\chi]$ in the predicted way for $\chi$ primitive, and this is the extent to which Conjecture \ref{c:boy 1} was known for $h > 2$.

The approach of this paper is of a somewhat different nature to the earlier work towards Boyarchenko's conjectures. Before this paper, the idea was to define a unipotent group scheme $U_{h,k}^{n,q}$ over $\FF_{q^n}$ such that $U_{h,k}^{n,q}(\FF_{q^n})$ is a subquotient of $\cO_D^\times$ and stabilizes a $\FF_{q^n}$-subscheme $X_h \subset U_{h,k}^{n,q}$ via the natural multiplication action. This approach has the disadvantage that $U_{h,k}^{n,q}$ does not arise naturally from the set-up and also does not have a mixed characteristic analogue. (Only when $h=2$ does the unipotent group scheme work in mixed characteristic, and \cite{BW16} makes use of this.) In this paper we work with a unipotent group scheme $\Unip$ over $\FF_q$ that arises naturally from the set-up and has the feature that $\Unip(\FF_q)$ is a subquotient of $\cO_D^\times$ and stabilizes a $\FF_{q^n}$-subscheme $X_h \subset \Unip$ via the natural multiplication action. There is a subtle issue that $X_h$ is not defined over $\FF_q$, but this can be dealt with by relating the cohomology of $X_h$ to its Weil restriction $\Res_{\FF_{q^n}/\FF_q}X_h$, which can be realized as the union of $q$-Frobenius translates of $X_h$ in $\Unip$.

\begin{main}\label{mt:2}
The two conjectures of Boyarchenko are true.
\end{main}

The proof of the two conjectures of Boyarchenko comprises most of the present paper. We expect that the techniques we develop and use to prove these conjectures should be applicable to studying the cohomology of Deligne--Lusztig varieties over finite rings and the homology groups of semi-infinite Deligne--Lusztig varieties for more general reductive groups. 

\begin{remark}
The term ``semi-infinite Deligne--Lusztig variety'' is intended to be reminiscent of the following analogies:
\begin{align*}
\text{Deligne--Lusztig variety} \quad &: \quad \text{flag variety} \\
\text{affine Deligne--Lusztig variety} \quad &: \quad \text{affine flag variety} \\
\text{semi-infinite Deligne--Lusztig variety} \quad &: \quad \text{semi-infinite flag variety}
\end{align*}
Here, we use the term ``semi-infinite'' in the sense of Feigin--Frenkel \cite{FF90}. We also remark that there is a close relationship between semi-infinite Deligne--Lusztig varieties and affine Deligne--Lusztig varieties of higher level in the sense of \cite{I15}. This is ongoing work with A.\ Ivanov.
\hfill $\Diamond$
\end{remark}

\subsection{Outline of the paper}

Fix coprime integers $k,n \geq 1$. In Section \ref{s:definitions}, we define semi-infinite Deligne--Lusztig sets and recall the (ind-pro-)scheme structure on the set $\widetilde X$ attached to a division algebra. This naturally leads us to study a family of finite-type $\FF_{q^n}$-schemes $X_h$ that arise as subschemes of a unipotent group scheme $\Unip$ over $\FF_q$. There is a natural subgroup scheme $\TUnip \subset \Unip$ analogous to $\bT \subset \bG$ as in the introduction, and the natural left- and right-multiplication actions of $\TUnip(\FF_q)$ and $\Unip(\FF_q)$ on $\Unip$ stabilize $X_h$. In Remark \ref{r:unip}, we discuss how $\Unip$ differs from the unipotent group schemes $U_{h,k}^{n,q}$ appearing in  \cite{BW16, B12, C16, C15}.

We also define the notion of the Howe factorization of a character of $\TUnip(\FF_q)$ in the sense of \cite{H77}. The Howe factorization gives rise to a pair of sequences $(\{m_i\}, \{h_i\})$ from which one can define a stratification of an indexing set $\cA^+$ (see Section \ref{s:index}). One of the most important features of $\cA^+$ is that it is normed and satisfies Lemma \ref{l:det contribution}. As a quick application, we prove in Section \ref{s:dim Xh} that $X_h$ is smooth, affine, and has dimension $(n-1)(h-1)$. 

In Section \ref{s:LS}, we prove several general results on the cohomology of constructible $\overline \QQ_\ell$-sheaves coming from pullbacks of local systems. Proposition \ref{p:B2.3} is a generalization of \cite[Proposition 2.3]{B12} that allows one to calculate spaces of homomorphisms between representations of $G(\FF_q)$ and cohomology groups of $X \subset G$ even if $X$ is not defined over $\FF_q$. Proposition \ref{p:induct} relates the cohomology of a scheme $S$ to the cohomology of a subscheme of smaller dimension, and Proposition \ref{p:inductm} is a particular specialization of this proposition which will be one of the main structural techniques used in the proof of Theorem \ref{t:hom}. We remark that one can view \cite[Proposition 2.10]{B12}, \cite[Propositions 3.4, 3.5]{C15}, and Proposition \ref{p:inductm} as variations of the same theme under the umbrella of Proposition \ref{p:induct}.

The content of Section \ref{s:morphisms} is Theorem \ref{t:hom}, and this is really the heart of the proof of Conjectures \ref{c:boy 1} and \ref{c:boy 2}. In this theorem, we prove
\begin{equation*}
\Hom_{\Unip(\F)}\left(\Ind_{\TUnip(\F)}^{\Unip(\F)}(\chi), H_c^i(X_h, \overline \QQ_\ell)\right) \neq 0 \quad \Longleftrightarrow \quad i = r_\chi,
\end{equation*}
where $r_\chi$ can be given explicitly in terms of the Howe factorization of $\chi$. The driving idea of the proof is to calculate certain cohomology groups by inducting along linear fibrations (using Proposition \ref{p:inductm}) determined by the stratification of $\cA^+$ associated to the Howe factorization of $\chi$.

We begin Section \ref{s:DL} by proving (Theorem \ref{t:R chi}) that the alternating sum
\begin{equation*}
\textstyle \sum (-1)^i H_c^i(X_h, \overline \QQ_\ell)[\chi]
\end{equation*}
is an irreducible representation of $\Unip(\FF_q)$. This is very close to results of Lusztig \cite{L79,L04} and Stasinski \cite{S09}. This result together with Theorem \ref{t:hom} immediately implies Conjecture \ref{c:boy 2}. There is a subtlety that occurs here: while it follows from Theorem \ref{t:hom} and Theorem \ref{t:R chi} that the cohomology groups $H_c^i(X_h, \overline \QQ_\ell)[\chi]$ are concentrated in a single degree, it requires a nontrivial argument to show that this non-vanishing cohomological degree is the $r_\chi$ from Theorem \ref{t:hom}. This is the content of Theorem \ref{t:s chi}. In Section \ref{s:examples}, we show that $r_\chi$ indeed specializes to the formulas obtained in \cite{BW16} and \cite{C15}. 

We also prove related results in the same theme. We prove a multiplicity-one statement (Theorem \ref{t:irreducibility}): the association $\chi \mapsto H_c^*(X_h, \overline \QQ_\ell)[\chi]$ defines an injection into the set of irreducible representations of $\Unip(\FF_q)$. In addition, we compute the character of $H_c^*(X_h, \overline \QQ_\ell)[\chi]$ on the set of very regular elements of $L^\times$ (Theorem \ref{t:character}), which we utilize in the proof of Theorem \ref{t:s chi}. Finally, we give an explicit formula for the zeta function of $X_h$ (Theorem \ref{t:zeta}). 

In the concluding section, Section \ref{s:pDL}, we use the discussion of semi-infinite Deligne--Lusztig varieties in Section \ref{s:def DL} and the theorems of Sections \ref{s:morphisms} and \ref{s:DL} to prove the Main Theorem. 

The techniques of this paper should be directly applicable to studying the homology groups of semi-infinite Deligne--Lusztig varieties attached to an arbitrary pair $T \subset G$, where $T$ is a maximal unramified torus of a reductive group $G$ over a non-Archimedean local field. For example, this can be done when $G$ is any pure inner form of $\GL_n$, and this is part of ongoing work with A.\ Ivanov. See Remark \ref{r:other groups} for a more technical discussion.

\addtocontents{toc}{\SkipTocEntry}

\subsection*{Acknowledgements}

I'd like to thank Bhargav Bhatt for several helpful conversations and Alex Ivanov for helpful comments on an earlier draft. This work was partially supported by NSF grants DMS-0943832 and DMS-1160720.


\section{Definitions}\label{s:definitions}

We fix, once and for all, an integer $n \geq 1$, a non-Archimedean local field $K$ with finite residue field $\FF_q$ of characteristic $p$, and a uniformizer $\pi$ of $K$. Let $L$ be the unique degree-$n$ unramified extension of $K$ whose ring of integers $\cO_L$ has a unique maximal ideal denoted by $P_L$. Write $U_L^h \colonequals 1 + P_L^h$. For a division algebra $D$ over $K$, we denote by $\cO_D$ its ring of integers (i.e.\ its unique maximal order) and denote by $P_D$ the unique maximal ideal of $\cO_D$. Write $U_D^h \colonequals 1 + P_D^h$. If $D$ has Hasse invariant $k/n \in \bQ/\bZ$, where $(k,n) = 1$, we sometimes write $D = D_{k/n}$.

If $K$ has characteristic $p$, we let $\bW(A) = A[\![\pi]\!]$ for any $\FF_q$-algebra $A$. If $K$ has characteristic $0$, we let $\bW$ be the ring scheme of $\cO_K$-Witt vectors. As usual, we have the Frobenius and Verschiebung morphisms
\begin{align*}
\varphi &\from \bW \to \bW, \qquad [a_i]_{i \geq 0} \mapsto [a_i^q]_{i \geq 0}, \\
V &\from \bW \to \bW, \qquad [a_i]_{i \geq 0} \mapsto [0, a_0, a_1, \ldots].
\end{align*}
One also has a morphism
\begin{equation*}
r \from A \to \bW(A), \qquad a \mapsto [a,0,\ldots]
\end{equation*}
for any $\FF_q$-algebra $A$. Note that
\begin{equation*}
\pi \cdot [a_i]_{i \geq 0} = \begin{cases}
[0, a_0, a_1, \ldots] & \text{if $\Char K > 0$}, \\
[0, a_0^q, a_1^q, \ldots] & \text{if $\Char K = 0$}.
\end{cases}
\end{equation*}
For any $h \in \bN$, let $\bW_h \colonequals \bW/V^h\bW$ be the corresponding truncated ring schemes. For any $r \in \bN$, consider the group schemes
\begin{align*}
\bW^{(r)} &\colonequals 1 + V^r \bW \subset \bW^\times \\
\bW_h^{(r)} &\colonequals 1 + V^r\bW_{h-r} \subset \bW_h^\times.
\end{align*}
These are all defined over $\FF_q$.

For any integer $m$, define $[m]$ to be the unique integer with $1 \leq [m] \leq n$ such that $m \equiv [m]$ modulo $n$.

\subsection{Semi-infinite Deligne--Lusztig varieties for division algebras}\label{s:def DL}

Let $\bG$ be a connected reductive group over $K$ and write $G = \bG(K)$ and $\widetilde G = \bG(\Khat)$ so that $G = \widetilde G{}^F$, where $F \from \widetilde G \to \widetilde G$ is the Frobenius map induced by the arithmetic Frobenius in $\Gal(\Khat/K)$. Let $\bT \subset \bG$ be an anisotropic maximal torus over $K$ and assume that there exists a Borel subgroup $\bB$ of $\bG$ defined over $\Khat$ such that $\bT(\Khat) \subset \bB(\Khat) \equalscolon \widetilde B$. Let $\widetilde U$ denote the unipotent radical of $\widetilde B$.

\begin{definition}\label{d:DL}
The \textit{semi-infinite Deligne--Lusztig set} associated to $G \colonequals \bG(K) = \widetilde G{}^F$ and $T \colonequals \bT(K) = \widetilde T{}^F$ is the quotient
\begin{equation*}
X \colonequals (\widetilde U \cap F^{-1}(\widetilde U)) \backslash \{g \in \widetilde G : F(g) g^{-1} \in \widetilde U\}.
\end{equation*}
It is clear that $X$ is endowed with a left-multiplication action of $T$ and a right-multiplication action of $G$.
\end{definition}

\begin{remark}
Definition \ref{d:DL} is due to Lusztig \cite{L79}, where he defined $X$ to be
\begin{equation*}
\{g \in \widetilde G : g^{-1} F(g) \in \widetilde U\}/(\widetilde U \cap F^{-1}(\widetilde U)),
\end{equation*}
which has a right-multiplication action of $T$ and a left-multiplication action of $G$. In this paper, we follow the convention set in \cite{B12}, where the quotient is taken on the other side. \hfill $\Diamond$
\end{remark}

The Brauer group of the local field $K$ is isomorphic to $\bQ/\bZ$. Hence for any integer $1 \leq k \leq n$ with $(k,n) = 1$, there is a corresponding division algebra $D_{k/n}$ of dimension $n^2$ over $K$. The group $L^\times$ is an unramified anisotropic torus in $D_{k/n}^\times$, and we can realize $L^\times \hookrightarrow D_{k/n}^\times$ in this framework in two ways. Set $\widetilde G \colonequals \GL_n(\Khat)$ and consider the automorphisms 
\begin{align}\label{e:F1}
F_1 &\from \widetilde G \to \widetilde G, & g &\mapsto \varpi_k^{-1} \varphi(g) \varpi_k, & \varpi_k &\colonequals \left(\begin{matrix} 0 & 1_{n-1} \\ \pi^k & 0 \end{matrix}\right), \\ \label{e:F2}
F_2 &\from \widetilde G \to \widetilde G, & g &\mapsto \varpi^{-k} \varphi(g) \varpi^k, & \varpi &\colonequals \left(\begin{matrix} 0 & 1_{n-1} \\ \pi & 0 \end{matrix}\right),
\end{align}
where $1_{n-1}$ denotes the identity matrix of size $(n-1) \times (n-1)$ and $\varphi(g)$ is the matrix obtained by applying the arithmetic Frobenius $\varphi \in \Gal(\Khat/K)$ to each entry of $g$. Then for $i = 1,2$, the morphism $F_i$ is a Frobenius relative to a $K$-rational structure and we denote the corresponding algebraic group by $\bG_i$. Consider the diagonal torus $\widetilde T \subset \widetilde G$ and let $\widetilde B \subset \widetilde G$ be the standard Borel. Since $F_i$ stabilizes $\widetilde T$, it defines a $K$-rational structures on $\widetilde T$, and we can denote the corresponding algebraic group by $\bT_i$. Note that neither $F_1$ nor $F_2$ stabilizes $\widetilde B$ and we have 
\begin{equation*}
\bG_1(K) \overset{\cong}{\longrightarrow} \bG_2(K), \qquad \bT_1(K) \overset{\cong}{\longrightarrow} \bT_2(K),
\end{equation*}
where the isomorphism is given by $f \from g \mapsto \gamma^{-1} \cdot g \cdot \gamma$, where $\gamma = \gamma_0 \cdot  \diag(\pi^{\lambda_1}, \ldots, \pi^{\lambda_n})$ for a permutation matrix $\gamma_0$ and for some $\lambda_1, \ldots, \lambda_n \in \bZ$. Since the image of $\varpi$ in the Weyl group has order $n$, we may assume that $e_1 \cdot \gamma_0 = e_1$, where $e_1$ is the first elementary row vector. For $i = 1,2$, set $G_i \colonequals \bG_i(K) \supset \bT_i(K) \equalscolon T_i$. We have $G_1 \cong G_2 \cong D_{k/n}^\times$ and $T_1 \cong T_2 \cong L^\times$.

Let $X$ denote the semi-infinite Deligne--Lusztig set associated to $G_1$ and $T_1$. We now recollect how to realize $X$ as the $\overline \FF_q$-points of an infinite-dimensional scheme using a method suggested by Lusztig in \cite{L79} and formalized by Boyarchenko in \cite{B12}. By \cite[Corollary 4.3]{B12},  $X$ can be identified with the set
\begin{equation*}
\widetilde X \colonequals \{g \in \widetilde G : F_1(g) g^{-1} \in \widetilde U \cap F_1(\widetilde U^-)\},
\end{equation*}
where $\widetilde U^-$ is the unipotent radical of the opposite Borel to $\widetilde B$. By \cite[Lemma 4.4]{B12}, a matrix $A \in \widetilde G$ belongs to $\widetilde X$ if and only if it has the form
\begin{equation}\label{e:A form}
A = x(a_1, \ldots, a_n) \colonequals \left(\begin{matrix}
a_1 & a_2 & a_3 & \cdots & a_n \\
\pi^k \varphi(a_n) & \varphi(a_1) & \varphi(a_2) & \cdots & \varphi(a_{n-1}) \\
\pi^k \varphi^2(a_{n-1}) & \pi^k \varphi^2(a_n) & \varphi^2(a_1) & \cdots & \varphi^2(a_{n-2}) \\
\vdots & \vdots & \ddots & \ddots & \vdots \\
\pi^k \varphi^{n-1}(a_2) & \pi^k \varphi^{n-1}(a_3) & \pi^k \varphi^{n-1}(a_4) & \cdots & \varphi^{n-1}(a_1)
\end{matrix}\right),
\end{equation}
where $a_i \in \Khat$ and $\det(A) \in K^\times$. (Note the indexing difference between Equation \eqref{e:A form} and \cite[Equation (4.5)]{B12}.) We may therefore write 
\begin{equation*}
\widetilde X = \bigsqcup_{m \in \bZ} \widetilde X^{(m)},
\end{equation*}
where $\widetilde X^{(m)}$ consists of all $A \in \widetilde X$ with $\det(A) \in \pi^m  \cO_K^\times$. Note that the action of $\varpi_k$ takes each $\widetilde X^{(m)}$ isomorphically onto $\widetilde X^{(m+k)}$, and the action of $\pi$ takes each $\widetilde X^{(m)}$ isomorphically onto $\widetilde X^{(m+n)}$. Since $(k,n) = 1$ by assumption, the $\widetilde X^{(m)}$ are all isomorphic. It is therefore sufficient to show that $\widetilde X^{(0)}$ can be realized as the $\overline \FF_q$-points of a scheme. To do this, we use Lemma \ref{l:4.5}, whose proof (see \cite[Lemma 7.1]{C15})  is nearly exactly the same as that of \cite[Lemma 4.5]{B12}. 

\begin{lemma}[Boyarchenko, {\cite[Lemma 4.5]{B12}}]\label{l:4.5}
Assume $(k,n) = 1$. If a matrix $A$ of the form \eqref{e:A form} satisfies $\det(A) \in \cO_K^\times$, then $a_j \in \pi^{- \lfloor (j-1)k/n \rfloor} \widehat \cO_K^{\ur}$ for $1 \leq j \leq n$ and $a_1 \in (\widehat \cO_K^{\ur})^\times$.
\end{lemma}

We have now shown that $\widetilde X_h^{(0)}$ consists of matrices of the form
\begin{equation*}
A(a_1,a_2, \ldots, a_n) \colonequals x(a_1', a_2',\ldots, a_n')
\end{equation*}
for some $a_1 \in (\widehat \cO_K^{\ur})^\times$ and $a_j \in \widehat \cO_K^{\ur}$ for $2 \leq j \leq n$, where we write
\begin{equation*}
a_j' \colonequals \pi^{-\lfloor (j-1)k/n \rfloor} a_j, \qquad \text{for $1 \leq j \leq n$}.
\end{equation*}
Note that the $(L^\times \times D^\times)$-action on $\widetilde X$ induces a $(\cO_L^\times \times \cO_D^\times)$-action on $\widetilde X^{(0)}$. (The stabilizer of $\widetilde X^{(0)}$ in $L^\times \times D^\times$ is actually slightly bigger, but we save this discussion for Section \ref{s:pDL}.) 

For each integer $h \geq 1$, define $\widetilde X_h^{(0)}$ to be the set of matrices
\begin{align*}
\widetilde X_h^{(0)} \colonequals \{A(a_1, a_2, \ldots, a_n) : {}
&a_1 \in (\widehat \cO_K^{\ur}/\pi^h \widehat \cO_K^{\ur})^\times, \\
&\text{$a_j \in \widehat \cO_K^{\ur}/\pi^{h-1}\widehat \cO_K^{\ur}$ for $2 \leq j \leq n$,} \\
&\text{$\det(A(a_1, \ldots, a_n)) \in (\cO_K/\pi^h \cO_K)^\times$}\}.
\end{align*}
This can be naturally viewed as the set of $\overline \FF_q$-points of a scheme of finite type over $\FF_q$ (see \cite[Section 4.5]{B12} for $k=1$ and \cite[Section 7.1]{C15} for the completely analogous general case). 

Note that $\widetilde X_h^{(0)}$ has a left-multiplication action of $\cO_L^\times/U_L^h$ and a right-multiplication action of $\cO_D^\times/U_D^{n(h-1)+1}$, and these actions are defined over $\FF_{q^n}$. Because of this, \textit{from now on}, we regard $\widetilde X_h^{(0)}$ as an $\FF_{q^n}$-scheme. By Lemma \ref{l:4.5}, we have $\widetilde X^{(0)} = \varprojlim_h \widetilde X_h^{(0)}$, so that $\widetilde X^{(0)}$ is the set of $\overline \FF_q$-points of a (pro-)scheme over $\FF_{q^n}$. Having this, we now define $\ell$-adic homology groups of $\widetilde X^{(0)}$.

\begin{lemma}[Boyarchenko, {\cite[Lemma 4.7]{B12}}]\label{l:4.7}
Set $W_h \colonequals \ker(\bW_h(\FF_{q^n})^\times \to \bW_{h-1}(\FF_{q^n})^\times)$ for $h \geq 2$. The action of $W_h$ on $\widetilde X_h^{(m)}$ preserves every fiber of the natural map $\widetilde X_h^{(m)} \to \widetilde X_{h-1}^{(m)}$, the induced morphism $W_h \backslash \widetilde X_h^{(m)} \to \widetilde X_{h-1}^{(m)}$ is smooth, and each of its fibers is isomorphic to the affine space $\bA^{n-1}$ over $\overline \FF_q$.
\end{lemma}

For a scheme $S$ of pure dimension $d$, set $H_i(S, \overline \QQ_\ell) \colonequals H_c^{2d-i}(S, \overline \QQ_\ell)$. By Lemma \ref{l:4.7},
\begin{equation*}
H_i(\widetilde X_{h-1}^{(m)}, \overline \QQ_\ell) \overset{\cong}{\longrightarrow} H_i(\widetilde X_h^{(m)}, \overline \QQ_\ell)^{W_h}
\end{equation*}
and in particular, we have a natural embedding $H_i(\widetilde X_{h-1}^{(m)}, \overline \QQ_\ell) \hookrightarrow H_i(\widetilde X_h^{(m)}, \overline \QQ_\ell)$. We set
\begin{equation*}
H_i(\widetilde X^{(m)}, \overline \QQ_\ell) \colonequals \varinjlim_h H_i(\widetilde X_h^{(m)}, \overline \QQ_\ell), \qquad H_i(\widetilde X, \overline \QQ_\ell) \colonequals \bigoplus_m H_i(\widetilde X^{(m)}, \overline \QQ_\ell).
\end{equation*}
For each $i \geq 0$, the vector space $H_i(\widetilde X, \overline \QQ_\ell)$ inherits commuting smooth actions of $L^\times$ and $D_{k/n}^\times$. 

We will need a slightly different incarnation of $\widetilde X_h^{(0)}$. The morphism $f \from \bG_1 \to \bG_2$ given by $g \mapsto \gamma^{-1} \cdot g \cdot \gamma$ induces a $(\cO_L^\times/U_L^h) \times (\cO_D^\times/U_D^{n(h-1)+1})$-equivariant isomorphism of $\FF_{q^n}$-schemes
\begin{equation*}
\widetilde X_h^{(0)} \overset{\cong}{\longrightarrow} \widetilde X_h'{}^{(0)},
\end{equation*}
where if we write $A'(a_0, \ldots, a_{n-1}) \colonequals \gamma^{-1} \cdot A(a_0, \ldots, a_{n-1}) \cdot \gamma$, then
\begin{align*}
\widetilde X_h'{}^{(0)} =  \{A'(a_0, \ldots, a_{n-1}) : {}
&a_0 \in (\widehat \cO_K^{\ur}/\pi^h \widehat \cO_K^{\ur})^\times, \\
&\text{$a_j \in \widehat \cO_K^{\ur}/\pi^{h-1}\widehat \cO_K^{\ur}$ for $1 \leq j \leq n-1$,} \\
&\text{$\det(A'(a_0, \ldots, a_{n-1})) \in (\cO_K/\pi^h \cO_K)^\times$}\}.
\end{align*}
Observe that the determinant condition holds by multiplicativity. This proves:

\begin{lemma}\label{l:X' to X}
For all $i \geq 0$, as representations of $\cO_L^\times/U_L^h \times \cO_D^\times/U_D^{n(h-1)+1}$,
\begin{equation*}
H_c^i(\widetilde X_h^{(0)}, \overline \QQ_\ell) \cong H_c^i(\widetilde X_h'{}^{(0)},\overline \QQ_\ell)
\end{equation*}
\end{lemma}

In the next subsections, we will define a subvariety $X_h \subset \widetilde X_h'{}^{(0)}$ satisfying the following: the stabilizer of $X_h$ in $\cO_L^\times/U_L^h \times \cO_D^\times/U_D^{n(h-1)+1}$ is equal to the subgroup
\begin{equation*}
\Gamma_h \colonequals \langle (\zeta^{-1}, \zeta) \rangle \cdot (U_L^1/U_L^h \times U_D^1/U_D^{n(h-1)+1}),
\end{equation*}
and $\widetilde X_h'{}^{(0)}$ is equal to the $(\cO_L^\times/U_L^h \times \cO_D^\times/U_D^{n(h-1)+1})$-translates of $X_h$. This implies that
\begin{equation*}
H_c^i(\widetilde X_h'{}^{(0)}, \overline \QQ_\ell) \cong \Ind_{\Gamma_h}^{\cO_L^\times/U_L^h \times \cO_D^\times/U_D^{n(h-1)+1}}\left(H_c^i(X_h, \overline \QQ_\ell)\right).
\end{equation*}
The bulk of this paper is devoted to studying the cohomology of $X_h$, and we only return to the setting of the semi-infinite Deligne--Lusztig variety $\widetilde X$ in Section \ref{s:pDL}, the final section of this paper.

\subsection{The unipotent group scheme $\Unip$ and the subscheme $X_h$}

For $r \in \bN$, let $I_r$ denote the $r$th subgroup in the standard Iwahori filtration for $\GL_n(\bW)$ over $\overline \FF_q$. Explicitly, for any $\overline \FF_q$-algebra $A$, let $\fb_0(A)$ be the preimage of the standard Borel subalgebra of $M_n(A)$ under the reduction $M_n(\bW(A)) \to M_n(A)$. Consider the morphism
\begin{equation*}
\cV \from M_n(\bW(A)) \to M_n(\bW(A)), \qquad g \mapsto \left(\begin{matrix} 0 & 1_{n-1} \\ V & 0\end{matrix}\right) g,
\end{equation*}
where if $g = (a_{ij})_{i,j=1}^n$, its image has $(i,j)$th coordinate
\begin{equation*}
(\cV g)_{ij} = \begin{cases}
a_{i+1,j} & \text{if $i = 1, \ldots, n-1$,} \\
V a_{1,j} & \text{if $i = n$.}
\end{cases}
\end{equation*}
For integers $r \geq 1$, define $\fb_{r/n} \colonequals \cV^r \fb_0.$ Then the Iwahori subgroup $I_0 \colonequals \fb_0^\times$ has a filtration given by $I_r \colonequals 1 + \fb_r,$
and for any $0 \leq r \leq s$, we may consider the quotient $I_{r,s} \colonequals I_r/I_s.$ We write $I_{r^+} \colonequals \cup_{s > r} I_s$.



Let $F = F_2$ from Equation \eqref{e:F2}. Note that $I_r$ is stable under $F$ since $\varpi^{-1}\fb_0\varpi \subset \fb_0$.

\begin{definition}\label{d:unip}
$I_{0,(h-1)^+}$ admits an $\FF_q$-rational structure with associated Frobenius $F$. We denote by $\GUnip$ the resulting group scheme defined over $\FF_q$. Define $\TUnip' \subset \GUnip$ to be the subgroup consisting of diagonal matrices and define
\begin{align*}
\Unip &\colonequals I_{0^+,(h-1)^+} \subset \GUnip, \\
\TUnip &\colonequals \{\text{diagonal matrices in $\Unip$}\},
\end{align*}
so that we have group schemes 
\begin{equation*}
\begin{tikzcd}[column sep=small,row sep=small]
\TUnip' \ar[hookrightarrow]{r} & \GUnip \\
\TUnip \ar[hookrightarrow]{r} \ar[hookrightarrow]{u} & \Unip \ar[hookrightarrow]{u}
\end{tikzcd}
\end{equation*}
all defined over $\FF_q$. We view the determinant as a morphism
\begin{equation*}
\det \from \GUnip \to \bW_h^\times.
\end{equation*}
\end{definition}

\begin{remark}
Over $\overline \FF_q$, the above groups have the following explicit description:
\begin{enumerate}[label=(\roman*)]
\item
$\GUnip$ can be identified with the group of $n \times n$ matrices $(M_{ij})$ with $M_{ii} \in \bW_h^\times$, $M_{ij} \in \bW_{h-1}$ for $i < j$, and $M_{ij} \in V\bW_{h-1}$ for $i > j$.

\item
$\TUnip'$ is the subgroup consisting of matrices $(M_{ij}) \in \GUnip$ with $M_{ij} = 0$ if $i \neq j$.

\item
$\Unip$ is the subgroup consisting of matrices $(M_{ij}) \in \GUnip$ with $M_{ii} \in \bW_h^{(1)}$.

\item
$\TUnip$ is the subgroup consisting of matrices $(M_{ij}) \in \Unip$ with $M_{ij} = 0$ if $i \neq j$.
\end{enumerate}
In addition,
\begin{align*}
\GUnip(\FF_q) &\cong \cO_D^\times/U_D^{n(h-1)+1}, \\
\TUnip'(\FF_q) &\cong \cO_L^\times/U_L^h, \\
\Unip(\FF_q) &\cong U_D^1/U_D^{n(h-1)+1}, \\
\TUnip(\FF_q) &\cong U_L^1/U_L^h.
\end{align*}
Observe also that $\GUnip(\FF_q) \cong \F^\times \ltimes \Unip(\FF_q).$ \hfill $\Diamond$
\end{remark}

By construction, $\widetilde X_h^{(0)}(\overline \FF_q) \subset \GUnip(\overline \FF_q)$. Moreover, it is stabilized by $F^n$ (but not by $F$!) and the resulting $\FF_{q^n}$-rational structure agrees with the standard $\FF_{q^n}$-rational structure on $\widetilde X_h^{(0)}$. Hence $\widetilde X_h^{(0)}$ is a subscheme of $\GUnip$ defined over $\FF_{q^n}$.

\begin{definition}
For any $\FF_{q^n}$-algebra $A$, define
\begin{equation*}
X_h(A) \colonequals \widetilde X_h^{(0)}(A) \cap \Unip(A).
\end{equation*}
The finite group $\TUnip(\FF_q) \times \GUnip(\FF_q)$ acts on $\Unip$ by:
\begin{equation*}
(t,(\zeta,g)) * x \colonequals \zeta^{-1} \cdot t^{-1} \cdot x \cdot g \cdot \zeta, \qquad t \in \TUnip(\FF_q), \, (\zeta, g) \in \GUnip(\FF_q).
\end{equation*}
This action stabilizes the $\FF_{q^n}$-subscheme $X_h \subset \Unip$.
\end{definition}

\begin{remark}\label{r:unip}
Note that the unipotent group scheme $\Unip$ is a rather different object to the unipotent group schemes appearing in previous work. 
\begin{enumerate}[label=(\roman*)]
\item
In \cite[Section 4.4.1]{BW16}, the unipotent group $U^{n,q}$ over $\F$ is defined to be the group consisting of formal expressions $1 + a_1 \cdot e_1 + \cdots + a_n \cdot e_n$ which are multiplied according to the rule $e_i \cdot a = a^{q^i} \cdot e_i$ for all $1 \leq i \leq n$ and 
\begin{equation*}
e_i \cdot e_j = \begin{cases}
e_{i+j} & \text{if $i + j \leq n$}, \\
0 & \text{otherwise.}
\end{cases}
\end{equation*}
This can be viewed as the unipotent group associated to the parameters: $h=2$,  $k=1$, arbitrary $n$, and $\Char(K)$ arbitrary.

\item \label{unip B12}
In \cite[Definition 5.5]{B12}, for any $\FF_p$-algebra $A$, the unipotent group $U_h^{n,q}(A)$ is defined to be the elements of $A \langle \tau \rangle/(\tau^{n(h-1)+1})$ with constant term $1$. Here, $A \langle \tau \rangle$ is the twisted polynomial ring with the commutation relation $\tau \cdot a = a^q \cdot \tau$ for $a \in A$. This can be viewed as the unipotent group associated to the parameters: arbitrary $h$, $k=1$, arbitrary $n$, and $\Char(K) = p$. Note that $U_2^{n,q} = U^{n,q}$ and $h=2$ is the only $h$ such that $U_2^{n,q}$ can be used when $\Char(K) = 0$.

\item
The unipotent groups in \cite{C15} and \cite{C16} are specializations of those in \ref{unip B12}. In \cite{C15}, we define a unipotent group $U_{h,k}^{n,q}$ together with a subscheme $X_h \subset U_{h,k}^{n,q}$. The unipotent group is isomorphic to $U_{h}^{n,q^{l}}$ described in \ref{unip B12} where $l$ is an integer with $lk \equiv 1$ modulo $n$, but the variety $X_h$ depends on $k$ (and hence $l$).
%
%
\end{enumerate}
When $\Char(K) = p$, one can prove the theorems of this paper using $U_{h,k}^{n,q}$ instead of $\Unip$, and the proofs are very similar. However, it does not seem possible to formulate a characteristic zero analogue of $U_{h,k}^{n,q}$. The upshot of the unipotent group scheme $\Unip$ over $\FF_q$ defined in Definition \ref{d:unip} is that it removes all hypotheses on the parameters, and in particular, we are able to work with arbitrary $h$ over $K$ of arbitrary characteristic. Furthermore, the definition of $\Unip$ seems to lend itself to generalizations to other reductive groups over local fields. \hfill $\Diamond$
\end{remark}

\subsection{An explicit description of $X_h$}

\begin{definition}\label{d:tau}\mbox{}
We define three bijections associated to $F \colonequals F_2$ and $F' \colonequals F_1$ (see Equations \eqref{e:F1} and \eqref{e:F2}).
\begin{enumerate}[label=(\arabic*)]
\item
Let $\sigma \from \{1, \ldots, n\} \to \{1, \ldots, n\}$ be the permutation such that
\begin{equation*}
F(\diag(y_1, \ldots, y_n)) = \diag(\varphi(y_{\sigma(1)}), \ldots, \varphi(y_{\sigma(n)})).
\end{equation*}

\item
Let $\sigma' \from \{1, \ldots, n\} \to \{1, \ldots, n\}$ be the permutation such that 
\begin{equation*}
F'(\diag(y_1, \ldots, y_n)) = \diag(\varphi(y_{\sigma'(1)}), \ldots, \varphi(y_{\sigma'(n)})).
\end{equation*}

\item
Let $\tau \from \{1, \ldots, n\} \to \{0, \ldots, n-1\}$ be the bijection such that
\begin{equation*}
\gamma_0^{-1} \cdot \diag(y_1, \ldots, y_n) \cdot \gamma_0 = \diag(y_{\tau(1) + 1}, \ldots, y_{\tau(n) + 1}).
\end{equation*}
By the assumption that $e_1 \cdot \gamma_0 = e_1$, we have $\tau(1) = 0$.
\end{enumerate}
\end{definition}

The following lemma is an easy calculation.

\begin{lemma}\label{l:sigma tau}
For each $1 \leq i \leq n$,
\begin{equation*}
\tau(\sigma(i))+1 = \sigma'(\tau(i)+1) = [\tau(i)+1-1] = [\tau(i)].
\end{equation*}
\end{lemma}



\begin{definition}\label{d:M char}
A useful description of $X_h$ is the following. A matrix $M \in X_h$ is of the form
\begin{equation}\label{e:std form}
M = M_1 + M_2 \varpi^k + \cdots + M_n \varpi^{[k(n-1)]},
\end{equation}
where 
\begin{align*}
M_{\sigma^j(1)} &= \diag(M_{1,\sigma^j(1)}, M_{2,\sigma^j(2)}, \ldots, M_{n,\sigma^j(n)}), && \text{for $1 \leq j \leq n$,} \\ 
M_{i, j} &= \begin{cases}
[1,M_{(i,j,1)}, \ldots, M_{(i,j,h-1)}] \in \bW_h^{(1)} & \text{if $i = j$}, \\
[M_{(i,j,1)}, \ldots, M_{(i,j,h-1)}] \in \bW_{h-1} & \text{if $i \neq j$,}
\end{cases}, && \text{for $1 \leq i,j \leq n$.}
\end{align*}
By definition, this implies
\begin{equation}\label{e:entries}
M_{i,\sigma^j(i)} = \varphi^{\tau(i)}(M_{1,\sigma^j(1)}), \qquad \text{for $1 \leq j \leq n$}.
\end{equation}
In particular, $M_{i,i} = \varphi^{\tau(i)}(M_{1,1})$ for $1 \leq i \leq n$. For any $\FF_{q^n}$-algebra $A$, we have $M \in X_h(A)$ if and only if $M$ is of the form \eqref{e:std form} satisfying the condition \eqref{e:entries} together with $\varphi(\det(M)) = \det(M)$, where $\det(M) \in \bW_h^{(1)}(A)$. We call \eqref{e:std form} the \textit{standard form} of an $A$-point of $X_h$.
\end{definition}

\subsection{The Howe factorization}\label{s:howe}

Let $\sT_{n,h}$ denote the set of characters $\chi \from \bW_h^{(1)}(\F) \to \overline \QQ_\ell^\times$. Recall that we have natural surjections $\pr \from \bW_h^{(1)} \to \bW_{h-1}^{(1)}$ and injections $\bG_a \to \bW_h^{(1)}$ given by $x \mapsto [1, 0, \ldots, 0, x]$. Furthermore, for any subfield $F \subset L$, the norm map $L^\times \to F^\times$ induces a map $\Nm \from \bW_h^{(1)}(k_L) \to \bW_h^{(1)}(k_F)$. These maps induce 
\begin{align*}
\pr^* \from \sT_{n,h'} &\to \sT_{n,h}, && \text{for $h' < h$}, \\
\Nm^* \from \sT_{m,h} &\to \sT_{n,h}, && \text{for $m \mid n$}.
\end{align*}
By pulling back along $\bG_a \to \bW_h^{(1)}, x \mapsto (1, 0, \ldots, 0, x)$, we may restrict characters of $\bW_h^{(1)}(\F)$ to characters of $\F$. We say that $\chi \in \sT_{n,h}$ has \textit{conductor $m$} if the stabilizer of $\chi|_{\F}$ in $\Gal(\F/\FF_q)$ is $\Gal(\F/\FF_{q^m})$. If $\chi \in \sT_{n,h}$ has conductor $n$, we say that $\chi$ is \textit{primitive}. We write $\sT_{n,h}^0 \subset \sT_{n,h}$ denote the subset of primitive characters.

We can decompose $\chi \in \sT_{n,h}$ into primitive components in the sense of Howe \cite[Corollary after Lemma 11]{H77}. 

\begin{definition}
A \textit{Howe factorization} of a character $\chi \in \sT_{n,h}$ is a decomposition
\begin{equation*}
\chi = \prod_{i=1}^r \chi_i, \qquad \text{where $\chi_i = \pr^* \Nm^* \chi_i^0$ and $\chi_i^0 \in \sT_{m_i, h_i}^0$},
\end{equation*}
such that $m_i < m_{i+1}$, $m_i \mid m_{i+1}$, and $h_i > h_{i+1}$. It is automatic that $m_i \leq n$ and $h \geq h_i$. For any integer $0 \leq t \leq r$, set $\chi_0$ to be the trivial character and define
\begin{equation*}
\chi_{\geq t} \colonequals \prod_{i=t}^r \chi_i \in \sT_{n,h_t}.
\end{equation*}
\end{definition}

Observe that the choice of $\chi_i$ in a Howe factorization $\chi = \prod_{i=1}^r \chi_i$ is not unique, but the $m_i$ and $h_i$ only depend on $\chi$. Hence the Howe factorization attaches to each character $\chi \in \sT_{n,h}$ a pair of well-defined sequences
\begin{align*}
&1 \equalscolon m_0 \leq m_1 < m_2 < \cdots < m_r \leq m_{r+1} \colonequals n \\
&h \equalscolon h_0 \geq h_1 > h_2 > \cdots > h_r \geq h_{r+1} \colonequals 1
\end{align*}
satisfying the divisibility $m_i \mid m_{i+1}$ for $0 \leq i \leq r$.

\begin{example}
We give some examples of the sequences associated to characters $\chi \in \sT_{n,h}$.
\begin{enumerate}[label=(\alph*)]
\item
If $\chi$ is the trivial character, the associated sequences are 
\begin{equation*}
\{m_0, m_1, m_2\} = \{1, 1, n\}, \qquad \{h_0, h_1, h_2\} = \{h, 1, 1\}.
\end{equation*}

\item
We say that $\chi$ is a primitive character of level $h'$ if $\chi|_{U_L^{h'}} = 1$ and $\chi|_{U_L^{h'-1}/U_L^{h'}}$ has trivial $\Gal(\FF_{q^n}/\FF_q)$-stabilizer. In this setting, $\chi \in \sT_{n,h}$ if we have $h \geq h'$, and then the associated sequences are 
\begin{equation*}
\{m_0, m_1, m_2\} = \{1, n, n\}, \qquad \{h_0, h_1, h_2\} = \{h, h', 1\}. 
\end{equation*}
The case $h = h'$ is studied in \cite{C15,C16}.

\item
If $\chi|_{U_L^2} = 1$ and $\chi|_{U_L^1/U_L^2}$ has conductor $m$, then the associated sequences are 
\begin{equation*}
\{m_0, m_1, m_2\} = \{1, m, n\}, \qquad \{h_0, h_1, h_2\} = \{h, 2, 1\}. 
\end{equation*}
The case $h = 2$ is studied in \cite{B12,BW16}.
\end{enumerate}

\end{example}

\section{Indexing sets}\label{s:index} 

We first recall some basic facts about the ramified Witt vectors, which we will need to work with to handle the $\Char K = 0$ setting. In Section \ref{s:sets}, we define indexing sets and prove some fundamental properties of these indexing sets that will be heavily used in Section \ref{s:morphisms}. As a quick application of Section \ref{s:sets}, we calculate the dimension of $X_h$ in Section \ref{s:dim Xh}.

\subsection{Ramified Witt vectors}\label{s:witt}

In this section, we assume $K$ has characteristic $0$. We first defined a ``simplified version'' of the ramified Witt ring $\bW$.

\begin{definition}
For any $\FF_q$-algebra $A$, let $W(A)$ be the set $A^\bN$ endowed with the following coordinatewise addition and multiplication rule:
\begin{align*}
[a_i]_{i \geq 0} +_W [b_i]_{i \geq 0} &= [a_i + b_i]_{i \geq 0}, \\
[a_i]_{i \geq 0} *_W [b_i]_{i \geq 0} &= \left[\textstyle \sum\limits_{j = 0}^i a_j^{q^{i-j}} b_{i-j}^{q^i}\right]_{i \geq 0}.
\end{align*}
It is a straightforward check that $W$ is a commutative ring scheme over $\FF_q$. It comes with Frobenius and Verschiebung morphisms $\varphi$ and $V$. 
\end{definition}

\begin{lemma}\label{l:simple Witt 2}
Let $A$ be an $\FF_q$-algebra.
\begin{enumerate}[label=(\alph*)]
\item
For any $[a] = [a_i]_{i \geq 0}, [b] = [b_i]_{i \geq 0} \in A^\bN$,
\begin{equation*}
[a] *_\bW [b] = [a] *_W [b] +_W [c],
\end{equation*}
where $[c] = [c_i]_{i \geq 0}$ for some $c_i \in A[a_{i_1}^{e_1} b_{i_2}^{e_2} : i_1 + i_2 < i, \, e_1, e_2 \in \bZ_{\geq 0}]$.

\item
For any $[a] = [a_i]_{i \geq 0}, [b] = [b_i]_{i \geq 0} \in A^\bN$,
\begin{equation*}
[a] +_\bW [b] = [a] +_W [b] +_W [c], 
\end{equation*}
where $[c] = [c_i]_{i \geq 0}$ for some $c_i \in A[a_j, b_j : j < i]$.

\item
For any $[a] = [a_i]_{i \geq 0} \in A^\bN$,
\begin{equation*}
\pi *_\bW [a] = [0,1,0,\ldots] *_W [a] = [0,a_0^q,a_1^q,\ldots].
\end{equation*}
\end{enumerate}
\end{lemma}

\begin{lemma}\label{l:simple Witt n}
Let $A$ be an $\FF_q$-algebra. 
\begin{enumerate}[label=(\alph*)]
\item
For any $[a_1], \ldots, [a_n] \in A^\bN$ where $[a_j] = [a_{j,i}]_{i \geq 0}$,
\begin{equation*}
\prod_{\substack{1 \leq j \leq n \\ \text{w.r.t.\ $\bW$}}} [a_j] = \left(\prod_{\substack{1 \leq j \leq n \\ \text{w.r.t.\ $W$}}} [a_j]\right) +_W [c],
\end{equation*}
where $[c] = [c_i]_{i \geq 0}$ for some $c_i \in A[a_{1,i_1}^{e_1} \cdots a_{n,i_n}^{e_n} : i_1 + \cdots + i_n < i, \, e_1, \ldots, e_n \in \bZ_{\geq 0}].$

\item
For any $[a_1], \ldots, [a_n] \in A^\bN$ where $[a_j] = [a_{j,i}]_{i \geq 0}$,
\begin{equation*}
\sum_{\substack{1 \leq j \leq n \\ \text{w.r.t.\ $\bW$}}} [a_j] = \left(\sum_{\substack{1 \leq j \leq n \\ \text{w.r.t.\ $W$}}} [a_j]\right) +_W [c],
\end{equation*}
where $[c] = [c_i]_{i \geq 0}$ for some $c_i \in A[a_{1,j}, \ldots, a_{n,j} : j < i].$
\end{enumerate}
We call the portion coming from $W$ the ``major contribution'' and $[c]$ the ``minor contribution.''
\end{lemma}

\subsection{Normed indexing sets}\label{s:sets}

We define indexing sets associated to the unipotent group $\Unip$ that will be crucial to the proof of Theorem \ref{t:hom}. 

Define
\begin{align*}
\cA^+ &\colonequals \{(i,j,l) : 1 \leq i, j \leq n, \, 1 \leq l \leq h-1\}, \\
\cA &\colonequals \{(i,j,l) \in \cA^+ : i \neq j\} \\ 
\cA^- &\colonequals \{(i,j,l) \in \cA : i = 1\}.
\end{align*}
We fix an identification $\Unip = \bA[\cA^+]$ as follows: Every point of $\Unip$ is of the form $(A_{i,j})_{1 \leq i,j \leq n}$ where 
\begin{equation}\label{e:Aij}
A_{i,j} = \sum_{l \geq 0} V^l r(A_{(i,j,l^*)}), \qquad \text{where } l^* \colonequals \begin{cases} l & \text{if $i \geq j$,} \\ l+1 & \text{if $i < j$.} \end{cases}
\end{equation}
Thus the element $(i,j,l^*) \in \cA^+$ corresponds to the coefficient of $\pi^l$ in the $(i,j)$th entry of an element of $\Unip$. Continuing this dictionary, $\cA$ corresponds to the elements of $\Unip$ with $1$'s along the diagonal, and $\cA^-$ corresponds to the elements of $\Unip$ with $1$'s along the diagonal and zeros everywhere else but the top row.

\begin{definition}\label{d:norm}
Define a norm on $\cA^+$:
\begin{align*}
\cA^+ &\to \{1, 2, \ldots, n(h-1)\}, \\
(i,j,l) &\mapsto |(i,j,l)| \colonequals [j-i] + n(l-1).
\end{align*}
\end{definition}

Given two sequences of integers
\begin{align} \label{e:m filtration}
&1 \equalscolon m_0 \leq m_1 < m_2 < \cdots < m_r \leq m_{r+1} \colonequals n, \qquad m_i \mid m_{i+1}, \\ \label{e:h filtration}
&h \equalscolon h_0 \geq h_1 > h_2 > \cdots > h_r \geq h_{r+1} \colonequals 1,
\end{align}
we can define the following subsets of $\cA$ for $0 \leq s,t \leq r$:
\begin{align*}
\cA_{s,t} &\colonequals \{(i,j,l) \in \cA : j \equiv i \!\!\!\!\!\pmod{m_s}, \, j \not\equiv i \!\!\!\!\!\pmod{m_{s+1}}, \, 1 \leq l \leq h_t - 1\}, \\
\cA_{s,t}^- &\colonequals \cA_{s,t} \cap \cA^-.
\end{align*}
For the proof of Theorem \ref{t:hom}, we will need to consider the following decomposition of $\cA_{s,t}^-$:
\begin{align*}
\cI_{s,t} &\colonequals \{(1,j,l) \in \cA_{s,t}^- : |(1,j,l)| > n(h_t-1)/2\}, \\
\cJ_{s,t} &\colonequals \{(1,j,l) \in \cA_{s,t}^- : |(1,j,l)| \leq n(h_t-1)/2\}.
\end{align*}
For any real number $\nu \geq 0$, define
\begin{equation*}
\cA_{\geq \nu, t} \colonequals \bigsqcup_{s = \lceil \nu \rceil}^r \cA_{s,t}, \quad \cI_{\geq \nu, t} \colonequals \bigsqcup_{s = \lceil \nu \rceil}^r \cI_{s,t}, \quad \cJ_{\geq \nu, t} \colonequals \bigsqcup_{s = \lceil \nu \rceil}^r \cJ_{s,t}, \quad \cA_{\geq \nu, t}^- \colonequals \cA_{\geq \nu, t} \cap \cA^-,
\end{equation*}
and observe that
\begin{equation*}
\cA_{\geq s, t} = \{(i,j,l) \in \cA : j \equiv i \!\!\!\!\! \pmod{m_s}, \, 1 \leq l \leq h_t-1\}.
\end{equation*}

\begin{remark}
Recall that the Howe factorization of $\chi \in \cT_{n,h}$ gives rise to two sequences of integers exactly of the form \eqref{e:m filtration} and \eqref{e:h filtration}, where \eqref{e:m filtration} corresponds to the conductors appearing in the Howe factorization and \eqref{e:h filtration} corresponds to the levels in which the conductor jumps. \hfill $\Diamond$ 
\end{remark}

\begin{lemma}\label{l:IJ}
There is an order-reversing injection $\cI_{s,t} \hookrightarrow \cJ_{s,t}$ that is a bijection if and only if $\#\cA_{s,t}$ is even. Explicitly, it is given by
\begin{equation*}
\cI_{s,t} \hookrightarrow \cJ_{s,t}, \qquad (1, \sigma^j(1), l) \mapsto (1, \sigma^{n-j}(1), h_t - l),
\end{equation*}
where $\sigma$ is the permutation defined in Definition \ref{d:tau}. Note that $\#\cA_{s,t}$ is even unless $n$ and $h_t$ are both even.
\end{lemma}

\begin{proof}
It is clear that if $m_s \mid j$ and $m_{s+1} \nmid j$, then $m_s \mid (n-j)$ and $m_{s+1} \nmid (n-j)$. By assumption, $|(1, \sigma^j(1), l)| = \sigma^j(1)-1+n(l-1) > n(h_t-1)/2$. Then
\begin{equation*}
|(1, \sigma^{n-j}(1), h_t-l)| = (n-\sigma^j(1)+1)+n(h_t-l-1) = n(h_t - 1) - |(1,\sigma^j(1),l)| < n(h_t-1)/2.
\end{equation*}
It is clear that the map is a bijection if and only if $\cJ_{s,t}$ contains an element of norm $n(h_t-1)/2$.
\end{proof}

\begin{remark}
Note that the divisibility condition on $\cJ_{s,t}$ implies that there is at most one $s$ such that $(1,\sigma^{n/2}(1),h_t/2) \in \cJ_{s,t}$. Hence outside this $s$, the sets $\cI_{s,t}$ and $\cJ_{s,t}$ are in bijection. \hfill $\Diamond$
\end{remark}

\begin{definition}
For $\lambda = (i, j, l) \in \cA^+$, define 
\begin{equation*}
\lambda^\vee \colonequals \begin{cases}
(j,i,h_t - 1 - l) & \text{if $i = j$.} \\
(j,i,h_t - l) & \text{if $i \neq j$.}
\end{cases}
\end{equation*}
\end{definition}

\begin{lemma}\label{l:det contribution}
Write $A = (A_{i,j})_{1 \leq i,j \leq n} \in \Unip$, where $A_{i,j}$ is as in Equation \eqref{e:Aij}. Assume that for $\lambda_1, \lambda_2 \in \cA^+$, the variables $A_{\lambda_1}$ and $A_{\lambda_2}$ appear in the same monomial in $\det(A) \in \bW_{h'}^{(1)}$ for some $h' \leq h$. 
\begin{enumerate}[label=(\alph*)]
\item
Then $|\lambda_1| + |\lambda_2| \leq n(h' - 1)$.

\item
If $|\lambda_1| + |\lambda_2| = n(h'-1)$, then $\lambda_2 = \lambda_1^\vee$.
\end{enumerate}
\end{lemma}

\begin{proof}
By definition,
\begin{equation*}
\det(A) = \sum_{\gamma \in S_n} \prod_{1 \leq i \leq n} A_{i,\gamma(i)} \in \bW_{h'}^{(1)}(\overline \FF_q).
\end{equation*}
Let $l \leq h'-1$. If $K$ has characteristic $p$, then the contributions to the $\pi^l$-coefficient coming from $\gamma \in S_n$ are of the form
\begin{equation*}
\prod_{i=1}^n A_{(i, \gamma(i), l_i^*)},
\end{equation*} 
where $(l_1, \ldots, l_n)$ is a partition of $l$. Then
\begin{align*}
|(i, \gamma(i), l_i^*)| &= [\gamma(i) - i] + n(l_i^* - 1)
= \gamma(i) - i + nl_i,
\end{align*}
and therefore
\begin{equation}\label{e:char p ineq}
\sum_{i=1}^n |(i,\gamma(i), l_i^*)| = \sum_{i=1}^n \gamma(i) - i + nl_i = \sum_{i=1}^n n l_i = nl \leq n(h' - 1).
\end{equation}
If $K$ has characteristic $0$, then by Lemma \ref{l:simple Witt n}, the major contributions to the $\pi^l$-coefficient coming from $\gamma$ are of the form
\begin{equation*}
\prod_{i=1}^n A_{(i, \gamma(i), l_i^*)}^{e_i},
\end{equation*} 
where the $e_i$ are some nonnegative integers and where $(l_1, \ldots, l_n)$ is a partition of $l$. Hence 
\begin{equation}\label{e:char 0 ineq}
\sum_{i=1}^n |(i, \gamma(i), l_i^*)| = nl \leq n(h' - 1).
\end{equation} 
The minor contributions to the $\pi^l$-coefficient coming from $\gamma$ are polynomials in $\prod_{i=1}^n A_{i,\gamma(i), l_i^*}^{e_i'}$ where $l_1 + \cdots + l_n < l$ and the $e_i'$ are some nonnegative integers. Hence $\sum_{i=1}^n |(i,\gamma(i),l_i^*)| < n(h' - 1)$.

Suppose now that $\lambda_1 = (i_1, j_1, l_1), \, \lambda_2 = (i_2, j_2, l_2) \in \cA^+$ are such that $A_{\lambda_1}$ and $A_{\lambda_2}$ contribute to the same monomial in $\det(M) \in \bW_{h'}^{(1)}$. Then there exists some $\gamma \in S_n$ such that $\gamma(i_1) = j_1$ and $\gamma(i_2) = j_2$, and by Equations \eqref{e:char p ineq} and \eqref{e:char 0 ineq},
\begin{equation*}
|\lambda_1| + |\lambda_2| \leq n(h' - 1).
\end{equation*}
Observe that if $K$ has characteristic $0$ and $\lambda_1$ and $\lambda_2$ occur in a minor contribution, then $|\lambda_1| + |\lambda_2| < n(h' - 1)$. This proves (a) and furthermore, we see that if $|\lambda_1| + |\lambda_2| = n(h' - 1)$, then the simultaneous contribution of $A_{\lambda_1}$ and $A_{\lambda_2}$ comes from a major contribution. But now (b) follows: since $|\lambda| > 0$ for any $\lambda \in \cA$, it follows from Equations \eqref{e:char p ineq} and \eqref{e:char 0 ineq} that if $|\lambda_1| + |\lambda_2| = n(h' - 1)$, then necessarily the associated $\gamma \in S_n$ must be a transposition.
\end{proof}

\begin{example}\label{ex:indexing sets}
We illustrate a way to visualize the indexing sets $\cA_{\geq s,t}$ in a small example. Consider the sequences
\begin{equation*}
\{m_0, m_1, m_2, m_3, m_4\} = \{1, 2, 4, 8, 8\}, \qquad \{h_0, h_1, h_2, h_3, h_4\}.
\end{equation*}
This corresponds to a character of the form
\begin{equation*}
\chi = \chi_1^0(\Nm_{\FF_{q^8}/\FF_{q^2}}) \cdot \chi_2^0(\Nm_{\FF_{q^8}/\FF_{q^4}}) \cdot \chi_3^0,
\end{equation*}
where $\chi_i^0 \from \bW_{h_t}^{(1)}(\FF_{q^{m_t}}) \to \overline \QQ_\ell^\times$ is primitive. Then we have:
\begin{align*}
\cA_{\geq 0,t} &= \left(\left(\begin{smallmatrix}
1 & * & * & * & * & * & * & * \\
* & 1 & * & * & * & * & * & * \\
* & * & 1 & * & * & * & * & * \\
* & * & * & 1 & * & * & * & * \\
* & * & * & * & 1 & * & * & * \\
* & * & * & * & * & 1 & * & * \\
* & * & * & * & * & * & 1 & * \\
* & * & * & * & * & * & * & 1 
\end{smallmatrix}\right), \text{ depth $h_t$}\right) &
\cA_{\geq 1,t} &= \left(\left(\begin{smallmatrix}
1 &  & * &  & * &  & * &  \\
 & 1 &  & * &  & * &  & * \\
* &  & 1 &  & * &  & * &  \\
 & * &  & 1 &  & * &  & * \\
* &  & * &  & 1 &  & * &  \\
 & * &  & * &  & 1 &  & * \\
* &  & * &  & * &  & 1 &  \\
 & * &  & * &  & * &  & 1 
\end{smallmatrix}\right), \text{ depth $h_t$}\right) \\
\cA_{\geq 2,t} &= \left(\left(\begin{smallmatrix}
1 &  &  &  & * &  &  &  \\
 & 1 &  &  &  & * &  &  \\
 &  & 1 &  &  &  & * &  \\
 &  &  & 1 &  &  &  & * \\
* &  &  &  & 1 &  &  &  \\
 & * &  &  &  & 1 &  &  \\
 &  & * &  &  &  & 1 &  \\
 &  &  & * &  &  &  & 1 
\end{smallmatrix}\right), \text{ depth $h_t$}\right) &
\cA_{\geq 3,t} &= \left(\left(\begin{smallmatrix}
1 &  &  &  &  &  &  &  \\
 & 1 &  &  &  &  &  &  \\
 &  & 1 &  &  &  &  &  \\
 &  &  & 1 &  &  &  &  \\
 &  &  &  & 1 &  &  &  \\
 &  &  &  &  & 1 &  &  \\
 &  &  &  &  &  & 1 &  \\
 &  &  &  &  &  &  & 1 
\end{smallmatrix}\right), \text{ depth $h_t$}\right)
\end{align*}
In Section \ref{s:morphisms}, we will calculate certain cohomology groups $H_c^i(\fX, \cF)$ inductively: 
\begin{align*}
\fX \longleftrightarrow \cA = \cA_{\geq 0,0} &\rightsquigarrow \cA_{\geq 0,1} && \text{Proposition \ref{p:factor AS}} \\
&\rightsquigarrow \cA_{\geq 1,1} && \text{Proposition \ref{p:coprime induct}} \\
&\rightsquigarrow \cA_{\geq 1,2} && \text{Proposition \ref{p:factor AS}} \\
&\rightsquigarrow \cA_{\geq 2,2} && \text{Proposition \ref{p:coprime induct}} \\
&\rightsquigarrow \cA_{\geq 2,3} && \text{Proposition \ref{p:factor AS}} \\
&\rightsquigarrow \cA_{\geq 3,3} = \varnothing \longleftrightarrow \{*\} && \text{Proposition \ref{p:coprime induct}}
\end{align*}
\end{example}

\subsection{Dimension of $X_h$}\label{s:dim Xh}

\begin{proposition}\label{p:dim Xh}
$X_h$ is a smooth affine scheme of dimension $(n-1)(h-1)$.
\end{proposition}

\begin{proof}
By Definition \ref{d:M char}, $X_h \subset \bA[\cA^+]$ is defined by the following polynomials:
\begin{align*}
f_{(i,\sigma^j(i),l)} &\colonequals M_{(i,\sigma^j(i),l)} - M_{(1,\sigma^j(1),l)}^{q^{\tau(i)}}, && \text{for $2 \leq i \leq n$, $1 \leq j \leq n,$ $1 \leq l \leq h-1$}, \\
g_l &\colonequals \pr_l(\det(M)) - \pr_l(\det(M))^q, && \text{for $1 \leq l \leq h-1$}, 
\end{align*}
where
\begin{equation*}
\pr_l \from \bW_h^{(1)} \to \bA^1, \qquad [1, a_1, \ldots, a_{h-1}] \mapsto a_l.
\end{equation*}
This gives us $n(n-1)(h-1) + (h-1)$ equations and $\# \cA^+ = n^2(h-1)$, and so to prove the proposition, it suffices to find a submatrix of size $(n^2 - n + 1)(h-1) \times (n^2-n+1)(h-1)$ of the Jacobian matrix that is nonsingular for every point of $X_h$. 

Consider the submatrix $J_0$ of the Jacobian corresponding to the partial derivatives with respect to the following subset of $\cA^+$:
\begin{equation*}
A_0 \colonequals \{(i, \sigma^j(i), l) \in \cA^+ : i \neq 1\} \cup \{(1, 1, l) \in \cA^+\}. 
\end{equation*}
Since we are working in characteristic $p$, we have
\begin{align} \label{e:f partial}
\frac{\partial f_{\lambda}}{\partial M_{\lambda'}} &= \begin{cases}
1 & \text{if $\lambda = \lambda'$,} \\
0 & \text{otherwise.}
\end{cases} \\ \label{e:g partial}
\frac{\partial g_l}{\partial M_{(1,1,l')}} &= \begin{cases}
1 & \text{if $l' = l$,} \\
0 & \text{if $l' > l$,} \\
? & \text{if $l' < l$.}
\end{cases}
\end{align}
Reorder the rows of $J_0$ so that the first $n(n-1)(h-1)$ rows correspond to the $f_\lambda$ for $\lambda \in A_0$ and the remaining $h-1$ rows correspond to $g_1, \ldots, g_{h-1}$. Reorder the columns of $J_0$ so that the first $n(n-1)(h-1)$ columns correspond to the $M_\lambda$ for $\lambda \in A_0$ and the remaining $h-1$ columns correspond to $M_{(1,1,1)}, \ldots, M_{(1,1,h-1)}$. Then
\begin{equation*}
J_0 = \left(\begin{matrix}
A & B \\ C & D
\end{matrix}\right),
\end{equation*}
where $A$ is a permutation matrix of size $n(n-1)(h-1) \times n(n-1)(h-1)$ (by Equation \eqref{e:f partial}), $B$ is the zero matrix of size $n(n-1)(h-1) \times (h-1)$ (by Equation \eqref{e:f partial}), and D is a unipotent lower-triangular matrix of size $(h-1) \times (h-1)$ (by Equation \eqref{e:g partial}). Hence for any point in $X_h$, the matrices $A$ and $D$ are nonsingular, and so $J_0$ is as well. This shows that $X_h$ is a smooth complete intersection of dimension $n^2(h-1) - (n^2-n+1)(h-1) = (n-1)(h-1)$.
\end{proof}


\section{On the cohomology of certain $\overline {\QQ}_\ell$-local systems}\label{s:LS}

In this section, we prove some general results on the cohomology of constructible $\overline \QQ_\ell$-sheaves. Proposition \ref{p:B2.3} is a generalization of \cite[Proposition 2.3]{B12} that allows one to study the $G(\FF_q)$-representations arising from the cohomology of a variety $X \subset G$ defined over $\FF_{q^N}$. From this perspective, \cite[Proposition 2.3]{B12} is the $N=1$ setting of Proposition \ref{p:B2.3}.

The other main result of this section is Proposition \ref{p:induct}, which is a general result on the cohomology of constructible $\overline \QQ_\ell$-sheaves coming from pullbacks of local systems. Its proof is similar to that of \cite[Proposition 2.10]{B12} and \cite[Proposition 3.4]{C15}. We then apply Proposition \ref{p:induct} to the special case when the local system is an Artin-Schreier-like sheaf (see Proposition \ref{p:inductm}). 

Throughout this paper, we choose the convention that $\Fr_q$ is the geometric Frobenius so that $\Fr_q$ acts on the Tate twist $\overline \QQ_\ell(-1)$ by multiplication by $q$.

\begin{proposition}\label{p:B2.3}
Assume that we are given the following data:
\begin{enumerate}[label=$\cdot$]
\item
an algebraic group $G$ with a connected subgroup $H \subset G$ over $\FF_q$ with Frobenius $F_q$;

\item
a section $s \from G/H \to G$ of the quotient morphism $G \to G/H$;

\item
a character $\chi \from H(\FF_q) \to \overline \QQ_\ell^\times$;

\item
an affine scheme $Y \subset G$ defined over $\FF_{q^N}$ such that for any $0 \leq i \neq j \leq N-1$, the intersection $F_q^i(Y) \cap F_q^j(Y) = S$ for a finite set of points $S$.
\end{enumerate}
Set $X \colonequals L_q^{-1}(Y)$, where $L_q$ is the Lang map $g \mapsto F_q(g)g^{-1}$ on $G$. The right multiplication action of $G(\FF_q)$ on $X$ induces linear representations of $G(\FF_q)$ on the cohomology $H_c^i(X, \overline \QQ_\ell)$, and for each $i \geq 0$, we have a vector space isomorphism
\begin{equation*}
\Hom_{G(\FF_q)}\left(\Ind_{H(\FF_q)}^{G(\FF_q)}(\chi), H_c^i(X, \overline \QQ_\ell)\right) \cong H_c^i(\beta^{-1}(Y), P^* \Loc_\chi)
\end{equation*}
compatible with the action of $\Fr_{q^N}$ on both sides, where $\Loc_\chi$ is the rank 1 local system on $H$ corresponding to $\chi$, the morphism $\beta \from (G/H) \times H \to G$ is given by $\beta(x,h) = s(F_q(x)) \cdot h \cdot s(x)^{-1}$, and the morphism $P \from \beta^{-1}(Y) \to H$ is the composition $\beta^{-1}(Y) \hookrightarrow (G/H) \times H \overset{\pr_2}{\longrightarrow} H$.
\end{proposition}

We will need a couple of lemmas in the proof of Proposition \ref{p:B2.3}.

\begin{lemma}\label{l:smooth curve split}
Let $\cF$ be a pure rank-$1$ lisse sheaf. Let $X$ be a smooth affine curve over $\FF_q$ and let $\overline X$ be a compactification of $X$. Let $j \from X \to \overline X$. Then
\begin{equation*}
H_c^1(X, \cF) \cong H^0(\overline X \smallsetminus X, j_*\cF)/H^0(\overline X, j_*\cF) \oplus H^1(\overline X, j_*\cF) 
\end{equation*}
as $\Fr_q$-vector spaces.
\end{lemma}

\begin{proof}
The open/closed decomposition
\begin{equation*}
j \colon X \hookrightarrow \overline X \hookleftarrow \overline X \smallsetminus X \colon i.
\end{equation*}
implies that we have the following short exact sequence of sheaves on $\overline X$:
\begin{equation*}
0 \to j_! \cF \to \cF \to i_* \cF \to 0.
\end{equation*}
Since $X$ is affine, $H^0(X, \cF) = 0$, so the first part of the induced long exact sequence is
\begin{equation*}
0 \to H^0(\overline X, j_* \cF) \to H^0(\overline X \smallsetminus X, j_*\cF) \to H_c^1(X, \cF) \to H^1(\overline X, j_*\cF) \to 0,
\end{equation*}
which implies 
\begin{equation}\label{e:SES compactification}
0 \to H^0(\overline X \smallsetminus X, j_*\cF)/H^0(\overline X, j_*\cF) \to H_c^1(X, \cF) \to H^1(\overline X, j_*\cF) \to 0.
\end{equation}
We now show that this sequence splits. By Deligne's work on the Weil conjectures \cite[Corollary 3.3.9 and Theorem 3.2.3]{D80}, $H^0(\overline X \smallsetminus X, j_*\cF)/H^0(\overline X, j_*\cF)$ has weight $\wt(\cF)$ and $H^1(\overline X, j_*\cF)$ has weight $\wt(\cF)+1$, and so the extension $H_c^1(X, j_*\cF)$ is classified by an element of $H^1(\widehat \ZZ, V)$, where $V \colonequals H^0(\overline X \smallsetminus X, j_*\cF)/H^0(\overline X, j_*\cF) \otimes H^1(\overline X, j_*\cF)^*$ has weight $-1$. Now,
\begin{equation*}
H^1(\widehat \ZZ, V) = Z^1(\widehat \ZZ, V)/B^1(\widehat \ZZ, V),
\end{equation*}
where
\begin{align*}
Z^1(\widehat \ZZ, V) &= \{f \from \widehat \ZZ \to V \mid \text{$f(gh) = g f(h) + f(g)$ for all $g,h \in G$}\}, \\
B^1(\widehat \ZZ, V) &= \{f \in Z^1(\widehat \ZZ, V) \mid \text{there exists $v \in V$ such that $f(g) = ga = a$ for all $g \in G$}\}. 
\end{align*}
Observe that
\begin{equation*}
Z^1(\widehat \ZZ, V) \cong V, \qquad B^1(\widehat \ZZ, V) \cong (\Fr_q - 1)V.
\end{equation*}
Since $V$ has weight $-1$, this implies that $(\Fr_q - 1)V = V$ and so 
\begin{equation*}
H^1(\widehat \ZZ, V) = 0.
\end{equation*}
Thus the short exact sequence in \eqref{e:SES compactification} splits, and the desired conclusion follows.
\end{proof}

\begin{lemma}\label{l:curve split}
Let $\cF$ be a rank-$1$ $\overline \QQ_\ell$-sheaf. Let $X$ be an affine curve over $\FF_q$ and let $S$ be a finite set of points. Then the short exact sequence
\begin{equation*}
0 \to H_c^0(S, \cF) \to H_c^1(X \smallsetminus S, \cF) \to H_c^1(X, \cF) \to 0
\end{equation*}
of $\Fr_q$-vector spaces splits.
\end{lemma}

\begin{proof}
First assume that $X$ is smooth. Let $\overline X$ be a compactification of $X$. Then it must also be a compactification of $X \smallsetminus S$. Therefore by Lemma \ref{l:smooth curve split}, 
\begin{align*}
H_c^1(X, \cF) &\cong H^0(\overline X \smallsetminus X, j_*\cF)/H^0(\overline X, j_*\cF) \oplus H^1(\overline X, j_*\cF), \\
H_c^1(X \smallsetminus S, \cF) &\cong H^0(S \cup (\overline X \smallsetminus X), j_*\cF)/H^0(\overline X, j_*\cF) \oplus H^1(\overline X, j_*\cF) \\
&\cong H^0(S, \cF) \oplus H^0(\overline X \smallsetminus X, j_* \cF)/H^0(\overline X, j_* \cF) \oplus H^1(\overline X, j_* \cF) \\
&\cong H_c^0(S, \cF) \oplus H_c^1(X, \cF). 
\end{align*}
Thus the desired short exact sequence splits.

Now assume that $X$ is not smooth and let $\widetilde X$ be the normalization of $X$. Then $\widetilde X$ is a smooth affine curve and by the previous paragraph, the upper short exact sequence of $\Fr_q$-vector spaces of the following commutative diagram splits:
\begin{equation*}
\begin{tikzcd}
0 \ar{r}& H_c^0(\widetilde S, \cF) \ar{r}\ar[dotted,bend right=30]{d}& H_c^1(\widetilde X \smallsetminus \widetilde S, \cF) \ar{r}\ar[dotted,bend right=20]{l}& H_c^1(\widetilde X, \cF) \ar{r}& 0 \\
0 \ar{r}& H_c^0(S, \cF) \ar{r}\ar{u}& H_c^1(X \smallsetminus S, \cF) \ar{r}\ar{u}& H_c^1(X, \cF) \ar{r}\ar{u}& 0
\end{tikzcd}
\end{equation*}
Here, where we view $\cF$ as a sheaf on $\widetilde X$ by pulling back along the morphism $\widetilde X \to X$, and this morphism induces the upward vertical maps. Restriction gives the dotted downward arrow. The composition
\begin{equation*}
H_c^1(X \smallsetminus S, \cF) \to H_c^1(\widetilde X \smallsetminus \widetilde S, \cF) \to H_c^0(\widetilde S, \cF) \to H_c^0(S, \cF)
\end{equation*}
gives a splitting of $\Fr_q$-vector spaces of the lower short exact sequence.
\end{proof}

We are now ready to prove Proposition \ref{p:B2.3}.

\begin{proof}
For convenience, set $V_\chi \colonequals \Ind_{H(\FF_q)}^{G(\FF_q)}(\chi)$. Consider the scheme
\begin{equation*}
Y^+ \colonequals Y \cup F_q(Y) \cup \cdots \cup F_q^{N-1}(Y) \subset G
\end{equation*}
and note that it is locally closed and defined over $\FF_q$. (Note that $Y^+ \cong \Res_{\FF_{q^n}/\FF_q}(Y)$. We work with $Y^+$ as it is a subscheme of $G$ by construction.) By \cite[Proposition 2.3]{B12}, we have $\Fr_q$-compatible vector-space isomorphisms
\begin{equation}\label{e:hom to Y+}
\Hom_{G(\FF_q)}\left(V_\chi, H_c^i(X^+, \overline \QQ_\ell)\right) \cong H_c^i(\beta^{-1}(Y^+), (P^+)^* \Loc_\chi) \qquad \text{for all $i \geq 0$,}
\end{equation}
where $X^+ = L_q^{-1}(Y^+)$ and $P^+$ is the composition $\beta^{-1}(Y^+) \hookrightarrow (G/H) \times H \to H.$ We will use open/closed decompositions and the associated long exact sequences to relate the cohomology of $X^+$ to the cohomology of $X$, and the cohomology of $\beta^{-1}(Y^+)$ to the cohomology of $\beta^{-1}(Y)$.

Let $Y'{}^+ \colonequals Y^+ \smallsetminus S$. The open/closed decomposition
\begin{equation*}
\beta^{-1}(Y'{}^+) \stackrel{j}{\hookrightarrow} \beta^{-1}(Y^+) \stackrel{i}{\hookleftarrow} \beta^{-1}(S)
\end{equation*}
induces a short exact sequence of sheaves
\begin{equation*}
0 \to j_!(P^+)^* \Loc_\chi \to (P^+)^* \Loc_\chi \to i_*(P^+)^* \Loc_\chi \to 0
\end{equation*}
whose associated long exact sequence in cohomology gives rise to $\Fr_q$-compatible isomorphisms
\begin{equation}\label{e:Y+ iso}
H_c^i(\beta^{-1}(Y'{}^+), \cF^+) \cong H_c^i(\beta^{-1}(Y^+), \cF^+), \qquad \text{for $i \geq 2$}, 
\end{equation}
and an exact sequence
\begin{align} \nonumber
0 &\to H_c^0(\beta^{-1}(Y'{}^+), \cF^+) \to H_c^0(\beta^{-1}(Y^+), \cF^+) \to H_c^0(\beta^{-1}(S), \cF^+) \to \\ \label{e:Y+ seq}
&\to H_c^1(\beta^{-1}(Y'{}^+), \cF^+) \to H_c^1(\beta^{-1}(Y^+), \cF^+) \to H_c^1(\beta^{-1}(S), \cF^+) = 0.
\end{align}
Here, we write $\cF^+ \colonequals (P^+)^* \Loc_\chi$ for convenience and abuse notation by writing $\cF^+$ for $(P^+|_{\beta^{-1}(Y'{}^+)})^* \Loc_\chi$ and $(P^+|_{\beta^{-1}(S)})^* \Loc_\chi$ as well. The same argument applied to the $\FF_{q^N}$-scheme $Y' \colonequals Y \smallsetminus S$ shows that we have $\Fr_{q^N}$-compatible isomorphisms
\begin{equation}\label{e:Y iso}
H_c^i(\beta^{-1}(Y'), P^* \Loc_\chi) \cong H_c^i(\beta^{-1}(Y), P^* \Loc_\chi), \qquad \text{for $i \geq 2$,}
\end{equation}
and an exact sequence
\begin{align} \nonumber
0 &\to H_c^0(\beta^{-1}(Y'), \cF) \to H_c^0(\beta^{-1}(Y), \cF) \to H_c^0(\beta^{-1}(S), \cF) \to \\ \label{e:Y seq}
&\to H_c^1(\beta^{-1}(Y'), \cF) \to H_c^1(\beta^{-1}(Y), \cF) \to H_c^1(\beta^{-1}(S), \cF) = 0,
\end{align}
where $P = P^+|_{\beta^{-1}(Y)}$ and we write $\cF \colonequals P^* \Loc_\chi$ and abuse notation as before.

Applying the same argument to $X'{}^+ \colonequals X^+ \smallsetminus L_q^{-1}(S)$ and $X' \colonequals X \smallsetminus L_q^{-1}(S)$ together with the constant sheaf $\overline \QQ_\ell$ gives $\Fr_q$-compatible isomorphisms
\begin{align}
H_c^i(X'{}^+, \overline \QQ_\ell) &\cong H_c^i(X^+, \overline \QQ_\ell), \label{e:X+ iso}
\end{align}
exact sequences
\begin{align}\nonumber
0 &\to H_c^0(X'{}^+, \overline \QQ_\ell) \to H_c^0(X^+, \overline \QQ_\ell) \to H_c^0(L_q^{-1}(S), \overline \QQ_\ell) \to \\ \label{e:X+ seq}
&\to H_c^1(X'{}^+, \overline \QQ_\ell) \to H_c^1(X^+, \overline \QQ_\ell) \to H_c^1(L_q^{-1}(S), \overline \QQ_\ell) = 0,
\end{align}
$\Fr_{q^N}$-compatible isomorphisms
\begin{align}
H_c^i(X', \overline \QQ_\ell) &\cong H_c^i(X, \overline \QQ_\ell), \qquad \text{for $i \geq 2$}, \label{e:X iso}
\end{align}
and exact sequences
\begin{align}\nonumber
0 &\to H_c^0(X', \overline \QQ_\ell) \to H_c^0(X, \overline \QQ_\ell) \to H_c^0(L_q^{-1}(S), \overline \QQ_\ell) \to \\ \label{e:X seq}
&\to H_c^1(X', \overline \QQ_\ell) \to H_c^1(X, \overline \QQ_\ell) \to H_c^1(L_q^{-1}(S), \overline \QQ_\ell) = 0.
\end{align}

By construction, $F^i(\beta^{-1}(Y')) \cap F^j(\beta^{-1}(Y')) = \varnothing$ for all $0 \leq i \neq j \leq N-1$, and hence we have $\Fr_{q^N}$-compatible isomorphisms
\begin{equation*}
H_c^i(\beta^{-1}(Y'{}^+), \cF^+) \cong \bigoplus_{j=0}^{N-1} H_c^i(F_q^j(\beta^{-1}(Y')), \cF^+|_{F_q^j(\beta^{-1}(Y_h'))})
\end{equation*}
for all $i \geq 0$. For any $0 \leq j \leq N-1$, we have a commutative diagram
\begin{equation*}
\begin{tikzcd}
\beta^{-1}(Y') \ar[hookrightarrow]{r} \ar{d}{F_q^j} \ar[bend left]{rr}{P} & \beta^{-1}(Y'{}^+) \ar{r}{P^+} \ar{d}{F_q^j} & H \ar{d}{F_q^j} \\
F_q^j(\beta^{-1}(Y')) \ar[hookrightarrow]{r} & \beta^{-1}(Y'{}^+) \ar{r}{P^+} & H
\end{tikzcd}
\end{equation*}
and so it follows that $\cF^+|_{F_q^j(\beta^{-1}(Y'))} = (F_q^j)_* \cF$ and
\begin{equation*}
H_c^i(\beta^{-1}(Y'), \cF) = H_c^i(F_q^j(\beta^{-1}(Y')), \cF^+|_{F_q^j(\beta^{-1}(Y'))}).
\end{equation*}
Thus
\begin{equation*}
H_c^i(\beta^{-1}(Y'{}^+), \cF^+) \cong H_c^i(\beta^{-1}(Y'), \cF)^{\oplus N} \qquad \text{for all $i \geq 0$}.
\end{equation*}

We first prove the proposition for $i \geq 2$. By Equations \eqref{e:Y+ iso} and \eqref{e:Y iso}, for all $i \geq 2$, we have
\begin{equation*}
H_c^i(\beta^{-1}(Y^+), \cF^+) \cong H_c^i(\beta^{-1}(Y'{}^+), \cF^+) \cong H_c^i(\beta^{-1}(Y'), \cF)^{\oplus N} \cong H_c^i(\beta^{-1}(Y), \cF)^{\oplus N}.
\end{equation*}
Similarly, by Equations \eqref{e:X+ iso} and \eqref{e:X iso}, for all $i \geq 2$, we have
\begin{equation*}
H_c^i(X^+, \overline \QQ_\ell) \cong H_c^i(X'{}^+, \overline \QQ_\ell) \cong H_c^i(X', \overline \QQ_\ell)^{\oplus N} \cong H_c^i(X, \overline \QQ_\ell)^{\oplus N}.
\end{equation*}
Combining this with Equation \eqref{e:hom to Y+}, we have $\Fr_{q^N}$-compatible vector-space isomorphisms
\begin{align*}
\Hom_{G(\FF_q)}\left(V_\chi, H_c^i(X, \overline \QQ_\ell)\right)^{\oplus N}
&\cong \Hom_{G(\FF_q)}\left(V_\chi, H_c^i(X^+, \overline \QQ_\ell)\right)  \\
&\cong H_c^i(\beta^{-1}(Y^+), P^* \Loc_\chi) \\
&\cong H_c^i(\beta^{-1}(Y), P^* \Loc_\chi)^{\oplus N} && \text{for all $i \geq 2$.}
\end{align*}

It remains to prove the proposition for $i = 0, 1$. Since $X$ is an affine scheme, we see that
\begin{equation*}
H_c^0(X, \overline \QQ_\ell) \neq 0 \quad \Rightarrow \quad \dim X = 0, \qquad \text{and} \qquad
H_c^1(X, \overline \QQ_\ell) \neq 0 \quad \Rightarrow \quad \dim X = 1. 
\end{equation*}
First observe that $S$ is defined over $\FF_q$ and by \cite[Proposition 2.3]{B12}, we have
\begin{equation}\label{e:S}
\Hom_{G(\FF_q)}(V_\chi, H_c^0(L_q^{-1}(S), \overline \QQ_\ell)) \cong H_c^0(\beta^{-1}(S), \cF)
\end{equation}
as $\Fr_q$-vector spaces. We will use this to complete the proof.

Assume $\dim X = 0$. By Equations \eqref{e:Y+ seq}, \eqref{e:Y seq}, \eqref{e:X+ seq}, and \eqref{e:X seq}, we have
\begin{align*}
H_c^0(\beta^{-1}(Y^+), \cF^+) &\cong H_c^0(\beta^{-1}(Y'{}^+), \cF^+) \oplus H_c^0(\beta^{-1}(S), \cF^+), 
\end{align*}
and similarly for $\beta^{-1}(Y)$, $X^+$, and $X$. Thus we have
\begin{align*}
\Hom_{G(\FF_q)}(V_\chi, H_c^0(X', \overline \QQ_\ell))^{\oplus N} 
&\oplus \Hom_{G(\FF_q)}(V_\chi, H_c^0(L_q^{-1}(S), \overline \QQ_\ell)) \\
&\cong \Hom_{G(\FF_q)}(V_\chi, H_c^0(X^+, \overline \QQ_\ell)) \\
&\cong H_c^0(\beta^{-1}(Y^+), \cF^+) \\
&\cong H_c^0(\beta^{-1}(Y'), \cF)^{\oplus N} \oplus H_c^0(\beta^{-1}(S), \cF).
\end{align*}
The last summand of the first and last lines are isomorphic by Equation \eqref{e:S}. Therefore as $\Fr_{q^N}$-vector spaces,
\begin{equation*}
\Hom_{G(\FF_q)}(V_\chi, H_c^0(X', \overline \QQ_\ell)) \cong H_c^0(\beta^{-1}(Y'), \cF),
\end{equation*}
and the desired conclusion now follows.

Finally, assume that $\dim X = 1$. Since $X$ is an affine variety, Equations \eqref{e:Y+ seq}, \eqref{e:Y seq}, \eqref{e:X+ seq}, and \eqref{e:X seq} reduce to short exact sequences, and by Lemma \ref{l:curve split},
\begin{equation*}
H_c^1(\beta^{-1}(Y'{}^+), \cF^+) \cong H_c^1(\beta^{-1}(Y^+), \cF^+) \oplus H_c^0(\beta^{-1}(S), \cF),
\end{equation*}
and similarly for $\beta^{-1}(Y)$, $X^+$, and $X$. Using Equation \eqref{e:S}, we have $\Fr_{q^N}$-vector space isomorphisms
\begin{align*}
\Hom_{G(\FF_q)}\left(V_\chi, H_c^1(X, \overline \QQ_\ell)\right)^{\oplus N} 
&\oplus H_c^0(\beta^{-1}(S), \cF)^{\oplus N} \\
&\cong \Hom_{G(\FF_q)}\left(V_\chi, H_c^1(X', \overline \QQ_\ell)\right)^{\oplus N} \\
&\cong \Hom_{G(\FF_q)}\left(V_\chi, H_c^1(X'{}^+, \overline \QQ_\ell)\right) \\
&\cong H_c^1(\beta^{-1}(Y'{}^+), \cF^+) \\
&\cong H_c^1(\beta^{-1}(Y'), \cF)^{\oplus N} \\
&\cong H_c^1(\beta^{-1}(Y), \cF)^{\oplus N} \oplus H_c^0(\beta^{-1}(S), \cF)^{\oplus N},
\end{align*}
and this completes the proof of the proposition.
\end{proof}

\begin{proposition}\label{p:induct}
Let $G$, $G'$ be algebraic groups over $\FF_q$, let $\cF$ be a multiplicative local system on $G$, and consider a filtration of finite-type $\FF_q$-schemes $S_3 \subset S_2 \subset S$. Assume that $H_c^i(G', \overline \QQ_\ell)$ is concentrated in a single degree $i = r$ and assume $S = S_2 \times G'$. Let $\eta \from S_2 \times G' \to G$ and $P_2 \from S_2 \to G$ be any morphisms, set $P_3 \colonequals P_2|_{S_3}$, and define
\begin{align*}
P &\from S = S_2 \times G' \to G, & (x,g) &\mapsto \eta(x,g) \cdot P_2(x), \\
\eta_x &\from G' \to G, & g &\mapsto \eta(x,g), && \text{for $x \in S_2$}.
\end{align*}
If $\eta$ has the property that $\eta^* \cF|_{S_3 \times G'}$ is the constant local system and $\eta_x^* \cF$ is a nontrivial multiplicative local system on $G'$ for every $x \in S_2(\overline \FF_q) \smallsetminus S_3(\overline \FF_q)$, then for all $i \in \bZ$,
\begin{equation*}
H_c^i(S, P^* \cF) \cong H_c^{i-r}(S_3, P_3^* \cF \otimes \underline{H_c^r(G', \overline \QQ_\ell}))
\end{equation*}
as $\Fr_q$-vector spaces, where $\underline{H_c^r(G', \overline \QQ_\ell)}$ is a constant sheaf.
\end{proposition}

\begin{proof}
Consider the following commutative diagram, where $(*)$ and $(**)$ are Cartesian squares and $x$ is any point in $S_3(\overline \FF_q)$.
\begin{equation*}
\begin{tikzcd}[row sep=small,column sep=small]
&&&& S \ar[transform canvas={xshift=0.3ex},-]{d} \ar[transform canvas={xshift=-0.4ex},-]{d} &&  \\
G' \ar{rr}{f'} \ar{dd}{g'} \ar[draw=none]{dr}[very near start,description]{\textstyle\lrcorner} && S_3 \times G' \ar{rr}{f} \ar{dd}{g} \ar[draw=none]{dr}[very near start,description]{\textstyle\lrcorner} && S_2 \times G' \ar{dd}{\pr} \ar{rr}{\eta} && G \ar{drr}{(-,1)} \\
& (**) & & (*) & & & & & G \times G \ar{rr}{m} && G \\
\{x\} \ar{rr}{i_x} && S_3 \ar{rr}{\iota} \ar[bend right]{rrrr}[below]{P_3} && S_2 \ar{rr}{P_2} && G \ar{urr}[below right]{(1,-)}
\end{tikzcd}
\end{equation*}
By construction,
\begin{equation*}
P^* \cF \cong (\eta^* \cF) \otimes \pr^*(P_2^* \cF),
\end{equation*}
hence by the projection formula,
\begin{equation*}
R \pr_!(P^* \cF) \cong P_2^* \cF \otimes R \pr_!(\eta^* \cF) \qquad \text{in $D_c^b(S_2, \overline \QQ_\ell).$}
\end{equation*}
Since $\eta^* \cF|_{S_3 \times G'} = \overline \QQ_\ell$ by assumption, the proper base change theorem applied to $(*)$ implies that
\begin{equation*}
\iota^* R\pr_!(\eta^* \cF) \cong Rg_! f^*(\eta^* \cF) = Rg_! (\overline \QQ_\ell).
\end{equation*}
For any $x \in S_3(\overline \FF_q)$, the proper base change theorem applied to $(**)$ implies that
\begin{equation*}
(R^i g_! \overline \QQ_\ell)_x = i_x^* (R^i g_! \overline \QQ_\ell) = R^i g_!' f'{}^* \overline \QQ_\ell = R^i g_!' \overline \QQ_\ell = H_c^i(G', \overline \QQ_\ell).
\end{equation*}
It therefore follows that
\begin{equation*}
\iota^* R^i \pr_!(\eta^* \cF) \cong \underline{H_c^i(G', \overline \QQ_\ell)},
\end{equation*}
where the righthand side is a constant sheaf on $S_3$.

We now show that $R^i \pr_!(\eta^* \cF)$ is supported on $S_3$. Now let $x \in S_2(\overline \FF_q) \smallsetminus S_3(\overline \FF_q)$. Then by the proper base change theorem applied to the Cartesian square
\begin{equation*}
\begin{tikzcd}
G' \ar{r}[above]{f''} \ar{d}[left]{g''} \ar[draw=none]{dr}[very near start,description]{\textstyle\lrcorner} & S_2 \times G' \ar{d}[right]{\pr} \\
\{x\} \ar{r}[below]{i_x} & S_2
\end{tikzcd}
\end{equation*}
we have
\begin{equation*}
i_x^* Rg_! (\eta^* \cF) \cong Rg_!'' f''{}^*(\eta^* \cF) \cong Rg_!'' (\eta_x^* \cF).
\end{equation*}
On the other hand, by \cite[Lemma 9.4]{B10}, we have
\begin{equation*}
H_c^i(G', \eta_x^* \cF) = 0 \qquad \text{for all $i \geq 0$}.
\end{equation*}
Therefore
\begin{equation*}
(R^i g_!(\eta^* \cF))_x \cong R^i g_!''(\eta_x^* \cF) = H_c^i(G', \eta_x^* \cF) = 0, \qquad \text{for all $i \geq 0$.}
\end{equation*}

Combining all statements above, we obtain
\begin{equation*}
R^i \pr_!(P^* \cF) \cong P_3^* \cF \otimes \underline{H_c^i(G', \overline \QQ_\ell)}. \qedhere
\end{equation*}
\end{proof}

In this paper, the most important application of Proposition \ref{p:induct} is Proposition \ref{p:inductm}, which we will use repeatedly to prove Theorem \ref{t:hom}.

\begin{proposition}\label{p:inductm}
Let $\chi \from \bW_h^{(1)}(\F) \to \overline \QQ_\ell^\times$ be a character of conductor $m$ and let $\Loc_\chi$ be the corresponding local system on $\bW_h^{(1)}$. Let $S_2$ be a finite-type scheme over $\F$, put $S = S_2 \times \Affine^1$ and suppose that a morphism $P \from S \to \TUnip$ has the form
\begin{equation*}
P(x,y) = g(y^{q^{j+\tau}}(f(x)^{q^n} - f(x))^{q^\sigma} + f(x)^{q^{n-j+\sigma}}(y^{q^n} - y)^{q^\tau}) \cdot P_2(x)
\end{equation*}
where
\begin{enumerate}[label=$\cdot$]
\item
$j, \sigma, \tau \geq 0$ are integers and $m$ does not divide $j$,

\item
$f \from S_2 \to \bA^1$ and $P_2 \from S_2 \to \bW_h^{(1)}$ are morphisms defined over $\F$, and

\item
$g \from \bA^1 \to \bW_h^{(1)}$ is the morphism $z \mapsto (1,0,\ldots, 0,z)$.
\end{enumerate}
Let $S_3 \subset S_2$ be the subscheme defined by $f(x)^{q^n} - f(x) = 0$ and let $P_3 = P_2|_{S_3} \from S_3 \to Z$. Then for all $i \in \bZ$, we have
\begin{equation*}
H_c^i(S, P^* \Loc_\chi) \cong H_c^{i-2}(S_3, P_3^* \Loc_\chi)(-1)
\end{equation*}
as vector spaces equipped with an action of $\Fr_{q^n}$.
\end{proposition}

\begin{proof}
It is clear that the conclusion holds if we can apply Proposition \ref{p:induct} to the situation when:
\begin{enumerate}[label=$\cdot$]
\item
$G = \bW_h^{(1)}$ and $G' = \Affine^1 = \GG_a$

\item
$\cF = \Loc_\chi$

\item
$\eta \from S_2 \times \bA^1 \to H$ is $(x,y) \mapsto g(y^{q^{j+\tau}}(f(x)^{q^n} - f(x))^{q^\sigma} + f(x)^{q^{n-j+\sigma}}(y^{q^n} - y)^{q^\tau})$
\end{enumerate}
To this end, we must verify that:
\begin{enumerate}[label=(\alph*)]
\item
$\eta^* \Loc_\chi|_{S_3 \times \Affine^1}$ is the constant local system

\item
For any $x \in S_2(\overline \FF_q) \smallsetminus S_3(\overline \FF_q)$, the sheaf $\eta_x^* \Loc_\chi$ is a nontrivial local system on $\bW_h^{(1)}$
\end{enumerate}
(a) is clear since $S_3 = \bV(f^{q^n} - f)$ by definition and $\chi$ is a character of $\bW_h^{(1)}(\F)$. To see (b), we use the approach of Section 6.4 of \cite{B12}. First, note that $\eta = g \circ \eta_0$ where
\begin{equation*}
\eta_0 \from S_2 \times \bA^1 \to \bA^1 = \bG_a, \qquad (x,y) \mapsto y^{q^{j+\tau}}(f(x)^{q^n} - f(x))^{q^\sigma} + f(x)^{q^{n-j+\sigma}}(y^{q^n} - y)^{q^\tau},\end{equation*}
and hence
\begin{equation*}
\eta^* \Loc_\chi \cong \eta_0^* \Loc_\psi, \qquad \text{where $\psi = \chi|_{\{(1,0,\ldots,0,*) \in \bW_h^{(1)}(\F)\}}$}.
\end{equation*}
Note that $\Loc_\psi$ is a multiplicative local system on $\GG_a$. Fix an auxiliary nontrivial additive character $\psi_0 \from \FF_p \to \overline \QQ_\ell^\times$, and for any $z \in \overline \FF_p$, define
\begin{equation*}
\Loc_z \colonequals m_z^* \Loc_{\psi_0}, \qquad \text{where $m_z \from \bG_a \to \bG_a$ is the map $x \mapsto xz$.}
\end{equation*}
Then there exists a unique $z \in \overline \FF_q$ such that $\Loc_\psi = \Loc_z$. Since $\psi$ has conductor $q^m$, then the stabilizer of $z$ in $\Gal(\overline \FF_q/\FF_q)$ is exactly $\Gal(\overline \FF_q/\FF_{q^m})$, and hence $z \in \FF_{q^m}$. By \cite[Corollary 6.5]{B12}, we have $\eta_x^* \Loc_\psi \cong \Loc_{\eta_x^*(z)}$, where
\begin{align*}
\eta_x^*(z) 
&= f(x)^{q^{n+\sigma-j-\tau}} z^{q^{-j-\tau}} - f(x)^{q^{\sigma-j-\tau}} z^{q^{-j-\tau}} + f(x)^{q^{n-j+\sigma-n-\tau}} z^{q^{-n-\tau}} - f(x)^{q^{n-j+\sigma-\tau}} z^{q^{-\tau}} \\
&= f(x)^{q^{n-j+\sigma-\tau}}(z^{q^{-j-\tau}} - z^{q^{-\tau}}) - f(x)^{q^{-j+\sigma-\tau}}(z^{q^{-j-\tau}} - z^{q^{-n-\tau}}) \\
&= (f(x)^{q^{n-j+\sigma-\tau}} - f(x)^{q^{-j+\sigma-\tau}})(z^{q^{-j-\tau}} - z^{q^{-\tau}}).
\end{align*}
But this is nonzero by assumption, so $\eta_x^* \Loc_\psi$ is a nontrivial local system on $\bG_a$. 
\end{proof}


\section{Morphisms to the cohomology}\label{s:morphisms}

In this section we prove a theorem calculating the space of homomorphisms
\begin{equation*}
\Hom_{\Unip(\FF_q)}\left(\Ind_{\TUnip(\FF_q)}^{\Unip(\FF_q)}(\chi), H_c^i(X_h, \overline \QQ_\ell)\right).
\end{equation*}
This result is crucial to the proofs of many of the theorems in Section \ref{s:DL}. The finale of the proof of Theorem \ref{t:hom} is in Section \ref{s:hom}. Throughout this section, for $x \in \bA[\cA^+]$, we write $x_{(i,j,k)}$ to mean the coordinate of $x$ corresponding to $(i,j,k) \in \cA^+$. Recall from Section \ref{s:howe} that the Howe factorization attaches to any character of $\TUnip(\FF_q) \cong \bW_h^{(1)}(\F)$ two sequences of integers
\begin{align*}
&1 \equalscolon m_0 \leq m_1 < m_2 < \cdots < m_r \leq m_{r+1} \colonequals n \\
&h \equalscolon h_0 = h_1 > h_2 > \cdots > h_r \geq h_{r+1} \colonequals 1
\end{align*}
satisfying the divisibility $m_i \mid m_{i+1}$ for $0 \leq i \leq r$.

%

\begin{theorem}\label{t:hom}
For any character $\chi \from \TUnip(\FF_q) \to \overline \QQ_\ell^\times$,
\begin{equation*}
\Hom_{\Unip(\FF_q)}\left(\Ind_{\TUnip(\FF_q)}^{\Unip(\FF_q)}(\chi), H_c^i(X_h, \overline \QQ_\ell)\right) =
\begin{cases}
\overline \QQ_\ell^{\oplus q^{nd_\chi/2}} \otimes ((-q^{n/2})^{r_\chi})^{\deg} & \text{if $i = r_\chi$}, \\
0 & \text{otherwise,}
\end{cases}
\end{equation*}
where
\begin{align*}
d_\chi &= \sum_{t=1}^{r+1} \left(\frac{n}{m_{t-1}} - \frac{n}{m_t}\right)(h_t - 1), \\
r_\chi &= \sum_{t=1}^{r+1} \left(\left(\frac{n}{m_{t-1}} - \frac{n}{m_t}\right)(h_t-1) + 2\left(\frac{n}{m_{t-1}}-1\right)(h_{t-1} - h_{t})\right).
\end{align*}
Moreover, $\Fr_{q^n}$ acts on $H_c^i(X_h, \overline \QQ_\ell)$ by multiplication by the scalar $(-1)^i q^{ni/2}$.
\end{theorem}

\begin{notation}
Following Katz \cite{K01}, we write $\alpha^{\deg}$ to mean the rank-$1$ $\overline \QQ_\ell$-sheaf with the action of $\Fr_{q^n}$ given by multiplication by $\alpha$. 
\end{notation}

The following corollary is an easy but crucial consequence of Theorem \ref{t:hom}.

\begin{corollary}\label{c:hom}
Let $\pi$ be an irreducible constituent of $H_c^r(X_h, \overline \QQ_\ell)$ for some $r$. Then 
\begin{equation*}
\Hom_{\Unip(\FF_q)}(\pi, H_c^i(X_h, \overline \QQ_\ell)) = 0, \qquad \text{for all $i \neq r$.}
\end{equation*}
\end{corollary}

\begin{proof}
The representation $\pi$ of $\Unip(\FF_q)$ is a constituent of $\Ind_{\TUnip(\FF_q)}^{\Unip(\FF_q)}(\chi)$ for some $\chi$. Hence
\begin{equation*}
\Hom_{\Unip(\FF_q)}\left(\Ind_{\TUnip(\FF_q)}^{\Unip(\FF_q)}(\chi), H_c^r(X_h, \overline \QQ_\ell)\right) \neq 0.
\end{equation*}
By Theorem \ref{t:hom}, it follows that $r = r_\chi$ and 
\begin{equation*}
\Hom_{\Unip(\FF_q)}\left(\Ind_{\TUnip(\FF_q)}^{\Unip(\FF_q)}(\chi), H_c^i(X_h, \overline \QQ_\ell)\right) = 0, \qquad \text{for all $i \neq r_\chi$.}
\end{equation*}
The desired conclusion now holds.
\end{proof}

%

As hinted in Proposition \ref{p:B2.3}, the spaces of $\Unip(\FF_q)$-homomorphisms in Theorem \ref{t:hom} are isomorphic to the cohomology groups of certain (pullbacks of) Artin--Schreier sheaves. To make this precise, first recall that over $\FF_q$, the quotient $\Unip/\TUnip$ can be identified with the affine space $\bA[\cA]$. There is a natural section of the quotient map $\Unip \to \Unip/\TUnip$ given by
\begin{equation*}
s \from (x_{(i,j,l)})_{(i,j,l) \in \cA} \mapsto x_1 + x_2 \varpi^{[k]} + \cdots + x_n \varpi^{[k(n-1)]},
\end{equation*}
where $x_1$ is the $n \times n$ identity matrix and
\begin{equation*}
x_j = \diag(x_{1,j}, x_{2,[j+1]}, \ldots, x_{n,[j+n-1]}), \qquad x_{i,j} = [x_{(i,j,1)}, \ldots, x_{(i,j,h-1)}], \qquad \text{for $j = 2, \ldots, n.$}
\end{equation*}
Let $Y_h \colonequals L_q(X_h)$ and note that since $X_h$ is affine and $L_q$ is  \'etale, $Y_h$ is affine. Define
\begin{equation*}
\beta \from (\Unip/\TUnip) \times \TUnip \to \Unip, \qquad (x,g) \mapsto s(F(x)) \cdot g \cdot s(x)^{-1}.
\end{equation*}
The $\FF_{q^n}$-scheme $\beta^{-1}(Y_h) \subset (\Unip/\TUnip) \times \TUnip$ comes with two maps:
\begin{equation*}
\pr_1 \from \beta^{-1}(Y_h) \to \Unip/\TUnip = \bA[\cA], \qquad \pr_2 \from \beta^{-1}(Y_h) \to \TUnip.
\end{equation*}
We know that the Lang morphism $L_q$ is surjective. For any $y \in \TUnip$ such that $L_q(y) = g$, we have
\begin{equation}\label{e:equivalence}
(x,g) \in \beta^{-1}(Y_h) \qquad \Longleftrightarrow \qquad s(x) \cdot y \in X_h.
\end{equation}

\begin{proposition}\label{p:hom}
For any character $\chi \from \bW_h^{(1)}(\FF_{q^n}) \cong \TUnip(\FF_q) \to \overline \QQ_\ell^\times$, let $\Loc_\chi$ denote the corresponding $\overline \QQ_\ell$-local system on $\bW_h^{(1)}$. Then for $i \geq 0$, we have $\Fr_{q^n}$-compatible isomorphisms
\begin{equation*}
\Hom_{\Unip(\FF_q)}\left(\Ind_{\TUnip(\FF_q)}^{\Unip(\FF_q)}(\chi), H_c^i(X_h, \overline \QQ_\ell)\right) \cong H_c^i(\pr_1(\beta^{-1}(Y_h)), P^* \Loc_\chi),
\end{equation*}
where $P \from \bA[\cA] \to \bW_h^{(1)}(\FF_{q^n})$ is the morphism $x \mapsto L_q(\det(s(x)))^{-1}.$
\end{proposition}

\begin{proof}
By Proposition \ref{p:B2.3}, we have $\Fr_{q^n}$-compatible isomorphisms
\begin{equation*}
\Hom_{\Unip(\FF_q)}\left(\Ind_{\TUnip(\FF_q)}^{\Unip(\FF_q)}(\chi), H_c^i(X_h, \overline \QQ_\ell)\right) \cong H_c^i(\beta^{-1}(Y_h), \pr_2^* \Loc_\chi'),
\end{equation*}
where $\Loc_\chi'$ is the rank-$1$ local system on $\TUnip$ corresponding to $\chi$ and $\pr_2$ is the composition $\beta^{-1}(Y_h) \hookrightarrow (\Unip/\TUnip) \times \TUnip \stackrel{\pr}{\longrightarrow} \TUnip$. By Lemmas \ref{l:beta graph} and \ref{l:P0}, we have that
\begin{equation*}
H_c^i(\beta^{-1}(Y_h), P_0^* \Loc_\chi') \cong H_c^i(\pr_1(\beta^{-1}(Y_h)), P^* \Loc_\chi),
\end{equation*}
where $\Loc_\chi$ is the rank-$1$ local system on $\bW_h^{(1)}$ corresponding to $\chi$.
\end{proof}

\begin{remark}\label{r:other groups}
The arguments presented in this section can be applied to more general contexts. First observe that once an analogue of the Section \ref{s:P} is established, the general result Proposition \ref{p:B2.3} reduces Theorem \ref{t:hom} to a statement about the cohomology groups of a pullback $P^* \Loc_\chi$ of an Artin--Schreier sheaf to an affine space $S$. These cohomology groups can be calculated using an inductive method on certain linear fibrations of $S$. To achieve this in the proof of Theorem \ref{t:hom}, we use the carefully defined indexing sets $\cA_{s,t}$ determined by a Howe decomposition of the character $\chi$, where the main features of these indexing sets are Lemmas \ref{l:IJ} and \ref{l:det contribution}. 

It is possible to axiomatize the main steps of the proof of Theorem \ref{t:hom} so that we specify exactly what context one can apply the proof of Propositions \ref{p:factor AS} and \ref{p:coprime induct}, which would allow us to reduce the problem of computing the cohomology of certain subschemas of unipotent groups to a combinatorial problem about defining indexing sets analogous to $\cA_{s,t}$. However, we choose to forgo this approach as we feel the present exposition is clearer and more explicit. \hfill $\Diamond$
\end{remark}

We will prove Theorem \ref{t:hom} by combining Proposition \ref{p:hom} with a calculation of the cohomology groups of the Artin--Schreier sheaves $P^* \Loc_\chi$. This calculation is driven by two ideas: the first (Proposition \ref{p:coprime induct}) is an inductive argument that can be viewed as an instance of the techniques established in Section \ref{s:LS}, and the second (Proposition \ref{p:factor AS}) comes from factoring the morphism $P$ through appropriate Lang maps (see Section \ref{s:P}).



\subsection{Compatibility of the morphism $P$ with the Howe factorization} \label{s:P}

Recall the morphisms
\begin{align*}
\pr_1 &\from \beta^{-1}(Y_h) \hookrightarrow \bA[\cA] \times \TUnip \to \bA[\cA], \\ 
\pr_2 &\from \beta^{-1}(Y_h) \hookrightarrow \bA[\cA] \times \TUnip \to \TUnip.
\end{align*}
Define
\begin{align*}
P_{\geq s,t} &\from \fX_{\geq s,t} \colonequals \pr_1(\beta^{-1}(Y_h)) \cap \bA[\cA_{\geq s,t}] \to \bW_{h_t}^{(1)}, \qquad x \mapsto L_q(\det(s(x)))^{-1}.
\end{align*}

\begin{lemma}\label{l:beta graph}
$\beta^{-1}(Y_h)$ is the graph of $P \colonequals P_{\geq 0,0} \from \pr_1(\beta^{-1}(Y_h)) \to \bW_h^{(1)}$.
\end{lemma}

\begin{proof}
Let $x \in \bA[\cA]$ and $y \in \bT_{h,k}$ be such that $s(x) \cdot y \in X_h$. This implies that we have
\begin{equation*}
L_q(\det(s(x))) \cdot L_q(\det(y)) = L_q(\det(s(x) \cdot y)) = (1,0,\ldots,0) \in \bW_h^{(1)},
\end{equation*}
and hence we see that $L_q(\det(s(x)))^{-1} = L_q(\det(y))$. In the remainder of the proof, we show that $L_q(\det(y)) = L_q(y_1)$ for some $y_1 \in \bW_h^{(1)}$ and that $L_q(y)$ is determined by $y_1$.

By construction, the entries along the diagonal of $s(x) \cdot y$ are $y_1, \ldots, y_n$ and by Definition \ref{d:M char}, 
\begin{equation*}
y_i = \varphi^{\tau(i)}(y_1), \qquad \text{for $1 \leq i \leq n$.}
\end{equation*}
Thus the $i$th coordinate of
\begin{equation*}
L_q(y) 
= \diag(\varphi(y_{\sigma(1)}), \varphi(y_{\sigma(2)}), \ldots, \varphi(y_{\sigma(n)})) \cdot \diag(y_1, y_2, \ldots, y_n)^{-1}.
\end{equation*}
is the expression
\begin{equation}\label{e:Lq coordinate}
\varphi(y_{\sigma(i)}) \cdot y_{i}^{-1} = \varphi(\varphi^{\tau(\sigma(i))}(y_1)) \cdot \varphi^{\tau(i)}(y_1^{-1}),
\end{equation}
where $\sigma$ is the bijection defined in Definition \ref{d:tau}. By Lemma \ref{l:sigma tau}, we have $\tau(\sigma(i)) + 1 = [\tau(i)]$ for all $1 \leq i \leq n$, and therefore
\begin{equation*}
\varphi(y_{\sigma(i)}) \cdot y_i^{-1} =
\begin{cases}
\varphi^{[\tau(i)]}(y_1) \cdot \varphi^{\tau(i)}(y_1^{-1}) = 1 & \text{if $i \neq 1$,} \\
\varphi^{[\tau(i)]}(y_1) \cdot \varphi^{\tau(i)}(y_1^{-1}) = \varphi^n(y_1) \cdot y_1^{-1} & \text{if $i = 1$.}
\end{cases}
\end{equation*}
Therefore
\begin{equation*}
L_q(y) 
= \diag(\varphi^n(y_1)y_1^{-1}, 1, \ldots, 1). \qedhere
\end{equation*}
\end{proof}

\begin{lemma}\label{l:P0}
The following diagram commutes:
\begin{equation*}
\begin{tikzcd}[row sep=small]
\beta^{-1}(Y_h) \ar{rr}{\pr_2} \ar{dr}[below left]{P \circ \pr_1} & & \TUnip \\
& \bW_h^{(1)} \ar[hook]{ur}[below right]{x \mapsto \diag(x,1,\ldots,1)} &
\end{tikzcd}
\end{equation*}
\end{lemma}

\begin{proof}
This follows from the proof of Lemma \ref{l:beta graph}.
\end{proof}

\begin{lemma}\label{l:P factor}
The morphism $P_{\geq s,t}$ factors through $L_q$ via the morphism
\begin{equation*}
Q_{\geq s,t} \from \fX_{\geq s,t} \to \bW_{h_t}^{(1)}, \qquad x \mapsto \det(s(x))^{-1}.
\end{equation*}
Equivalently, we have a commutative diagram
\begin{equation*}
\begin{tikzcd}
\fX_{\geq s,t} \ar{rr}{P_{\geq s,t}} \ar{dr}[below left]{Q_{\geq s,t}} && \bW_h^{(1)} \\
& \bW_h^{(1)} \ar{ur}[below right]{L_q}
\end{tikzcd}
\end{equation*}
\end{lemma}

\begin{lemma}\label{l:project to minus}
The projection map
\begin{equation*}
\fX_{\geq s,t} \to \bA[\cA_{\geq s,t}^-]
\end{equation*}
is an isomorphism. Moreover, we have an isomorphism $\fX_{\geq s,t} \cong  \bA[\cA_{\geq s,t}^- \smallsetminus \cA_{\geq s,t+1}^-] \times \fX_{\geq s,t+1}$. 
\end{lemma}

\begin{proof}
We will show that this map is surjective and then use a counting argument to show injectivity.

Let $x \in \bA[\cA_{\geq s,t}]$ and $y \in \bT_{h_t,k}$ be such that $z \colonequals s(x) \cdot y \in X_{h_t}$. Recall that this implies that $y$ is determined by $(1, y_1, \ldots, y_{h_t-1}) \in \bW_{h_t}^{(1)}$. Then for any $\lambda = (i,j,l) \in \cA_{\geq s,t}$,
\begin{equation}\label{e:lambda extension}
z_\lambda = \begin{cases}
x_\lambda + (\text{terms with $x_{(i,j,l')}$ for $l' < l$ and $y_{l''}$ for $l'' < l$}) & \text{if $i < j$,} \\
x_\lambda + (\text{terms with $x_{(i,j,l')}$ for $l' < l$ and $y_{l''}$ for $l'' \leq l$}) & \text{if $i > j$.}
\end{cases}
\end{equation}
This implies that the condition $\det(s(x) \cdot y) \in \bW_{h_t}^{(1)}(\FF_q)$ is equivalent to the vanishing of $h_t-1$ polynomials in $y$ and $x_\nu$ for $\nu \in \cA_{\geq s,t}^-$. Explicitly, if we write $\det(s(x) \cdot y) = (1, d_1, \ldots, d_{h_t-1})$, then the $h_t-1$ polynomials are $d_i^q - d_i$ for $1 \leq i \leq h_t-1$.

Now fix $x^- \in \bA[\cA_{\geq s,t}^-]$. Observe that $d_1$ is a polynomial of degree $q^{n-1}$ in $y_1$, so there are at most $q^n$ roots of the polynomial $d_1^q - d_1$. Let $y_1$ be a root, and observe that $d_2$ is a polynomial of degree $q^{n-1}$ in $y_2$, so that there are at most $q^n$ roots of the polynomial $d_2^q - d_2$. Proceeding in this way, we see that there exists a $y \in \bT_{h_t, k}$ such that $\det(s(x) \cdot y) \in \bW_{h_t}^{(1)}(\FF_q)$ and that there are at most $q^{n(h_t-1)}$ such $y$.

The existence of such a $y \in \bT_{h_t,k}$ now allows us to extend any $x^- \in \bA[\cA_{\geq s,t}^-]$ to an element $x \in \fX_{\geq s,t} \subset \bA[\cA_{\geq s,t}]$ satisfying $s(x) \cdot y \in X_{h_t}$. This shows surjectivity of the projection.

Recall that $x \in \fX_{\geq s,t}$ determines $y$ up to $\bT_{h_t,k}(\FF_q)$-translates. There are exactly $q^{n(h_t-1)}$ such translates, and since there are at most $q^{n(h_t-1)}$ points $y$ satisfying $d_i^q - d_i$ for $1 \leq i \leq h_t - 1$, we see that all these choices of $y$ determine the same extension of $x^- \in \bA[\cA_{\geq s,t}^-]$ to $x \in \fX_{\geq s,t} \subset \bA[\cA_{\geq s,t}]$. This shows injectivity.
\end{proof}

\begin{lemma}\label{l:nu vanish}
Let $x \in \fX_{\geq s,t}$. If $x_\nu = 0$ for some $\nu \in \cA_{\geq s,t}^-$, then $x_\lambda = 0$ for all $\lambda \in \cA_{\geq s,t}$ with $|\lambda| = |\nu|$.
\end{lemma}

\begin{proof}
This follows easily from the proof of Lemma \ref{l:project to minus} (see Equation \eqref{e:lambda extension}). 
\end{proof}

For a moment, we write $\fX_{\geq s,t}^{n,q}$ to emphasize the dependence on $n,q$. For each $0 \leq t \leq r+1$, the sequences
\begin{align*}
&1 \equalscolon m_0 \leq m_1 < m_2 < \cdots < m_r \leq m_{r+1} \colonequals n 
\\
&h \equalscolon h_0 = h_1 > h_2 > \cdots > h_r \geq h_{r+1} \colonequals 1
\end{align*}
induce corresponding sequences
\begin{align*}
&1 \equalscolon m_0' \leq m_1' < m_2' < \cdots < m_{r-t}' \leq m_{r+1-t}' \colonequals n/m_t, && m_{i-t}' \colonequals m_i/m_t, \\
&h' \equalscolon h_0' = h_1' > h_2' > \cdots > h_{r-t}' \geq h_{r+1-t}' \colonequals 1, && h_{i-t}' \colonequals h_i.
\end{align*}
Under this correspondence, $\fX_{\geq t+s_0, t+t_0}^{n,q}$ can be matched with $\fX_{\geq s_0, t_0}^{n/m_t, q^{m_t}}$. We make this precise in the next lemma.

\begin{lemma}\label{l:conductor lowering}
The natural projection 
\begin{equation}\label{e:lowering}
\cA_{\geq, t,t}^{n,q} \to \cA_{\geq 0,0}^{n/m_t, q^{m_t}}, \qquad (i,j,l) \mapsto \left(i, \tfrac{j-i}{m_t}+1, l\right)
\end{equation}
induces an isomorphism $\fX_{\geq t,t}^{n,q} \cong \fX_{\geq 0,0}^{n/m_t, q^{m_t}}$ under which $P_{\geq t,t}^{n,q} = P_{\geq 0,0}^{n/m_t, q^{m_t}}.$
\end{lemma}

\begin{proof}
The map in Equation \eqref{e:lowering} induces a bijection $(\cA_{\geq t,t}^{n,q})^- = (\cA_{\geq 0,0}^{n/m_t, q^{m_t}})^-$ and therefore an isomorphism between the associated affine spaces. We have a commutative diagram
\begin{equation*}
\begin{tikzcd}
\fX_{\geq t,t}^{n,q} \ar{r}\ar{d} & \fX_{\geq 0,0}^{n/m_t, q^{m_t}} \ar{d} \\
\bA[(\cA_{\geq t,t}^{n,q})^-] \ar{r} &\bA[(\cA_{\geq 0,0}^{n/m_t, q^{m_t}})^-]
\end{tikzcd}
\end{equation*}
where the vertical maps are isomorphisms by Lemma \ref{l:project to minus}. It therefore follows that the top map must be an isomorphism.
\end{proof}

\subsection{Calculation of Artin--Schreier sheaves} \label{s:AS calculation}

In this section, we prove the two key propositions required to calculate the cohomology of the (pullbacks of) Artin--Schreier sheaves at hand. Fix a character $\chi \from \bW_h^{(1)}(\F) \cong \TUnip(\FF_q) \to \overline \QQ_\ell^\times$ together with a Howe factorization $\chi = \prod_{i=1}^r \chi_i$ and recall from Section \ref{s:definitions} that this gives rise to associated indexing sets $\cT_{h_t}, \cI_{s,t}, \cJ_{s,t}, \cA_{s,t}$. 

The main idea behind Proposition \ref{p:factor AS} is to use the conductor-lowering compatibility lemmas of Section \ref{s:P} to induct according to the Howe factorization. This proposition allows us to reduce a cohomological calculation about $\chi_{\geq t} = \prod_{i=t}^r \chi_i$ to a calculation about $\chi_{\geq t+1}$, \textit{as long as everything is `divisible' by the conductor of $\chi_t$}. Proposition \ref{p:factor AS} is proved in Section \ref{s:p factor AS}.

\begin{proposition}\label{p:factor AS}
For $1 \leq t \leq r$, we have $\Fr_{q^n}$-compatible isomorphisms
\begin{equation*}
H_c^i\left(\fX_{\geq t,t}, P^* \Loc_{\chi \geq t}\right) \cong H_c^i\left(\fX_{\geq t,t+1}, P^* \Loc_{\chi_{\geq t+1}}\right)[2 e_t] \otimes ((-q^{n/2})^{2e_t})^{\deg},
\end{equation*}
where $e_t = \#(\cA_{\geq t,t} \smallsetminus \cA_{\geq t,t+1})$.
\end{proposition}

Proposition \ref{p:coprime induct} is essentially proved by applying Proposition \ref{p:inductm} inductively. The main calculation is to show that the hypotheses of Proposition \ref{p:inductm} hold. Proposition \ref{p:coprime induct} allows us to \textit{`get rid of' all the parts that are not `divisible' by the conductor of $\chi_t$}, which returns us to the setting of Proposition \ref{p:factor AS}. Proposition \ref{p:coprime induct} is proved in Section \ref{s:p coprime induct}.

\begin{proposition}\label{p:coprime induct}
For $1 \leq t \leq r+1$, we have $\Fr_{q^n}$-compatible isomorphisms
\begin{equation*}
H_c^i\left(\fX_{\geq t-1,t}, P^* \Loc_{\chi_{\geq t}}\right) \cong H_c^i\left(\fX_{\geq t,t}, P^* \Loc_{\chi_{\geq t}} \right)^{\oplus q^{n d_t/2}}[d_t] \otimes ((-q^{n/2})^{d_t})^{\deg},
\end{equation*}
where $d_t = \# \cA_{t-1,t}$.
\end{proposition}

\subsubsection{Proof of Proposition \ref{p:factor AS}}\label{s:p factor AS}

By definition, as characters of $\bW_{h_t}^{(1)}(\FF_{q^n})$, we have
\begin{equation*}
\chi_{\geq t} = \chi_t \cdot \chi_{\geq t+1},
\end{equation*} 
where $\chi_t$ has level $\leq h_t$ and conductor $m_t$, and $\chi_{\geq t+1}$ has level $h_{t+1} < h_t$ and conductor $m_{t+1} > m_t$ where $m_t \mid m_{t+1}$. This implies that
\begin{equation*}
\Loc_{\chi_{\geq t}} = \Loc_{\chi_t} \otimes \pr^*\Loc_{\chi_{\geq t+1}}
\end{equation*}
as sheaves on $\bW_{h_t}^{(1)}$, where $\pr \from \bW_{h_t}^{(1)} \to \bW_{h_{t+1}}^{(1)}.$ 

By Lemma \ref{l:conductor lowering}, we have an isomorphism $\fX_{\geq t,t}^{n,q} \cong \fX_{\geq 0,0}^{n/m_t, q^{m_t}}$ and $P_{\geq t,t}^{n,q} = P_{\geq 0,0}^{n/m_t, q^{m_t}}$. Therefore, by Lemma \ref{l:P factor}, we have the following equality of $\overline \QQ_\ell$-sheaves on $\fX_{\geq t,t}^{n,q}$:
\begin{equation*}
P^* \Loc_{\chi_{\geq t}} = P^* \Loc_{\chi_t} \otimes P^* \pr^* \Loc_{\chi_{\geq t+1}} = Q^*(L_{q^{m_t}}^* \Loc_{\chi_t}) \otimes P^* \pr^* \Loc_{\chi_{\geq t+1}},
\end{equation*}
where $Q = Q_{\geq 0,0}^{n/m_t, q^{m_t}}$. By construction, $\chi_t$ factors through $\bW_{h_t}^{(1)}(\FF_{q^{m_t}})$ and hence $L_{q^{m_t}}^* \Loc_{\chi_t}$ is the trivial local system. By Lemma \ref{l:project to minus}, 
\begin{equation*}
\fX_{\geq t,t} = \bA[\cA_{\geq t,t}^- \smallsetminus \cA_{\geq t,t+1}^-] \times \fX_{\geq t,t+1}.
\end{equation*}
Since the push-forward of $P^* \pr^* \Loc_{\chi_{\geq t+1}}$ to $\bA[\cA_{\geq t,t}^- \smallsetminus \cA_{\geq t,t+1}^-]$ is the trivial local system, we have
\begin{equation*}
P^* \Loc_{\chi_{\geq t}} = \overline \QQ_\ell \boxtimes P^* \Loc_{\chi_{\geq t+1}},
\end{equation*}
where $\overline \QQ_\ell$ is the constant sheaf on $\bA[\cA_{\geq t,t}^- \smallsetminus \cA_{\geq t,t+1}^-]$ and $P^* \Loc_{\chi_{\geq t+1}}$ is the pullback along $P^* \from \fX_{\geq t,t+1} \to \bW_{h_{t+1}}^{(1)}$. Thus by the K\"unneth formula, 
\begin{align*}
H_c^i\big(\fX_{\geq t,t}, {}&P^* \Loc_{\chi_{\geq t}}\big) \\
&= H_c^i\left(\bA[\cA_{\geq t,t}^- \smallsetminus \cA_{\geq t,t+1}^-] \times \fX_{\geq t,t+1}, \overline \QQ_\ell \boxtimes P^* \Loc_{\chi_{\geq t+1}}\right) \\
&\cong \bigoplus_{r+s=i} H_c^r\left(\bA[\cA_{\geq t,t}^- \smallsetminus \cA_{\geq t,t+1}^-], \overline \QQ_\ell\right) \otimes H_c^s\left(\fX_{\geq t,t+1}, P^* \Loc_{\chi_{\geq t+1}}\right) \\
&= H_c^i\left(\fX_{\geq t,t+1}, P^* \Loc_{\chi_{\geq t+1}}\right)[2\#(\cA_{\geq t,t}^- \smallsetminus \cA_{\geq t,t+1}^-)]  \otimes \left((-q^{n/2})^{2\#(\cA_{\geq t,t}^- \smallsetminus \cA_{\geq t,t+1}^-)}\right)^{\deg}.
\end{align*}

\subsubsection{Proof of Proposition \ref{p:coprime induct}}\label{s:p coprime induct}

The main ideas in the proof of Proposition \ref{p:coprime induct} are Lemmas \ref{l:coprime induct} and \ref{l:midpoint}. We will need some notation that will be used in these two lemmas and also in the proof of the proposition at hand. Recall that
\begin{equation*}
\cA_{t-1,t}^- \colonequals \{(1,j,l) \in \cA : j \equiv 1 \!\!\!\!\!\pmod{m_{t-1}}, \, j \not\equiv 1 \!\!\!\!\!\pmod{m_{t}}, \, 1 \leq l \leq h_t - 1\},
\end{equation*}
and define
\begin{align*}
I_0 &\colonequals \cI_{t-1,t} = \{(1,j,l) \in \cA_{t-1,t}^- : |(1,j,l)| > n(h_t-1)/2\},
\\
J_0 &\colonequals \cJ_{t-1,t} = \{(1,j,l) \in \cA_{t-1,t}^- : |(1,j,l)| \leq n(h_t-1)/2\}.
\end{align*}
Recall that these two sets are totally ordered and come with an order-reversing injection
\begin{equation*}
I_0 \hookrightarrow J_0, \qquad (1,\sigma^j(1),l) \mapsto (1, \sigma^j(1), l)' \colonequals (1,\sigma^{n-j}(1),h_t-l)
\end{equation*}
as in Lemma \ref{l:IJ}. This map is a bijection if and only if $(1,\sigma^{n/2}(1),h_t/2) \notin J_0$, and $\# J_0 - \# I_0 = 1$ otherwise. Given this, there is a natural filtration on $I_0$ by iteratively removing a highest-norm element, and analogously a natural filtration on $J_0$ by iteratively removing a lowest-norm element. These filtrations can be specified so that they are compatible under the map $I_0 \to J_0$ above. Precisely, fix a labeling $\nu_0, \nu_1, \ldots, \nu_{\# I_0-1}$ of the elements of $I_0$ where $|\nu_i| \geq |\nu_{i + 1}|$ for $0 \leq i < \# I_0-1$. We may define
\begin{equation*}
I_{\kappa+1} \colonequals I_{\kappa} \smallsetminus \{\nu_\kappa\}, \qquad J_{\kappa+1} \colonequals J_{\kappa} \smallsetminus \{\nu_\kappa'\}.
\end{equation*}
Define
\begin{align*}
\widetilde I_\kappa &\colonequals \{(i,j-i+1,l) \in \cA_{t-1,t} : (1,\sigma^j(1),l) \in I_\kappa\}, \\
\widetilde J_\kappa &\colonequals \{(i,j-i+1,l) \in \cA_{t-1,t} : (1,\sigma^j(1),l) \in J_\kappa\}.
\end{align*}

\begin{lemma}\label{l:coprime induct}
For each $\kappa \leq \# I_0 - 1$, we have $\Fr_{q^n}$-compatible isomorphisms
\begin{align*}
H_c^i(\fX_{\geq t-1,t} \cap {}&{}\bA[\cA_{\geq t,t} \cup (\widetilde I_\kappa \cup \widetilde J_\kappa)], P^* \Loc_{\chi_{\geq t}}) \\
&\cong H_c^i(\fX_{\geq t-1,t} \cap \bA[\cA_{\geq t,t} \cup (\widetilde I_{\kappa+1} \cup \widetilde J_{\kappa+1})], P^* \Loc_{\chi_{\geq t}})^{\oplus q^n}[2] \otimes (q^n)^{\deg}.
\end{align*}
\end{lemma}

\begin{remark}
A variation of Lemma \ref{l:coprime induct} for the equal characteristic case with $(\{1, n, n\}, \{h, h, 1\})$ was proved in Lemma 5.11 of \cite{C15}. We give a different proof here that does not refer to juggling sequences or the explicit equations cutting out $X_h$. \hfill $\Diamond$
\end{remark}

\begin{proof}
Recall that by definition $I_\kappa$ and $J_\kappa$ depend on $t$. This proof is driven by the following simple goal: to apply Proposition \ref{p:inductm}. To this end, the content of this lemma is the calculation that the polynomial 
\begin{equation*}
P \from \fX_{\geq t-1,t} \cap \bA[\cA_{\geq t, t} \cup (\widetilde I_\kappa \cup \widetilde J_\kappa)] \to \bW_h^{(1)}, \qquad x \mapsto L_q(\det(s(x)))^{-1}
\end{equation*}
has the required form for $f(x) = x_{\nu_\kappa}$ and $y = x_{\nu_\kappa'}$.

Let $x \in \fX_{\geq t-1,t} \cap \bA[\cA_{\geq t, t} \cup (\widetilde I_\kappa \cup \widetilde J_\kappa)]$ and let $y \in \TUnip$ be such that $z \colonequals s(x) \cdot y \in X_h$. Set $\nu \colonequals \nu_\kappa \colonequals (1,\sigma^j(1),l) \in I_\kappa$ and $\nu' \colonequals \nu_\kappa' \colonequals (1, \sigma^{n-j}(1), h_t - l) \in J_\kappa$, and note that $|\nu| + |\nu'| = n(h_t - 1)$. 

Take $\lambda_1 \colonequals (i, \sigma^j(i), l) \in \widetilde I_\kappa$ so that $|\lambda_1| = |\nu|$. Let $\lambda_2 \in \cA_{\geq t,t} \cup \widetilde I_\kappa \cup \widetilde J_\kappa$ be such that both $z_{\lambda_1}$ and $z_{\lambda_2}$ contribute to a monomial in $\det(s(x))$. I first claim that necessarily $\lambda_2 \in \widetilde J_\kappa$. Indeed, if $\gamma \in S_n$ is such that $\gamma(i) = \sigma^j(i)$, then $m_t \nmid \gamma(i) - i$, and hence there must be another $i'$ such that $m_t \nmid \gamma(i') - i'$. If $\gamma$ corresponds to a nontrivial summand of $\det(z)$, then $|(i', \gamma(i'), l')| \leq n(h_t - 1) - |(i, \gamma(i), l)| \leq n(h_t-1)/2$ and so $(i', \gamma(i'), l') \in \widetilde J_\kappa$. But now $|\lambda_1| + |(i', \gamma(i'), l')| \geq |\nu| + |\nu'| = n(h_t - 1)$ and by Lemma \ref{l:det contribution}, this inequality must be an equality. Thus $\gamma$ must be a transposition, and we see that $\lambda_2 \in \widetilde J_\kappa$ and in fact $\lambda_2 = (\sigma^j(i), i, n-l) = (\sigma^j(i), \sigma^{n-j}(\sigma^j(i)), n-l)$. By Equation \eqref{e:entries}, $z_{\lambda_1}$ and $z_{\lambda_2}$ are powers of $z_\nu$ and $z_{\nu'}$, respectively:
\begin{equation*}
z_{\lambda_1} = z_\nu^{q^{\tau(i)}}, \qquad \text{and} \qquad z_{\lambda_2} = z_{\nu'}^{q^{\tau(\sigma^j(i))}} = z_{\nu'}^{q^{[\tau(i)-j+1]-1}}.
\end{equation*}
Therefore the contribution of $\lambda_1$ and $\lambda_2$ is
\begin{equation}\label{e:indiv contribution}
(-1)^{\sgn \gamma} z_{\lambda_1} z_{\lambda_2} = -z_{\lambda_1} z_{\lambda_2} = 
\begin{cases}
z_\nu^{q^{\tau(i)}} \cdot z_{\nu'}^{q^{[\tau(i)-j+1]-1}} & \text{if $\Char K > 0$}, \\
z_\nu^{q^{\tau(i)+h_t-l}} \cdot z_{\nu'}^{q^{[\tau(i)-j+1]-1+l}} & \text{if $\Char K = 0$}.
\end{cases}
\end{equation}
The contribution of $z_\nu$ to $L_q(\det(z))^{-1}$ is given by
\begin{equation}\label{e:distinct contribution}
\sum_{\lambda_1} (z_{\lambda_1}z_{\lambda_2})^q -  z_{\lambda_1}z_{\lambda_2},
\end{equation}
and after writing this sum in terms of $z_\nu$ and $z_{\nu'}$ as in Equation \eqref{e:indiv contribution}, all terms cancel except for those corresponding to when
\begin{equation*}
\tau(i) = 0, \quad \tau(i) = n-1, \quad [\tau(i)-j+1] - 1 = 0, \quad \text{or} \quad [\tau(i)-j+1] - 1 = n-1,
\end{equation*}
which exactly corresponds to when
\begin{equation*}
\tau(i) = 0, \qquad \tau(i) = n-1, \qquad \tau(i) = j, \qquad \text{or} \qquad \tau(i) = j-1.
\end{equation*}
Hence the sum in \eqref{e:distinct contribution} simplifies to
\begin{equation}\label{e:induction poly z}
\begin{cases}
(z_\nu^{q^n} - z_\nu) \cdot z_{\nu'}^{q^{j}} + (z_{\nu'}^{q^n} - z_{\nu'}) \cdot z_\nu^{q^{n-j}} & \text{if $\Char K > 0$}, \\
(z_\nu^{q^n} - z_\nu)^{q^{h_t - l}} \cdot z_{\nu'}^{q^{j+l}} + (z_{\nu'}^{q^n} - z_{\nu'})^{q^{l}} \cdot z_\nu^{q^{n-j+h_t-l}} & \text{if $\Char K = 0$}.
\end{cases}
\end{equation}

We now investigate how Equation \eqref{e:induction poly z} allows us to understand the contribution of $x_\nu$ and $x_{\nu'}$ to $L_q(\det(z)) = L_q(\det(s(x) \cdot y)) \in \bW_{h_t}$. If $x_\lambda$ appears in $z_{\lambda_0}$, then $|\lambda| \leq |\lambda_0|$. Necessarily $\lambda_0$ does not sit on the diagonal since $\lambda$ does not, and hence if $z_{\lambda_0}$ contributes to $L_q(\det(z))$, there must exist a $\lambda_0'$ such that $z_{\lambda_0'}$ also contributes to the same monomial. But by Lemma \ref{l:det contribution}, this implies that $|\lambda_0'| \leq n(h_t - 1) - |\lambda_0| \leq |\lambda'|$. This shows that the contribution of $x_\lambda$ and $x_{\lambda'}$ is contained in the contribution of $z_\lambda$ and $z_{\lambda'}$. Now, 
\begin{align}
\label{e:x lambda'}
z_{\lambda'} &= x_{\lambda'} + (\text{terms that each include a factor of $x_{\lambda_0'}$ for $\lambda_0' \in \widetilde J_0$ and $|\lambda_0'| < |\nu'|$}) = z_{\lambda'} \\ \nonumber
z_\lambda &= x_\lambda + (\text{terms that each include a factor of $x_{\lambda_0}$ for $\lambda_0 \in \widetilde J_0$ and $|\lambda_0| < |\nu|$})
\end{align}
and hence we see that the contribution of $x_\lambda$ is captured by the contribution of $z_\lambda$ which is captured by the contribution of $z_\nu$. That is, the contribution of $x_\lambda$ is equal to
\begin{equation}\label{e:induction poly}
\begin{cases}
(x_\nu^{q^n} - x_\nu) \cdot x_{\nu'}^{q^{j}} + (x_{\nu'}^{q^n} - x_{\nu'}) \cdot x_\nu^{q^{n-j}} & \text{if $\Char K > 0$}, \\
(x_\nu^{q^n} - x_\nu)^{q^{h_t - l}} \cdot x_{\nu'}^{q^{j+l}} + (x_{\nu'}^{q^n} - x_{\nu'})^{q^{l}} \cdot x_\nu^{q^{n-j+h_t-l}} & \text{if $\Char K = 0$}.
\end{cases}
\end{equation}

By assumption, $m_t$ does not divide $j$ since $\nu \notin \cA_{\geq t,t}$. Since $\chi_{\geq t}$ has conductor $m_t$, Equation \eqref{e:induction poly} shows that $P$ has the form required to apply Proposition \ref{p:inductm} in the case $f(x) = x_{\nu'}$ and $y = x_\nu$. Therefore
\begin{equation*}
H_c^i(\fX_{\geq t-1,t} \cap \bA[\cA_{\geq t, t} \cup (\widetilde I_\kappa \cup \widetilde J_\kappa)], P^* \Loc_{\geq t}) \cong H_c^i(S_3, P^* \Loc_{\geq t})[2] \otimes (q^n)^{\deg},
\end{equation*}
where $S_3 \subset \fX_{\geq t-1,t} \cap \bA[\cA_{\geq t, t} \cup (\widetilde I_\kappa \cup \widetilde J_\kappa) \smallsetminus \{\nu\}]$ is the subscheme defined by $f(x)^{q^n} - f(x) = 0$. Since $S_3$ is the disjoint union of $q^n$ copies of $\fX_{\geq t-1, t} \cap \bA[\cA_{\geq t, t} \cup (\widetilde I_\kappa \cup \widetilde J_\kappa) \smallsetminus \{\nu, \nu'\}]$,
\begin{align*}
H_c^i(S_3, P^* \Loc_{\geq t})[2] \otimes (q^n)^{\deg} 
&\cong H_c^i(\fX_{\geq t-1,t} \cap \bA[\cA_{\geq t, t} \cup (\widetilde I_\kappa \cup \widetilde J_\kappa) \smallsetminus \{\nu, \nu'\}], P^* \Loc_{\geq t})^{\oplus q^n}[2] \otimes (q^n)^{\deg} \\
&= H_c^i(\fX_{\geq t-1,t} \cap \bA[\cA_{\geq t, t} \cup (\widetilde I_{\kappa+1} \cup \widetilde J_{\kappa+1})], P^* \Loc_{\geq t})^{\oplus q^n}[2] \otimes (q^n)^{\deg},
\end{align*}
where the last equality holds by Lemma \ref{l:nu vanish}.
\end{proof}

\begin{lemma}\label{l:midpoint}
Suppose $\# J_0 - \# I_0 = 1$. Then $J_{\# I_0} = \{v \colonequals (1, n/2+1,h_t/2) \in J_0\}$ and we have $\Fr_{q^n}$-compatible isomorphisms
\begin{equation*}
H_c^i(\fX_{\geq t-1,t} \cap \bA[\cA_{\geq t,t} \cup \widetilde J_{\# I_0}], P^* \Loc_{\chi_{\geq t}}) = H_c^i(\fX_{\geq t-1,t} \cap \bA[\cA_{\geq t,t}], P^* \Loc_{\chi_{\geq t}})^{\oplus q^{n/2}}[1] \otimes (-q^{n/2})^{\deg}
\end{equation*}
\end{lemma}

\begin{proof}
For ease of notation, set $i_0 \colonequals \# I_0$ in this proof. By the divisibility assumption on $\cA_{\geq t, t}$, we see that 
\begin{equation*}
\fX_{\geq t-1,t} \cap \bA[\cA_{\geq t, t} \cup \widetilde J_{i_0}] = (\fX_{\geq t-1,t} \cap \bA[\cA_{\geq t, t}]) \times (\fX_{\geq t-1,t} \cap \bA[\widetilde J_{i_0}]).
\end{equation*}
Moreover, by Lemma \ref{l:project to minus}, the projection
\begin{equation*}
f \from (\fX_{\geq t-1,t} \cap \bA[\cA_{\geq t, t}]) \times (\fX_{\geq t-1,t} \cap \bA[\widetilde J_{i_0}]) \to (\fX_{\geq t-1,t} \cap \bA[\cA_{\geq t,t}]) \times \bA[\{v\}]
\end{equation*}
is an isomorphism of varieties. By Lemma \ref{l:det contribution}, it follows that any $\lambda \in \widetilde J_{i_0}$ can only contribute nontrivially to the last coordinate of $\bW_{h_t}^{(1)}$ in $\det(s(x))$. Therefore, for $P_0 = P|_{\fX_{\geq t,t}}$ and some morphism $\eta \from \bA[\{v\}] \to \bW_{h_t}^{(1)}$, we have 
\begin{equation*}
f_* P^* \Loc_{\chi_{\geq t}} = P_0^* \Loc_{\chi_{\geq t}} \boxtimes \eta^* \Loc_{\chi_{\geq t}},
\end{equation*}
and therefore
\begin{align*}
H_c^i(\fX_{\geq t,t} \cap \bA[\cA_{\geq t,t} \cup \widetilde J_{i_0}], P^* \Loc_{\chi_{\geq t}})
&= H_c^i(\fX_{\geq t,t} \times \bA[\{v\}], P_0^* \Loc_{\chi_{\geq t}} \boxtimes \eta^* \Loc_{\chi_{\geq t}}) \\
&\cong \bigoplus_{r + s = i} H_c^r(\fX_{\geq t,t}, P_0^* \Loc_{\chi_{\geq t}}) \otimes H_c^s(\bA[\{v\}], \eta^* \Loc_{\chi_{\geq t}}).
\end{align*}
It now remains to calculate $\eta$, which we do by a similar calculation to the one in Lemma \ref{l:coprime induct}. 

By Lemma \ref{l:IJ}, $v = (1, n/2+1, h_t/2)$. The contribution of $x_v$ to the last coordinate of $L_q(\det(s(x)))^{-1}$ is exactly given by
\begin{equation*}
\begin{cases}
\sum_{i=1}^n x_v^{q^{\tau(i)+1}} \cdot x_v^{q^{[\tau(i)+n/2]+1}} - x_v^{q^{\tau(i)}} \cdot x_v^{q^{[\tau(i)+n/2]}} & \text{if $\Char K > 0$,} \\
\sum_{i=1}^n x_v^{q^{\tau(i) + (h_t-1)/2+1}} \cdot x_v^{q^{[\tau(i)+n/2]+(h_t-1)/2+1}} - x_v^{q^{\tau(i)+(h_t-1)/2}} \cdot x_v^{q^{[\tau(i)+n/2]+(h_t-1)/2}} & \text{if $\Char K = 0$.}
\end{cases}
\end{equation*}
This simplifies to
\begin{equation*}
\begin{cases}
(x_v^{q^n} - x_v) \cdot x_v^{q^{n/2}} & \text{if $\Char K > 0$,} \\
(x_v^{q^n} - x_v)^{q^{h_t/2}} \cdot x_v^{q^{(n+h_t)/2}} & \text{if $\Char K = 0$,}
\end{cases}
\end{equation*}
and it follows that
\begin{equation*}
\eta \from \bA[\{v\}] \to \bW_{h_t}^{(1)}, \qquad x \mapsto 
\begin{cases}
(0, \ldots, 0, x^{q^{n/2}}(x^{q^n}-x)) & \text{if $\Char K > 0$,} \\
(0, \ldots, 0, x^{q^{(n+h_t)/2}}(x^{q^n}-x)^{q^{h_t/2}}) & \text{if $\Char K = 0$.} \\
\end{cases}
\end{equation*}
We can now make the final conclusion. By \cite[Equation (6.5.7)]{BW16} (or \cite[Step 4:\ Case 2 of Proposition 6.1]{C15}), we have
\begin{equation*}
\dim H_c^i(\bG_a, \eta^* \Loc_{\chi_{\geq t}}) = \begin{cases}
q^{n/2}, & \text{if $i = 1$,} \\
0, & \text{otherwise.}
\end{cases}
\end{equation*}
Moreover, $\Fr_{q^n}$ acts on $H_c^1(\GG_a, \eta^* \Loc_{\chi_{\geq t}})$ by multiplication by $-q^{n/2}$. 
\end{proof}

The desired conclusion of Proposition \ref{p:coprime induct} now follows by combining Lemmas \ref{l:coprime induct} and \ref{l:midpoint}. Explicitly, by Lemma \ref{l:coprime induct},
\begin{align*}
H_c^i(\fX_{\geq t-1,t} \cap {}&{} \bA[\cA_{\geq t-1,t}], P^* \Loc_{\chi_{\geq t}}) \\
&= H_c^i(\fX_{\geq t-1,t} \cap \bA[\cA_{\geq t,t} \cup (\widetilde I_0 \cup \widetilde J_0)], P^* \Loc_{\chi_{\geq t}}) \\
&\cong H_c^i(\fX_{\geq t-1,t} \cap \bA[\cA_{\geq t,t} \cup (\widetilde I_\kappa \cup \widetilde J_\kappa)], P^* \Loc_{\chi_{\geq t}})^{\oplus q^{\kappa n}}[2\kappa] \otimes ((-q^{n/2})^{2\kappa})^{\deg}
\end{align*}
for any $\kappa \leq \#I_0$. Clearly $I_{\# I_0} = \varnothing$. If $J_{\# I_0} = \varnothing$, then we are done. Otherwise, $J_{\# I_0} = \{(1,1+n/2, (h_t-1)/2)\}$ and by Lemma \ref{l:midpoint},
\begin{align*}
H_c^i(\fX_{\geq t-1,t} \cap \bA[\cA_{\geq t,t} \cup &(I_{\# I_0} \cup J_{\# I_0})], P^* \Loc_{\chi_{\geq t}})[2 \# I_0] \otimes ((-q^{n/2})^{2\# I_0})^{\deg} \\
&\cong H_c^i(\fX_{\geq t,t}, P^* \Loc_{\chi_{\geq t}})^{\oplus q^{n/2}}[2 \# I_0 + 1] \otimes ((-q^{n/2})^{2 \# I_0 + 1})^{\deg}.
\end{align*}
It remains to observe that if $J_{\# I_0} = \varnothing,$ then $2\# I_0 = \#(I_0 \cup J_0) = \#(\cI_{t-1,t} \cup \cJ_{t-1,t})$, and otherwise, $2 \# I_0 + 1 = \#(I_0 \cup J_0) = \#(\cI_{t-1,t} \cup \cJ_{t-1,t})$. This completes the proof.

\subsection{Morphisms: proof of Theorem \ref{t:hom}} \label{s:hom}

Let $\chi \from \bW_h^{(1)}(\F) \cong \TUnip(\FF_q) \to \overline \QQ_\ell^\times$ be any character, and let $\cA_{s,t}$ be indexing sets associated to a Howe factorization of $\chi$. By combining Propositions \ref{p:hom}, \ref{p:coprime induct}, and \ref{p:factor AS}, we have
\begin{align*}
&\Hom_{\Unip(\FF_q)}\left(\Ind_{\TUnip(\FF_q)}^{\Unip(\FF_q)}(\chi), H_c^i(X_h, \overline \QQ_\ell)\right) \\
&\qquad\cong H_c^i(\pr_1(\beta^{-1}(Y_h)), P^* \Loc_\chi) && \text{(Prop \ref{p:hom})} \\
&\qquad= H_c^i(\fX_{\geq 0,0}, P^* \Loc_{\chi_{\geq 0}}) && \text{(Lemma \ref{l:beta graph})} \\
&\qquad\cong  H_c^i(\fX_{\geq 0,1}, P^* \Loc_{\chi_{\geq 1}})[2 e_0]\left((-q^{n/2})^{2e_0}\right)^{\deg} && \text{(Prop \ref{p:factor AS})} \\
&\qquad\cong H_c^i(\fX_{\geq 1,1}, P^* \Loc_{\chi_{\geq 1}})^{\oplus q^{nd_1/2}}[d_1+2e_0]\left((-q^{n/2})^{d_1+2e_0}\right)^{\deg} && \text{(Prop \ref{p:coprime induct})} \\
&\qquad\cong H_c^i(\fX_{\geq 1,2}, P^* \Loc_{\chi_{\geq 2}})^{\oplus q^{nd_1/2}}[d_1 + 2(e_0+e_1)]\left((-q^{n/2})^{d_1+2(e_0+e_1)}\right)^{\deg} && \text{(Prop \ref{p:factor AS})}
\end{align*}
and so forth by iteratively applying Propositions \ref{p:factor AS} and \ref{p:coprime induct}. Recall that
\begin{align*}
d_t &= \#\cA_{t-1,t}^- \\
&= \{(1,j,l) : m_{t-1} \mid j-1, \, m_t \nmid j-1, \, 1 \leq l \leq h_t - 1\} \\
&= \textstyle\left(\frac{n}{m_{t-1}} - \frac{n}{m_t}\right)(h_t-1), && \text{for $1 \leq t \leq r+1$}, \\
e_t &= \#(\cA_{\geq t,t}^- \smallsetminus \cA_{\geq t,t+1}^-) \\
&= \#\{(1,j,l) : m_t \mid j-1, \, j \neq 1, \, n(h_{t+1}-1) < l \leq n(h_t - 1)\} \\
&= \textstyle \left(\frac{n}{m_t}-1\right)\left(h_t - h_{t+1}\right), && \text{for $0 \leq t \leq r$}.
\end{align*}
Thus setting
\begin{align*}
d_\chi
&\colonequals d_1 + \cdots + d_{r+1} \\
&= \sum_{t=1}^{r+1} \left(\frac{n}{m_{t-1}} - \frac{n}{m_t}\right)(h_t - 1), \\
r_\chi 
&\colonequals (d_1 + \cdots + d_{r+1}) + 2(e_0 + \cdots + e_r) \\
&= \sum_{t=1}^{r+1} \left(\left(\frac{n}{m_{t-1}} - \frac{n}{m_t}\right)(h_t-1) + 2\left(\frac{n}{m_{t-1}}-1\right)\left(h_{t-1} - h_{t}\right)\right),
\end{align*}
the cohomology groups above are isomorphic to
\begin{equation*}
H_c^i(\fX_{\geq r,r}, P^* \Loc_{\chi_{\geq r}})^{q^{nd_\chi/2}}[r_\chi]\left((-q^{n/2})^{r_\chi}\right)^{\deg}.
\end{equation*}
Note that $\fX_{\geq r,r}$ is a single point and hence we obtain 
\begin{equation*}
\Hom_{\Unip(\FF_q)}\left(\Ind_{\TUnip(\FF_q)}^{\Unip(\FF_q)}(\chi), H_c^i(X_h, \overline \QQ_\ell)\right) = \begin{cases}
\overline \QQ_\ell^{\oplus q^{nd_\chi/2}}, & \text{if $i = r_\chi$}, \\
0 & \text{otherwise.}
\end{cases}
\end{equation*}
Moreover, since $\Fr_{q^n}$ acts trivially on $H_c^0(*, \overline \QQ_\ell)$, then $\Fr_{q^n}$ acts by multiplication by $(-1)^{r_\chi} q^{nr_\chi/2}$ on the above space of $\Unip(\FF_q)$-homomorphisms. Since every representation occurs in at least one of $\Ind_{\TUnip(\FF_q)}^{\Unip(\FF_q)}(\chi)$, the above gives a complete description of the $\Fr_{q^n}$-action on $H_c^i(X_h, \overline \QQ_\ell)$, and this action is always by multiplication by $(-1)^i q^{ni/2}$. 


\section{Deligne--Lusztig theory for finite unipotent groups}\label{s:DL}

In this section, we prove the main theorems of this paper. We first calculate the alternating sum of the cohomology groups using a technique of \cite{L79}. This is very similar to the results of Lusztig \cite{L04} and Stasinski \cite{S09}, which study closely related groups in a reductive setting. Combining Theorem \ref{t:hom} and Theorem \ref{t:R chi} gives Theorems \ref{t:maximality} and \ref{t:irreducibility}, which prove Boyarchenko's conjectures \cite[Conjectures 5.16 and 5.18]{B12} in full generality. Note that strictly speaking, as stated in \cite{B12} these two conjectures assume that $\Char K > 0$ and $k = 1$; however, they can be easily extended and formulated without these assumptions (see \cite[Conjectures 7.4, 7.5]{C15}). Theorems \ref{t:maximality} and \ref{t:irreducibility} can be viewed naturally as the higher-dimensional analogues of the results of Boyarchenko and Weinstein on the cohomology of $X_2$ in \cite[Sections 4-6]{BW16}.

We remark that in all previous work (i.e.\ the $h=2$ work of Boyarchenko--Weinstein \cite{BW16} and the primitive-$\chi$, equal-characteristic work of the author in \cite{C15}, \cite{C16}), pinning down the nonvanishing cohomological degree $i=s_\chi$ of $H_c^i(X_h, \overline \QQ_\ell)[\chi]$ was a trivial consequence of (the analogues of) Theorem \ref{t:hom}. This is because the central character of $H_c^i(X_h, \overline \QQ_\ell)[\chi]$ determines $s_\chi$ in these cases and so $s_\chi$ agrees with the $r_\chi$ appearing in Theorem \ref{t:hom}. However, in the general setting, this no longer holds, and it is a nontrivial theorem that the equality $s_\chi = r_\chi$ is still true (Theorem \ref{t:s chi}). The proof is an application of the Deligne--Lusztig trace formula \cite[Theorem 3.2]{DL76} and is given in Section \ref{s:s chi}.

The trio of theorems \ref{t:maximality}, \ref{t:irreducibility}, and \ref{t:s chi} gives us a complete description of the $\TUnip(\FF_q)$-eigenspaces $H_c^i(X_h, \overline \QQ_\ell)[\chi]$ together with the Frobenius action on $H_c^i(X_h, \overline \QQ_\ell)$. Combining these theorems with Theorem \ref{t:hom} and the fact that the multiplicity of an irreducible $\rho$ in the regular representation is equal to the dimension of $\rho$, we may write down an explicit formula for the zeta function of $X_h$. This is done in Theorem \ref{t:zeta}.

In Section \ref{s:examples}, we demonstrate how to realize the main theorems of \cite{B12}, \cite{BW16}, \cite{C15}, and \cite{C16} as corollaries of the theorems in this paper.

\begin{theorem}\label{t:R chi}
Let $R_\chi \colonequals \sum_i (-1)^i H_c^i(X_h, \overline \QQ_\ell)[\chi]$. For each $\chi \from \TUnip(\FF_q) \to \overline \QQ_\ell^\times$, the $\Unip(\F)$-representation $\pm R_\chi$ is irreducible. If $\chi \neq \chi'$, then $\pm R_\chi, \pm R_{\chi'}$ are nonisomorphic.
\end{theorem}

Theorems \ref{t:maximality} and \ref{t:irreducibility} now follow from Theorem \ref{t:hom}, Theorem \ref{t:R chi}, and Corollary \ref{c:hom}.
 
\begin{theorem}\label{t:maximality}
$X_h$ is a maximal variety in the sense of Boyarchenko--Weinstein  \cite{BW16}. That is, for each $i \geq 0$, we have $H_c^i(X_h, \overline \QQ_\ell) = 0$ unless $i$ or $n$ is even, and the Frobenius morphism $\Fr_{q^n}$ acts on $H_c^i(X_h, \overline \QQ_\ell)$ by the scalar $(-1)^i q^{ni/2}$.
\end{theorem}

\begin{proof}
By Theorem \ref{t:hom}, $\Fr_{q^n}$ acts on $H_c^i(X_h, \overline \QQ_\ell)$ by multiplication by $(-1)^i q^{ni/2}$. To finish, we show that $H_c^i(X_h, \overline \QQ_\ell) = 0$ if $i$ and $n$ are both odd. Assume that $n$ is odd. By definition of $r_\chi$, it is enough to show that the sum $d_1 + \cdots + d_{r+1}$ is always even. We have
\begin{align*}
d_t = \left(\frac{n}{m_{t-1}} - \frac{n}{m_t}\right)(h_t - 1),
\end{align*}
and since $n$ is odd by assumption, then $n/m_{t-1}$ and $n/m_{t}$ must also be odd, and hence $d_t$ is even. This completes the proof.
\end{proof}

\begin{theorem}\label{t:irreducibility}
For any $\chi \in \sT_{n,h}$,  the cohomology groups $H_c^i(X_h, \overline \QQ_\ell)[\chi]$ are nonzero in a single degree $i = s_\chi$, and $H_c^{s_\chi}(X_h, \overline \QQ_\ell)[\chi]$ is an irreducible representation of $\Unip(\FF_q)$. Moreover, $H_c^{s_\chi}(X_h, \overline \QQ_\ell)[\chi] \cong H_c^{s_{\chi'}}(X_h, \overline \QQ_\ell)[\chi']$ if and only if $\chi = \chi'$.
\end{theorem}

\begin{proof}
Let $\pi$ be an irreducible constituent in $H_c^{s_\chi}(X_h, \overline \QQ_\ell)[\chi]$ for some $s_\chi$. Then by Corollary \ref{c:hom}, 
\begin{equation*}
\Hom_{\Unip(\FF_q)}(\pi, H_c^i(X_h, \overline \QQ_\ell)[\chi]) = 0, \qquad \text{for all $i \neq s_\chi$}.
\end{equation*}
But this implies that the alternating sum $\pm R_\chi = \sum_i (-1)^i H_c^i(X_h, \overline \QQ_\ell)[\chi]$ can have no cancellation, and $\pi \cong \pm R_\chi$. By Theorem \ref{t:R chi}, 
\begin{equation*}
H_c^i(X_h, \overline \QQ_\ell)[\chi] = \begin{cases}
\text{irreducible}, & \text{if $i = s_\chi$,} \\
0, & \text{otherwise.}
\end{cases}
\end{equation*}
Moreover, if $\chi \neq \chi'$, then $\pm R_\chi$ and $\pm R_{\chi'}$ are nonisomorphic by Theorem \ref{t:R chi} and it follows easily that $H_c^{s_\chi}(X_h, \overline \QQ_\ell)[\chi]$ and $H_c^{s_{\chi'}}(X_h, \overline \QQ_\ell)[\chi']$ must also be nonisomorphic.
\end{proof}

For any $\zeta \in \F$ and any $g_1, g_2 \in \TUnip(\FF_q)$, let $(\zeta, g_1, g_2)$ denote the map $X_h \to X_h$ given by $x \mapsto \zeta(h * x \cdot g) \zeta^{-1}$.

\begin{theorem}\label{t:character}
If $\zeta \in \FF_{q^n}^\times$ has trivial stabilizer in $\Gal(\FF_{q^n}/\FF_q)$, then for any $g \in \TUnip(\FF_q)$,
\begin{equation*}
\Tr\left((\zeta, 1, g)^* ; H_c^{s_\chi}(X_h, \overline \QQ_\ell)[\chi]\right) = (-1)^{s_\chi} \chi(g).
\end{equation*}
\end{theorem}

\begin{proof}
This is identical to the proof of Proposition 6.2 of \cite{C15}. The argument is very similar to the proof of Lemma \ref{l:tr reduction} in the present paper.
\end{proof}

Up to now, we have only shown that $H_c^i(X_h, \overline \QQ_\ell)[\chi]$ is concentrated in a single degree $s_\chi$. It is natural to expect, based on Theorem \ref{t:hom}, that $s_\chi = r_\chi$. We resolve this question in Theorem \ref{t:s chi}, whose proof we give in Section \ref{s:s chi}. The proof uses purely cohomological techniques and essentially is a combination of Theorem \ref{t:character} together with the Deligne--Lusztig fixed point formula. 

\begin{theorem}\label{t:s chi}
For any $\chi \from \TUnip(\FF_q) \to \overline \QQ_\ell^\times$,
\begin{equation*}
\Hom_{\Unip(\FF_q)}\left(\Ind_{\TUnip(\FF_q)}^{\Unip(\FF_q)}(\chi), H_c^{s_\chi}(X_h, \overline \QQ_\ell)[\chi]\right) \neq 0.
\end{equation*}
In particular, $s_\chi = r_\chi$.
\end{theorem}

Our last theorem of this section gives a formula for the zeta function of $X_h$.

\begin{theorem}\label{t:zeta}
The Hasse--Weil zeta function of $X_h$ is
\begin{equation*}
Z(X_h, t) = \prod_{i=0}^{2\dim X_h} \left(1 - (-q^{n/2})^i \cdot t\right)^{(-1)^{i+1} \dim H_c^i(X_h, \overline \QQ_\ell)},
\end{equation*}
where
\begin{equation*}
\dim H_c^i(X_h, \overline \QQ_\ell) = \sum_{\substack{\chi \from \TUnip(\FF_q) \to \overline \QQ_\ell^\times, \\ r_\chi = i}} q^{n d_\chi/2}.
\end{equation*}
Moreover, if $n$ is odd, then $Z(X_h, t)^{-1}$ is a polynomial.
\end{theorem}

\begin{proof}
By the Grothendieck--Lefschetz trace formula,
\begin{equation*}
Z(X_h, t) = \prod_{i=0}^{2\dim X_h} \left(\det((1-t \Fr_{q^n}) ; H_c^i(X_h, \overline \QQ_\ell))\right)^{(-1)^{i+1}}.
\end{equation*}
By Theorem \ref{t:maximality}, we know that $\Fr_{q^n}$ acts on $H_c^i(X_h, \overline \QQ_\ell)$ by multiplication by $(-q^{n/2})^i$, so 
\begin{equation*}
\det((1-t \Fr_{q^n}); H_c^i(X_h, \overline \QQ_\ell)) = (1 - (-q^{n/2})^i \cdot t)^{\dim H_c^i(X_h, \overline \QQ_\ell)}.
\end{equation*}
It remains to prove the dimension formula.

Let $H_c^\bullet(X_h, \overline \QQ_\ell) = \bigoplus_i H_c^i(X_h, \overline \QQ_\ell)$. By Theorem \ref{t:irreducibility}, $H_c^\bullet(X_h, \overline \QQ_\ell)$ is a direct sum of distinct irreducible representations of $\Unip(\FF_q)$ (parametrized by $\chi \from \TUnip(\FF_q) \to \overline \QQ_\ell$). Write
\begin{equation*}
H_c^i(X_h, \overline \QQ_\ell) = \pi_{i,1} \oplus \cdots \oplus \pi_{i, k_i},
\end{equation*}
where the $\pi_{ij}$ are nonisomorphic irreducible representations of $\Unip(\FF_q)$. Recall that the regular representation $\Reg \colonequals \Ind_{\{1\}}^{\Unip(\FF_q)}(1)$ has the property that an irreducible $\Unip(\FF_q)$-representation $\pi$ has multiplicity $\dim \pi$ in $\Reg$. Then
\begin{equation*}
\dim \Hom_{\Unip(\FF_q)}\left(\Reg, H_c^i(X_h, \overline \QQ_\ell)\right) = \dim \Hom_{\Unip(\FF_q)}\left(\Reg, \bigoplus \pi_{ij} \right) = \sum_{j=1}^{k_i} \dim \pi_{ij}.
\end{equation*}
On the other hand, by Theorem \ref{t:hom},
\begin{align*}
\dim \Hom_{\Unip(\FF_q)}&\left(\Reg, H_c^i(X_h, \overline \QQ_\ell)\right) \\
&= \dim \Hom_{\Unip(\FF_q)}\left(\bigoplus \Ind_{\TUnip(\FF_q)}^{\Unip(\FF_q)}(\chi), H_c^i(X_h, \overline \QQ_\ell)\right) \\
&= \dim\bigoplus_{\substack{\chi \from \TUnip(\FF_q) \to \overline \QQ_\ell^\times \\ r_\chi = i}}\Hom_{\Unip(\FF_q)}\left(\Ind_{\TUnip(\FF_q)}^{\Unip(\FF_q)}(\chi), H_c^i(X_h, \overline \QQ_\ell)\right) \\
&= \sum_{\substack{\chi \from \TUnip(\FF_q) \to \overline \QQ_\ell^\times \\ r_\chi = i}} q^{nd_\chi/2}
\end{align*}
This proves the dimension formula. 
The final assertion now follows from Theorem \ref{t:maximality} since if $n$ is odd and $H_c^i(X_h, \overline \QQ_\ell) \neq 0$, then $i$ must be even, and hence $Z(X_h, t)$ has no nontrivial factors in the numerator.
\end{proof}

\begin{example}
We demonstrate how to calculate the Hasse--Weil zeta function in the case that $n$ is prime. First observe that the Howe decomposition of a character of $\bW_h^{(1)}(\FF_{q^n})$ must be of the form
\begin{equation*}
\chi = \chi_1^0(\Nm_{\FF_{q^n}/\FF_q}) \cdot \chi_2^0,
\end{equation*}
where $\chi_1^0$ is a character of $\bW_h^{(1)}(\FF_q)$ and $\chi_2^0$ is a primitive character of $\bW_{h'}^{(1)}(\FF_{q^n})$. After fixing a level $h'$ with $1 \leq h' \leq h$, the number of such $\chi$ is equal to
\begin{equation}\label{e:chi count}
N_{h'} \colonequals (q^{h-1} - q^{h'-1} + 1) \cdot (q^{p(h'-1)} - q^{p(h'-2)} - q + 1),
\end{equation}
and 
\begin{equation*}
d_\chi = (n-1)(h'-1), \qquad r_\chi = (n-1)(h-1) + (n-1)(h-h').
\end{equation*}
Thus by Theorem \ref{t:zeta},
\begin{equation*}
\dim H_c^i(X_h, \overline \QQ_\ell) =
\begin{cases}
q^{h-1} & \text{if $i = 2(n-1)(h-1)$,} \\
N_{h'} \cdot q^{n(n-1)(h'-1)/2} & \text{if $i = (n-1)(h-1) + (n-1)(h-h')$.}
\end{cases}
\end{equation*}
We can now write down explicit formulas for the zeta function of $X_h$. For example:
\begin{corollary}
If $n = 2$, then
\begin{equation*}
Z(X_3, t) = 
\frac{(1+q^3 \cdot t)^{(q^2-q+1)(q^2-q)q}}{(1-q^2 \cdot t)^{q^2} \cdot (1-q^4 \cdot t)^{(q^4-q^2-q+1)q^2}}.
\end{equation*}
\end{corollary}
\end{example}

\subsection{Examples} \label{s:examples}

Prior to this work, the only cases in which the $\Unip(\FF_q)$-representations $H_c^i(X_h, \overline \QQ_\ell)[\chi]$ had been studied were in the following cases:
\begin{enumerate}[label=(\arabic*)]
\item
For $h = 2$ and $k = 1$, this was done by Boyarchenko--Weinstein in \cite[Theorem 4.5.1]{BW16}.

\item
For $h,k$ arbitrary, $\chi$ primitive, and $\Char K > 0$, this was done by the author in \cite{C15}. Before this some smaller cases were done:
\begin{enumerate}[label=(\alph*)]
\item
The $n=2$, $h=3$, $k=1$, $\Char K > 0$ case was done by Boyarchenko in \cite[Theorem 5.20]{B12}.
\item
The $n=2$, $h$ arbitrary, $k=1$, $\Char K > 0$ case was done by the author in \cite{C16}.
\end{enumerate}
\end{enumerate}
We explain how to specialize Theorems \ref{t:hom}, \ref{t:maximality}, and \ref{t:irreducibility} to recover these results.

Note that in previous work, the unipotent group schemes are called $U_{h,k}^{n,q}$ and are defined over $\FF_{q^n}$, whereas the unipotent group schemes in this paper are called $\Unip$ and are defined over $\FF_q$ (see Remark \ref{r:unip} for a more detailed discussion). However, $U_{h,k}^{n,q}(\FF_{q^n}) \cong \Unip(\FF_q)$, and there is a natural way to realize $\bW_h^{(1)}(\FF_{q^n})$ as a subgroup of each, so the distinction between working in the ambient $U_{h,k}^{n,q}$ and $\Unip$ only appears in the proofs and not in the theorem statements.

\subsubsection{The case $h = 2, k=1$}

In this setting, the subquotients of the multiplicative group of the division algebra in equal and mixed characteristic are isomorphic, and so one does not run into the mixed characteristic difficulties that arise when $h > 2$. 

Note that $\bT_{2,1}(\FF_q) \cong \bW_2^{(1)}(\FF_{q^n}) \cong \FF_{q^n}$. Let $\chi \from \FF_{q^n} \to \overline \QQ_\ell^\times$ be a character. If $\chi$ is trivial, then it corresponds to $(\{1,1,n\}, \{2,1,1\})$, and if $\chi$ is nontrivial of conductor $m$, then it corresponds to $(\{1,m,n\}, \{2,2,1\})$. Then by Theorem \ref{t:hom}, $r_\chi = (n-\frac{n}{m}) + 2(\frac{n}{m}-1) = n + \frac{n}{m} - 2$ and
\begin{equation*}
\Hom_{\bU_{2,1}(\FF_q)}\left(\Ind_{\FF_{q^n}}^{\bU_{2,1}(\FF_q)}(\chi), H_c^i(X_2, \overline \QQ_\ell)\right) \neq 0 \quad \Longleftrightarrow \quad i = n + n/m - 2.
\end{equation*}
The center of $\bU_{2,q}(\FF_q)$ is $\bT_{2,1}(\FF_q) \cong \FF_{q^n}$. Since the actions of $\bT_{2,1}(\FF_q)$ and $\bU_{2,1}(\FF_q)$ on $X_2$ agree on the center of $\bU_{2,1}(\FF_q)$, the above equation implies that
\begin{equation*}
H_c^i(X_2, \overline \QQ_\ell)[\chi] \neq 0 \quad \Longleftrightarrow \quad i = n + n/m - 2.
\end{equation*}
The centrality of $\bT_{2,1}(\FF_q)$ in $\bU_{2,1}(\FF_q)$ (which is not true for $h > 2$) allowed us to obtain Theorem \ref{t:s chi} from Theorem \ref{t:hom} automatically. We now see that maximality of $X_2$ holds by Theorem \ref{t:maximality} (this is \cite[Theorem 4.5.1(b)]{BW16}), and the irreducibility and multiplicity-one properties of $H_c^i(X_h, \overline \QQ_\ell)[\chi]$ follow from Theorem \ref{t:irreducibility} (this is \cite[Theorem 4.5.1(a)]{BW16}).

\subsubsection{The case $h,k$ arbitrary, $\chi$ primitive, $\Char K > 0$}

Let $\chi \from \TUnip(\FF_q) \to \overline \QQ_\ell^\times$ be primitive. Then the sequences $\{m_i\}, \{h_i\}$ associated to the Howe factorization of $\chi$ are $(\{1, n, n\},\{h, h, 1\})$. By Theorem \ref{t:hom}, $r_\chi = \left(n - \frac{n}{n}\right)(h-1) + 2 \left(\frac{n}{n} - \frac{n}{n}\right)(h-1) = (n-1)(h-1)$ and
\begin{equation*}
\Hom_{\Unip(\FF_q)}\left(\Ind_{\TUnip(\FF_q)}^{\Unip(\FF_q)}(\chi), H_c^i(X_h, \overline \QQ_\ell)[\chi]\right) \neq 0 \quad \Longleftrightarrow \quad i = (n-1)(h-1).
\end{equation*}
The subgroup $\TUnip(\FF_q)$ is not central in $\Unip(\FF_q)$, but the center of $\Unip(\FF_q)$ contains $H_0(\FF_q)$, where $H_0$ is the subgroup of $\Unip$ consisting of diagonal matrices with entries in $\{(1, 0, \ldots, 0, *)\} \subset \bW_h^{(1)}$. Moreover, $\chi$ is primitive if and only if its restriction to $H_0(\FF_q)$ is primitive, and it therefore follows that $r_\chi$ only depends on $\chi|_{H_0(\FF_q)}$. Hence
\begin{equation*}
H_c^i(X_h, \overline \QQ_\ell)[\chi] \neq 0 \quad \Longleftrightarrow \quad i = (n-1)(h-1).
\end{equation*} 
Note that again, it was the fact that the restriction of $\chi$ to the center of $\Unip(\FF_q)$ determines $r_\chi$ (which is not true for non-primitive $\chi$), which allowed us to immediately pinpoint the nonvanishing cohomological degree of $H_c^i(X_h, \overline \QQ_\ell)[\chi]$ from Theorem \ref{t:hom}. The irreducibility and multiplicity-one properties of $H_c^i(X_h, \overline \QQ_\ell)[\chi]$ follow from Theorem \ref{t:irreducibility}. In this generality, this is \cite[Theorem 6.3]{C15}, in the case (2a), this is \cite[Theorem 5.20]{B12}, and in the case (2b), irreducibility is \cite[Theorem 5.1]{C16} and multiplicity-one follows from the trace formula given in \cite[Theorem 5.2]{C16}.

\subsection{Alternating sums of eigenspaces: proof of Theorem \ref{t:R chi}} \label{s:R chi}

It is enough to show that the $(\TUnip(\FF_q) \times \TUnip(\FF_q))$-eigenspaces of the cohomology of $\Sigma = (X_h \times X_h)/\Unip(\FF_q)$ satisfy
\begin{equation*}
\textstyle\sum (-1)^i \dim H_c^i(\Sigma, \overline \QQ_\ell)_{\chi, \chi'{}^{-1}} = \begin{cases}
1, & \text{if $\chi = \chi',$} \\
0, & \text{otherwise.}
\end{cases}
\end{equation*}
We follow a technique of Lusztig demonstrated in \cite{L79} wherein we construct an action of an connected torus $\cT$ over $\overline \FF_q$ on $\Sigma$ and then use the fact (see, for example, \cite[Proposition 10.15]{DM91}) that
\begin{equation*}
\textstyle\sum (-1)^i H_c^i(\Sigma, \overline \QQ_\ell) = \sum (-1)^i  H_c^i(\Sigma^\cT, \overline \QQ_\ell).
\end{equation*}

Recall that
\begin{equation*}
X_h = \{g \in \Unip(\overline \FF_q) : F(g) g^{-1} \in \widetilde U_h \cap F(\widetilde U_h^-)\}.
\end{equation*}
It is clear that the map
\begin{align*}
X_h \times X_h &\to \{(x,x',y,y') \in (\widetilde U_h \cap F(\widetilde U_h^-)) \times (\widetilde U_h \cap F(\widetilde U_h^-)) \times \Unip(\overline \FF_q) \times \Unip(\overline \FF_q) : \\
&\qquad\qquad\qquad\qquad\qquad\qquad\qquad\qquad\qquad\qquad\qquad xy = F(y), \; y' = F(y')x'\}, \\
(g,g') &\mapsto (F(g) g^{-1}, F(g') g'{}^{-1}, g, g'{}^{-1})
\end{align*}
is an isomorphism. Since the $(\Unip(\overline \FF_q))^F = \Unip(\FF_q)$, then for any $g \in X_h$ and $h \in \Unip(\FF_q)$, we have $gh \in X_h$, and the image of $(gh, g'h)$ is $(F(g)g^{-1}, F(g')g'{}^{-1}, gh, h^{-1} g'{}^{-1})$. Moreover, the $F(gh)(gh)^{-1} = F(g)g^{-1}$ if and only if $h \in \Unip(\FF_q)$. It follows from this that the map
\begin{align*}
\Sigma &\to \{(x,x',y) \in (\widetilde U_h \cap F(\widetilde U_h^-)) \times (\widetilde U_h \cap F(\widetilde U_h^-)) \times \Unip(\overline \FF_q) : xy = F(y)x'\}, \\
(g,g') &\mapsto (F(g)g^{-1}, F(g')g'{}^{-1}, gg'{}^{-1})
\end{align*}
is a bijection.


We have
\begin{align} \label{e:U^-}
\widetilde U_h^- \cap F^{-1}(\widetilde U_h) 
&= \left\{\left(\begin{smallmatrix} 1 & & & \\ 0 & 1 & & \\ 0 & 0 & 1 & \\ * & * & * & 1 \end{smallmatrix}\right) \in \Unip(\overline \FF_q)\right\}, &
\widetilde U_h^- \cap F^{-1}(\widetilde U_h^-), 
&= \left\{\left(\begin{smallmatrix} 1 & & & \\ * & 1 & & \\ * & * & 1 & \\ 0 & 0 & 0 & 1 \end{smallmatrix}\right) \in \Unip(\overline \FF_q)\right\},
\\ \label{e:U}
\widetilde U_h \cap F^{-1}(\widetilde U_h^-) 
&= \left\{\left(\begin{smallmatrix} 1 & 0 & 0 & * \\  & 1 & 0 & * \\  &  & 1 & * \\  &  &  & 1 \end{smallmatrix}\right) \in \Unip(\overline \FF_q)\right\}, 
&
\widetilde U_h \cap F^{-1}(\widetilde U_h) 
&= \left\{\left(\begin{smallmatrix} 1 & * & * & 0 \\  & 1 & * & 0 \\  &  & 1 & 0 \\  &  &  & 1 \end{smallmatrix}\right) \in \Unip(\overline \FF_q)\right\}.
\end{align}
Hence we see that $y \in \Unip(\overline \FF_q)$ can be written uniquely in the form
\begin{align*}
y &= y_1' y_2' y_1'' y_2'', & y_1' &\in \widetilde U_h^- \cap F^{-1}(\widetilde U_h), \qquad y_2' \in \widetilde D_h(\widetilde U_h^- \cap F^{-1}(\widetilde U_h^-)), \\
& & y_1'' &\in \widetilde U_h \cap F^{-1} \widetilde U_h^-, \qquad y_2'' \in \widetilde U_h \cap F^{-1} \widetilde U_h,
\end{align*}
where $\widetilde D_h \subset \Unip(\overline \FF_q)$ is the subgroup of diagonal matrices. Hence $\Sigma$ is in bijection with the set of $(x, x', y_1', y_2', y_1'', y_2'')$ satisfying
\begin{equation*}
xy_1'y_2'y_1''y_2'' = F(y_1'y_2'y_1''y_2'') x'.
\end{equation*}

Notice now that any $z \in \widetilde U_h$ can be written uniquely in the form $y_2'' x' F(y_2'')^{-1}$ for some $y_2'' \in \widetilde U_h \cap F^{-1} \widetilde U_h$ and $x' \in \widetilde U_h \cap F \widetilde U_h^-$. Indeed, it is easy to see that for any $y_2''$, we have $(\widetilde U_h \cap F \widetilde U_h^-) \cdot F(y_2'') = \widetilde U_h$ and that for any $y_2''$ and any $z$, there exists a unique $x'$ such that $z = x' F(y_2'')$. In particular, for any $y_2''$ and any $z$, there exists a unique $x'$ such that $y_2'' z = x' F(y_2'')$. Thus, letting $\widetilde x \colonequals x y_1,$ we have that $\Sigma$ is in bijection with tuples $(\widetilde x, y_1', y_2', y_1'')$ satisfying
\begin{equation*}
\widetilde x y_2' \in F(y_1' y_2' y_1'') \cdot \widetilde U_h.
\end{equation*}

Consider the subgroup of $\Unip(\overline \FF_q) \times \Unip(\overline \FF_q)$ given by
\begin{equation*}
H \colonequals \{(t,t') : \text{$t,t'$ diag; $t \in \widetilde D_h t'$; $F(t)^{-1} t = F(t')^{-1} t'$ centralizes $\widetilde D_h(\widetilde U_h^- \cap F^{-1} \widetilde U_h^-)$}\}.
\end{equation*}
For $(t,t') \in H$ consider the map
\begin{equation*}
a_{(t,t')} \from (\widetilde x, y_1', y_2', y_1'') \mapsto (F(t) \widetilde x F(t)^{-1}, t y_1' t^{-1}, t y_2' t'{}^{-1}, t' y_1'' t'{}^{-1}).
\end{equation*}
We first show that $a_{(t,t')}$ defines a map $\Sigma \to \Sigma$: if $\widetilde x y_2' \in F(y_1' y_2' y_1'') \cdot \widetilde U$, then $F(t) \widetilde x y_2' \in F(t) F(y_1' y_2' y_1'') \cdot F(t')^{-1} \widetilde U_h F(t')$, which implies 
\begin{equation*}
F(t) \widetilde x F(t)^{-1} t y_2' t'{}^{-1} \in F(t) F(y_1') F(y_2') F(y_1'') F(t')^{-1} \widetilde U_h,
\end{equation*}
where we used that $F(t)^{-1} t$ centralizes $\widetilde D_h(\widetilde U_h^- \cap F^{-1} \widetilde U_h)$ and that $F(t') = F(t)^{-1} tt'{}^{-1}$. It is now clear that $a_{(t,t')}$ for $(t,t') \in H$ defines an action on $\Sigma$.

Observe that $H$ is abelian and contains $\TUnip(\FF_q) \times \TUnip(\FF_q)$ as a subgroup. Moreover, the action of $H$ restricts to the left action of $\TUnip(\FF_q) \times \TUnip(\FF_q)$ inherited from the left-multiplication action of $\TUnip(\FF_q)$ on $X_h$. We now pinpoint an algebraic torus $\cT \subset H$. By Equation \eqref{e:U^-}, we see that $a \colonequals \left(\begin{smallmatrix} 1 & & & \\ & \ddots & & \\ & & 1 & \\ & & & a_0 \end{smallmatrix}\right) \in \Unip(\overline \FF_q)$ centralizes $\widetilde D_h(\widetilde U_h^- \cap F^{-1} \widetilde U_h^-)$. If $t = \diag(t_1, \ldots, t_n)$ has the property that $t = F(t) a$, then
\begin{equation*}
\left(\begin{matrix}
t_1 & & & \\ & t_2 & & \\ & & \ddots & \\ & & & t_n
\end{matrix}\right) = \left(\begin{matrix}
\varphi(t_{\sigma(1)}) & & & \\ & \varphi(t_{\sigma(2)}) & & \\ & & \ddots & \\ & & & \varphi(t_{\sigma(n)}) \cdot a_0 \end{matrix}\right).
\end{equation*}
In particular, if we set $t_n = \xi \in \overline \FF_q^\times$, then $t(\xi) \colonequals \diag\left(\varphi^{\gamma(1)}(\xi), \varphi^{\gamma(2)}(\xi), \cdots, \varphi^{\gamma(n-1)}(\xi), \xi\right)$ has the property that $F(t)^{-1}t$ centralizes $\widetilde D_h(\widetilde U_h^- \cap F^{-1} \widetilde U_h^-)$. Here, $\gamma \in S_n$ is the unique permutation determined by $\sigma^{-\gamma(i)} = i$. Now 
\begin{equation*}
\cT \colonequals \{(t(\xi), t(\xi)) : \xi \in \overline \FF_q^\times\} \subset H
\end{equation*}
is a connected algebraic torus over $\overline \FF_q$ whose action on $\Sigma$ commutes with the $(\TUnip(\FF_q) \times \TUnip(\FF_q))$-action. One can easily see that the $\cT$-fixed point set of $\Sigma$ is
\begin{equation*}
\Sigma^\cT = \{(1,1,y_2', 1) : y_2' \in \widetilde D_h, \; y_2' \in F(y_2') \cdot \widetilde U_h\} = \{(1,1,y_2', 1) : y_2' \in \widetilde D_h^F\} \cong \TUnip(\FF_q),
\end{equation*}
where under the final identification, $\TUnip(\FF_q) \times \TUnip(\FF_q)$ acts by $(t,t') * y_2' = t y_2' t'{}^{-1}$. Thus
\begin{align*}
\textstyle\sum (-1)^i H_c^i(\Sigma^\cT, \overline \QQ_\ell)_{\chi, \chi'{}^{-1}} 
&= \textstyle\sum (-1)^i H_c^i(\TUnip(\FF_q), \overline \QQ_\ell)_{\chi, \chi'{}^{-1}} \\
&= H_c^0(\TUnip(\FF_q), \overline \QQ_\ell)_{\chi, \chi'{}^{-1}} \\
&= \begin{cases}
1 & \text{if $\chi = \chi'$,} \\
0 & \text{otherwise.}
\end{cases}
\end{align*}

\subsection{The nonvanishing cohomological degree: proof of Theorem \ref{t:s chi}}\label{s:s chi}

By Frobenius reciprocity, it is enough to show
\begin{equation}\label{e:goal}
\dim \Hom_{\bT_{h,k}^{n,q}(\FF_q)}(\chi, H_c^{s_\chi^{n,q}}(X_h^{n,q}, \overline \QQ_\ell)[\chi]) \neq 0,
\end{equation}
where we write $\bT_{h,k} = \TUnip^{n,q}$, $s_\chi = s_\chi^{n,q}$, and $X_h = X_h^{n,q}$ to emphasize the dependence on $n,q$. It is clear that once this is established, then by Theorem \ref{t:hom}, it follows immediately that $s_\chi = r_\chi$. For notational simplicity, we write $H_c^i(X)$ to mean $H_c^i(X, \overline \QQ_\ell)$. We first prove two lemmas.

\begin{lemma}\label{l:tr reduction}
Let $p_0$ be a prime dividing $n$. For any $x \in \FF_{q^{p_0}}^\times \smallsetminus \FF_q^\times$ and any $g \in \TUnip(\FF_q)$,
\begin{equation*}
(-1)^{s_\chi^{n,q}} \Tr\left((x,1,g) ; H_c^{s_\chi^{n,q}}(X_h^{n,q})[\chi]\right) = (-1)^{s_\chi^{n/p_0, q^{p_0}}} \Tr\left((1,1,g) ; H_c^{s_\chi^{n,q}}(X_h^{n/p_0, q^{p_0}})[\chi]\right).
\end{equation*}
\end{lemma}

\begin{proof}
Fix $x \in \FF_{q^{p_0}}^\times \smallsetminus \FF_q^\times$. Recall that the action of $x$ on $M = M_1 + M_2 \varpi^k + \cdots + M_n \varpi^{[k(n-1)]} \in X_h^{n,q}$ (recall the standard form of a point of $X_h^{n,q}$ in Definition \ref{d:M char}) is given by
\begin{equation*}
x * M = \diag(x,x^{q^{\tau(2)}}, \ldots, x^{q^{\tau(n)}})^{-1} \cdot M \cdot \diag(x,x^{q^{\tau(2)}}, \ldots, x^{q^{\tau(n)}}).
\end{equation*}
Note that
\begin{equation*}
x^{-1} \cdot \varpi^k \cdot x= \diag(x^{q^{\tau(2)-\tau(1)}}, \ldots, x^{q^{\tau(n) - \tau(n-1)}}, x^{q^{\tau(1) - \tau(n)}}) \cdot \varpi^k
\end{equation*}
and that 
\begin{equation*}
x^{q^{\tau(i+j) - \tau(i)}} = 1 \qquad \Longleftrightarrow \qquad p_0 \mid j.
\end{equation*}
Therefore, if $x * M = M$, then necessarily $M_j = 0$ for $j \not\equiv 1$ modulo $p_0$ and
\begin{equation*}
M = M_1 + M_{p_0+1} \varpi^{[kp_0]} + \cdots + M_{n - p_0 + 1} \varpi^{[k(n-p_0)]}.
\end{equation*}
For any integer $m$, let $[m]'$ be the unique integer $1 \leq [m]' \leq \frac{n}{p_0}$ such that $m \equiv [m]'$ modulo $\frac{n}{p_0}$. I now claim that there is a $\bT_{h,k}^{n,q}(\FF_q)$-equivariant morphism
\begin{align*}
f \from (X_h^{n,q})^x &\to X_h^{n/p_0,q^{p_0}}, \\
M_1 + M_{p_0+1} \varpi^{k[p_0]} + \cdots + M_{n - p_0 + 1} \varpi^{[k(n-p_0)]} &\mapsto M_1' + M_{p_0+1}' \varpi_{n/p_0}^{[k]'} + \cdots + M_{n - p_0 + 1}' \varpi_{n/p_0}^{[k(\frac{n}{p_0}-1)]'},
\end{align*}
where $M_j'$ denotes the top-left-justified $\frac{n}{p_0} \times \frac{n}{p_0}$ matrix in $M_j$ and $\varpi_{n/p_0} = \left(\begin{smallmatrix} 0 & 1_{\frac{n}{p_0}-1} \\ \pi & 0 \end{smallmatrix}\right).$ Using Definition \ref{d:M char}, it is a straightforward check to see that this morphism is well-defined since $f(M)$ is of the form \eqref{e:std form}, satisfies \eqref{e:entries}, and the determinant condition $\varphi(\det(M)) = \det(M)$ implies that $\varphi^{p_0}(\det(f(M))) = \det(f(M))$. (This last claim can be seen by observing that the rows and columns of $M$ can be swapped so that the matrix becomes block diagonal of the form $\diag(f(M), \varphi^{\tau(2)}(f(M)), \ldots, \varphi^{\tau(p_0-1)}(f(M)))$.) The equivariance under $\bT_{h,k}^{n,q}(\FF_q) \cong \bT_{h,k}^{n/p_0, q^{p_0}}(\FF_q)$ is clear.

By Proposition \ref{p:dim Xh}, $X_h^{n,q}$ is a separated, finite-type scheme over $\FF_{q^n}$, and the action of $(x,t,g) \in \FF_{q^n}^\times \times \TUnip(\FF_q) \times \TUnip(\FF_q)$ defines a finite-order automorphism. Moreover, $(x,t,g) = (1,t,g) \cdot (x,1,1)$, where $(1,t,g)$ is a $p$-power-order automorphism and $(x,1,1)$ has prime-to-$p$ order. Hence by the Deligne--Lusztig fixed point formula \cite[Theorem 3.2]{DL76}, we have
\begin{equation*}
\sum_i (-1)^i \Tr\left((x,t,g)^* ; H_c^i(X_h^{n,q})\right) = \sum_i (-1)^i \Tr\left((1,t,g)^* ; H_c^i((X_h^{n,q})^x)\right).
\end{equation*}
Therefore
\begin{align*}
\#\bT_{h,k}^{n,q}(\FF_q) &\sum_i (-1)^i \Tr\left((x,1,g)^* ; H_c^i(X_h^{n,q})[\chi]\right) \\
&= \sum_{t \in \bT_{h,k}^{n,q}(\FF_q)} \chi(t)^{-1} \sum_i (-1)^i \Tr\left((x,t,g)^* ; H_c^i(X_h^{n,q})\right) \\
&= \sum_{t \in \bT_{h,k}^{n,q}(\FF_q)} \chi(t)^{-1} \sum_i (-1)^i \Tr\left((1,t,g)^* ; H_c^i((X_h^{n,q})^x)\right) \\
&= \# \bT_{h,k}^{n,q}(\FF_q) \sum_i (-1)^i \Tr\left((1,1,g)^* ; H_c^i((X_h^{n,q})^x)[\chi]\right) \\
&= \# \bT_{h,k}^{n,q}(\FF_q) \sum_i (-1)^i \Tr\left((1,1,g)^* ; H_c^i(X_h^{n/p_0,q^{p_0}})[\chi]\right).
\end{align*}
The desired equality now follows since by Theorem \ref{t:irreducibility}, $H_c^i(X_h^{n,q})[\chi]$ and $H_c^i(X_h^{n/p_0,q^{p_0}})[\chi]$ are nonzero only when $i = s_\chi^{n,q}$ and $i = s_\chi^{n/p_0,q^{p_0}}$, respectively.
\end{proof}

\begin{lemma}\label{l:positive}
Let $\chi \from \bT_{h,k}^{n,q}(\FF_q) \to \overline \QQ_\ell^\times$. Assume that we are in one of the following cases:
\begin{enumerate}[label=(\arabic*)]
\item
$n > 1$ is odd and $p_0$ is a prime divisor of $n$.

\item
$n > 1$ is even and $p_0 = 2$.
\end{enumerate}
Fix a $\zeta \in \FF_{q^{p_0}}$ such that $\langle \zeta \rangle = \FF_{q^{p_0}}^\times$ and consider the extension of $\chi$ defined by
\begin{equation*}
\widetilde \chi \from \FF_{q^{p_0}}^\times \times \bT_{h,k}^{n,q}(\FF_q) \to \overline \QQ_\ell^\times, \qquad (\zeta,g) \mapsto \begin{cases}
\chi(g) & \text{if $q$ is even,} \\
(-1)^{s_\chi^{n,q} + s_\chi^{n/p_0, q^{p_0}}} \cdot \chi(g) & \text{if $q$ is odd.}
\end{cases}
\end{equation*}
Then
\begin{equation*}
\sum_{x \in \FF_{q^{p_0}}^\times \smallsetminus \FF_q^\times} \widetilde \chi(x,1)^{-1} \cdot (-1)^{s_\chi^{n,q} + s_\chi^{n/p_0, q^{p_0}}} \neq 0.
\end{equation*}
\end{lemma}

\begin{proof}
Assume that $q$ is even. Then 
\begin{equation*}
\sum_{x \in \FF_{q^{p_0}}^\times \smallsetminus \FF_q^\times} \widetilde \chi(x,1)^{-1} \cdot (-1)^{s_\chi^{n,q} + s_\chi^{n/p_0,q^{p_0}}} = \sum_{x \in \FF_{q^{p_0}}^\times \smallsetminus \FF_q^\times} (-1)^{s_\chi^{n,q} + s_\chi^{n/p_0,q^{p_0}}} \neq 0.
\end{equation*}
For the remainder of the proof, assume that $q$ is odd. If we are in Case (1), then by Theorem \ref{t:maximality}, we know that $s_\chi^{n,q}$ and $s_\chi^{n/p_0, q^{p_0}}$ are both even. It therefore follows that 
\begin{equation*}
\sum_{x \in \FF_{q^{p_0}}^\times \smallsetminus \FF_q^\times} \widetilde \chi(x,1)^{-1} \cdot (-1)^{s_\chi^{n,q} + s_\chi^{n/p_0, q^{p_0}}} = \sum_{x \in \FF_{q^{p_0}}^\times \smallsetminus \FF_q^\times} \widetilde \chi(x,1)^{-1} \cdot \widetilde \chi(x,1) > 0.
\end{equation*}
The same conclusion holds if we are in Case (2) and $s_\chi^{n,q} + s_\chi^{n/p_0,q^{p_0}}$ is even, so assume $s_\chi^{n,q} + s_\chi^{n/p_0,q^{p_0}}$ is odd. If $\zeta^m \in \FF_q$, then $m$ must be a multiple of $q+1$, which is even. This implies that
\begin{equation*}
\#\underbrace{(\zeta^{2\bZ + 1} \cap \FF_{q^{p_0}}^\times \smallsetminus \FF_q^\times)}_{A^{\text{odd}}} > \#\underbrace{(\zeta^{2\bZ} \cap \FF_{q^{p_0}}^\times \smallsetminus \FF_q^\times)}_{A^{\text{even}}}
\end{equation*}
and hence
\begin{align*}
\sum_{x \in \FF_{q^{p_0}}^\times \smallsetminus \FF_q^\times} &\widetilde \chi(x,1)^{-1} \cdot (-1)^{s_\chi^{n,q} + s_\chi^{n/p_0, q^{p_0}}} \\
&= \sum_{x \in A^{\text{odd}}} \widetilde \chi(x,1)^{-1} \cdot \widetilde \chi(x,1) + \sum_{x \in A^{\text{even}}} \widetilde \chi(x,1)^{-1} \cdot (-\widetilde \chi(x,1)) > 0. \qedhere
\end{align*}
\end{proof}

We are now ready to prove the theorem. First observe that $X_h^{1,q} = \bT_{h,k}^{1,q}(\FF_q)$ and hence for any $\chi \from \bT_{h,k}^{1,q}(\FF_q) \to \overline \QQ_\ell^\times$, we have
\begin{align*}
H_c^{s_\chi^{1,q}}(X_h^{1,q})[\chi] = H_c^0(\bT_{h,k}^{1,q}(\FF_q))[\chi] = \chi,
\end{align*}
so Equation \eqref{e:goal} holds for $n = 1$ and $q$ arbitrary. We now induct on the number of prime divisors of $n$. Assume that Equation \eqref{e:goal} holds for any $n = \prod_{i=1}^l p_i$ and arbitrary $q$, where the $p_i$ are (possibly non-distinct) primes. We will show that Equation \eqref{e:goal} holds for any $n = \prod_{i=0}^l p_i$ and arbitrary $q$.

If $n$ is odd, let $p_0$ be any prime divisor of $n$, and if $n$ is even, let $p_0 = 2$. Define the character $\widetilde \chi \from \FF_{q^{p_0}}^\times \times \bT_{h,k}^{n,q}(\FF_q) \to \overline \QQ_\ell^\times$ as in Lemma \ref{l:positive}. Then
\begin{align} \nonumber
&\sum_{(x, g) \in \FF_{q^{p_0}}^\times \times \bT_{h,k}^{n,q}(\FF_q)} \widetilde \chi(x, g)^{-1} \Tr\left((x,1,g); H_c^{s_\chi^{n,q}}(X_h^{n,q})[\chi]\right) \\ \nonumber
&\qquad= \sum_{\substack{(x,g) \\ x \in \FF_q^\times}} \widetilde \chi(x, g)^{-1} \Tr\left((x,1,g); H_c^{s_\chi^{n,q}}(X_h^{n,q})[\chi]\right) \\ \nonumber
&\qquad \qquad  + \sum_{\substack{(x,g) \\ x \in \FF_{q^{p_0}}^\times \smallsetminus \FF_q^\times}} \widetilde \chi(x, g)^{-1} \Tr\left((x,1,g); H_c^{s_\chi^{n,q}}(X_h^{n,q})[\chi]\right) \\ \nonumber
&\qquad= \#(\FF_q^\times \times \bT_{h,k}^{n,q}(\FF_q)) \cdot \dim \Hom_{\FF_q^\times \times \bT_{h,k}^{n,q}(\FF_q)}\left(\widetilde \chi, H_c^{s_\chi^{n,q}}(X_h^{n,q})[\chi]\right) \\ \label{e:sum}
&\qquad \qquad  + \sum_{\substack{(x,g) \\ x \in \FF_{q^{p_0}}^\times \smallsetminus \FF_q^\times}} \widetilde \chi(x, g)^{-1} \cdot (-1)^{s_\chi^{n,q} + s_\chi^{n/p_0,q^{p_0}}} \cdot \Tr\left((1,1,g); H_c^{s_\chi^{n/p_0,q^{p_0}}}(X_h^{n/p_0,q^{p_0}})[\chi]\right),
\end{align}
where the last equality holds by Lemma \ref{l:tr reduction}. By the inductive hypothesis,
\begin{equation*}
\sum_{\substack{(x,g) \\ x \in \FF_{q^{p_0}}^\times \smallsetminus \FF_q^\times}} \widetilde \chi(x, g)^{-1} \cdot (-1)^{s_\chi^{n,q} + s_\chi^{n/p_0,q^{p_0}}} \cdot \Tr\left((1,1,g); H_c^{s_\chi^{n/p_0,q^{p_0}}}(X_h^{n/p_0,q^{p_0}})[\chi]\right) > 0,
\end{equation*}
and by Lemma \ref{l:positive}, 
\begin{align*}
\sum_{x \in \FF_{q^{p_0}}^\times \smallsetminus \FF_q^\times} \widetilde \chi(x,1)^{-1} \cdot (-1)^{s_\chi^{n,q} + s_\chi^{n/p_0, q^{p_0}}} \neq 0.
\end{align*}
It therefore follows that the expression on line \eqref{e:sum} is the product of two nonzero integers, so it is positive or negative. If \eqref{e:sum} is positive, then we have shown
\begin{align*}
0 
&< \dim \Hom_{\FF_{q^{p_0}}^\times \times \bT_{h,k}^{n,q}(\FF_q)}\left(\widetilde \chi, H_c^{s_\chi^{n,q}}(X_h^{n,q})[\chi]\right) \\
&\leq \dim \Hom_{\FF_{q^{p_0}}^\times \times \bT_{h,k}^{n,q}(\FF_q)}\left(\Ind_{\bT_{h,k}^{n,q}(\FF_q)}^{\FF_{q^{p_0}}^\times \times \bT_{h,k}^{n,q}(\FF_q)}(\chi), H_c^{s_\chi^{n,q}}(X_h^{n,q})[\chi]\right) \\
&= \dim \Hom_{\bT_{h,k}^{n,q}(\FF_q)}\left(\chi, H_c^{s_\chi^{n,q}}(X_h^{n,q})[\chi]\right).
\end{align*}
If \eqref{e:sum} is negative, then we have shown
\begin{align*}
0
&<
\dim \Hom_{\FF_q^\times \times \bT_{h,k}^{n,q}(\FF_q)}\left(\widetilde \chi, H_c^{s_\chi^{n,q}}(X_h^{n,q})[\chi]\right) \\
&\leq \dim \Hom_{\FF_q^\times \times \bT_{h,k}^{n,q}(\FF_q)}\left(\Ind_{\bT_{h,k}^{n,q}(\FF_q)}^{\FF_q^\times \times \bT_{h,k}^{n,q}(\FF_q)}(\chi), H_c^{s_\chi^{n,q}}(X_h^{n,q})[\chi]\right) \\
&= \dim \Hom_{\bT_{h,k}^{n,q}(\FF_q)}\left(\chi, H_c^{s_\chi^{n,q}}(X_h^{n,q})[\chi]\right).
\end{align*}

\section{Torus eigenspaces of the homology of semi-infinite Deligne--Lusztig varieties}\label{s:pDL}

We return our discussion to the semi-infinite Deligne--Lusztig variety $\widetilde X$ associated to $L^\times \hookrightarrow D^\times.$ Our goal is to understand the representations of $D^\times$ arising from the $\theta$-eigenspaces $H_i(\widetilde X, \overline \QQ_\ell)[\theta]$, where $\theta \from L^\times \to \overline \QQ_\ell^\times.$ Once we understand the relationship between the (ind-pro-)scheme structure we defined on $\widetilde X$ in Section \ref{s:def DL} and the action of $L^\times \times D^\times$, the problem of computing the representations $H_i(\widetilde X, \overline \QQ_\ell)$ can be reduced to the problem of computing the representations $H_c^i(X_h, \overline \QQ_\ell)$. This reduction was already known by Boyarchenko \cite{B12}, and therefore, after we summarize \cite[Lemma 6.11, Corollary 6.12]{B12}, Theorems \ref{t:torus eigen}, \ref{t:irred}, and \ref{t:very reg} follow from the theorems in Section \ref{s:DL}.

We first define some terminology. If $\theta \from L^\times \to \overline \QQ_\ell^\times$ is a smooth character, then there exists an $h$ such that the restriction $\theta|_{U_L^h}$ is trivial. We call the smallest such $h$ the \textit{level} of $\theta$. We say that $x \in L^\times$ is \textit{very regular} if $x \in \cO_L^\times$ and its image in the residue field $\F^\times$ has trivial $\Gal(\F/\FF_q)$-stabilizer. Later in this section, we will give a character formula for $H_i(\widetilde X, \overline \QQ_\ell)[\theta]$ on the subgroup of $D^\times$ consisting of very regular elements of $L^\times$. It is known that in many cases, such formulae are enough to pinpoint irreducible representations.

We give a quick outline of the proof of \cite[Proposition 5.19]{B12}.

The action of $L^\times \times D^\times$ on $\widetilde X$ induces an $(L^\times/U_L^h) \times (D^\times/U_D^{n(h-1)+1})$-action on $\widetilde X_h \colonequals \sqcup_{m \in \bZ} \widetilde X_h{}^{(m)}.$ Recall that $X_h$ is a subvariety of $\widetilde X_h'{}^{(0)} \cong \widetilde X_h^{(0)} \hookrightarrow \widetilde X_h$ whose stabilizer is 
\begin{equation*}
\widetilde \Gamma_h \colonequals \langle (\pi, \pi^{-1}) \rangle \cdot \langle (\zeta, \zeta^{-1}) \rangle \cdot (U_L^1/U_L^h \times U_D^1/U_D^{n(h-1)+1}) \subset L^\times/U_L^h \times D^\times/U_D^{n(h-1)+1},
\end{equation*}
and $\widetilde X_h$ is equal to the union of $(L^\times/U_L^h \times D^\times/U_D^{n(h-1)+1})$-translates of (the image of) $X_h$. It therefore follows that there is a natural isomorphism
\begin{equation*}
H_i(\widetilde X_h, \overline \QQ_\ell) \cong \Ind_{\widetilde\Gamma_h}^{(L^\times/U_L^h) \times (D^\times/U_D^{n(h-1)+1})}\left(H_i(X_h, \overline \QQ_\ell)\right).
\end{equation*}
By Proposition \ref{p:dim Xh}, $H_i(X_h, \overline \QQ_\ell) \cong H_c^{2(n-1)(h-1)-i}(X_h, \overline \QQ_\ell) \otimes \left(q^{n(n-1)(h-1)}\right)^{\deg}$. Theorem \ref{t:torus eigen} now follows.

\begin{theorem}\label{t:torus eigen}
Let $\theta \from L^\times \to \overline \QQ_\ell^\times$ be a smooth character whose restriction to $U_L^h$ is trivial and set $\chi \colonequals \theta|_{U_L^1}$. Then for $r_\theta \colonequals 2(n-1)(h-1) - r_\chi$,
\begin{equation*}
H_i(\widetilde X, \overline \QQ_\ell)[\theta] \neq 0 \quad \Longleftrightarrow \quad i = r_\theta.
\end{equation*}
Moreover, 
\begin{equation*}
\eta_\theta \colonequals H_{r_\theta}(\widetilde X, \overline \QQ_\ell)[\theta] \cong \Ind_{\bZ \cdot \cO_D^\times}^{D^\times}\left(\eta_\theta'\right),
\end{equation*}
where $\eta_\theta'$ is a representation naturally obtained from $\theta$ and $H_c^{r_\chi}(X_h, \overline \QQ_\ell)[\chi]$. Explicitly:
\begin{enumerate}[label=(\roman*)]
\item
$H_c^{r_\chi}(X_h, \overline \QQ_\ell)[\chi]$ extends to a representation $\eta_\theta^\circ$ of $\bG_{h,k}(\FF_q) \cong \FF_{q^n}^\times \ltimes \Unip(\FF_q) \cong \cO_D^\times/U_D^{n(h-1)+1}$ with $\Tr(\eta_\theta^\circ(\zeta)) = (-1)^{r_\theta} \theta(\zeta)$, where $\zeta \in \cO_L^\times$ is very regular. We may view $\eta_\theta^\circ$ as a representation of $\cO_D^\times$.

\item
We can extend $\eta_\theta^\circ$ to a representation $\eta_\theta'$ of $\pi^\bZ \cdot \cO_D^\times$ by demanding $\pi \mapsto \theta(\pi)$.
\end{enumerate}
\end{theorem}

\begin{theorem}\label{t:irred}
If $\theta \from L^\times \to \overline \QQ_\ell^\times$ has trivial $\Gal(L/K)$-stabilizer, then the $D^\times$-representation $H_{r_\theta}(\widetilde X, \overline \QQ_\ell)[\theta]$ is irreducible.
\end{theorem}

\begin{proof}
To prove this, we need to show that the normalizer of the $(\bZ \cdot \cO_D^\times)$-representation $\eta_\theta'$ in $D^\times$ is exactly $\bZ \cdot \cO_D^\times$. To see this, it is sufficient to show that $\eta_\theta'$ is not invariant under the conjugation action of $\Pi$. Let $x \in \cO_L^\times \subset \cO_D^\times$ be very regular. Recall that by Theorem \ref{t:character} and the definition of $\eta_\theta'$ given in Theorem \ref{t:torus eigen},
\begin{equation*}
\Tr \eta_\theta'(x) = (-1)^{r_\theta} \cdot \theta(x).
\end{equation*}
Since $\Pi \cdot x \cdot \Pi^{-1} = \varphi(x)$, conjugation by $\Pi$ normalizes the set of very regular elements and
\begin{equation*}
\Tr \eta_\theta'(\Pi \cdot x \cdot \Pi^{-1}) = (-1)^{r_\chi} \cdot \theta(\varphi(x)).
\end{equation*}
Therefore, if $\theta$ has trivial $\Gal(L/K)$-stabilizer, then $\bZ \cdot \cO_D^\times$ is the normalizer of $\eta_\theta'$ in $D^\times$.
\end{proof}

\begin{theorem}\label{t:very reg}
Let $x \in \cO_L^\times$ be very regular. Then 
\begin{equation*}
\Tr \eta_\theta(x) = (-1)^{r_\theta} \cdot \sum_{\gamma \in \Gal(L/K)} \theta^\gamma(x).
\end{equation*}
\end{theorem}

\begin{proof}
We have
\begin{equation*}
\Tr \eta_\theta(x) = \sum_{\substack{g \in D^\times/(\bZ \cdot \cO_D^\times) \\ gxg^{-1} \in \bZ \cdot \cO_D^\times}} \Tr \eta_\theta'(gxg^{-1}) = \sum_{\gamma \in \Gal(L/K)} (-1)^{r_\theta} \cdot \theta^\gamma(x),
\end{equation*}
where the second equality holds by Lemma 5.1(b) of \cite{BW13} together with Theorem \ref{t:character}.
\end{proof}

\end{document}